\documentclass[10pt]{article}

\usepackage{pgfplots}
\usepackage{epsfig}

\usepackage{fancybox}
\usepackage{amsthm}

\usepackage{todonotes}
\newcounter{todocounter}
\usepackage{standalone} 
\usepackage{tikz}
\usepackage{fullpage}
\usepackage{amsmath}
\usepackage{appendix}
\usepackage{pifont}
\usepackage{dsfont}
\usepackage{subfigure}
\usepackage{bbding,pifont,utfsym,dingbat}
\newcommand{\Mod}[1]{\ (\mathrm{mod}\ #1)}

\usepackage{lastpage}

\usepackage{enumerate}
\usepackage{authblk}
\usepackage{amsfonts}
\usepackage{amssymb}
\usepackage{amsthm}
\usepackage{graphicx}
\usepackage[utf8]{inputenc}
\usepackage{ulem}
\usepackage{multirow}
\usepackage{aliascnt}
\usepackage[colorlinks=true,linkcolor=blue]{hyperref}
\newcounter{dader}[section]

\theoremstyle{plain}
\newaliascnt{thms}{dader}
\newtheorem{thm}[thms]{Theorem}
\newcounter{lettercounter}

\aliascntresetthe{thms}
\newtheorem*{thm*}{Theorem}
\newaliascnt{prop}{dader}
\newtheorem{prop}[prop]{Proposition}
\aliascntresetthe{prop}
\newtheorem{lemme}[dader]{Lemma}
\newtheorem*{lemme*}{Lemma}
\newaliascnt{cor}{dader}
\newtheorem{cor}[cor]{Corollary}
\aliascntresetthe{cor}
\usepackage{ bbold }

\theoremstyle{definition}
\newaliascnt{de}{dader}
\newtheorem{de}[de]{Definition}
\aliascntresetthe{de}

\newaliascnt{app}{dader}

\aliascntresetthe{app}

\newtheorem{rem}[dader]{Remark}
\newaliascnt{ex}{dader}

\aliascntresetthe{ex}

\def\equationautorefname~#1\null{(#1)\null}
\newcommand{\R}{\mathbb{R}}
\newcommand{\C}{\mathbb{C}}
\newcommand{\N}{\mathbb{N}}

\newcommand{\Z}{\mathbb{Z}}

\newcommand{\s}{\mathbb{S}}

\newcommand{\FF}{\mathcal{F}}

\newcommand{\e}{\mathrm{e}}
\usepackage{relsize}
\usepackage[all]{xy} 
\usepackage{pgf,tikz}
\usetikzlibrary{decorations.markings,arrows,backgrounds,automata,calc,shapes.arrows}

\newcommand{\somm}[2]{s_{q}^{{#1},{#2}}}
\newcommand{\som}[1]{s_{q}^{{#1}}}
\theoremstyle{plain}
\newtheorem{thmx}{Theorem}

\begin{document}
\title{The level of distribution of the sum-of-digits function in arithmetic progressions}
\author{Nathan Toumi}

\date{}
\maketitle

\begin{abstract}
For $q \geq 2$, $n \in \N$, let $s_{q}(n)$ denote the sum of the digits of $n$ written in base $q$. Spiegelhofer (2020) proved that the Thue--Morse sequence has level of distribution $1$, improving on a former result of Fouvry and Mauduit (1996). In this paper we generalize this result to sequences of type $\left\{\exp\left(2\pi i\ell s_q(n)/b\right)\right\}_{n \in \N}$ and provide an explicit exponent in the upper bound.
\end{abstract}

\tableofcontents
\section{Introduction}

For $q \geq 2$, each integer $n$ can be uniquely written as 
\begin{equation*}
    n=\sum_{k=0}^{\infty}n_{k}q^{k},
\end{equation*}
where 
$n_{k}\in \{0,\dots,q-1\}$ for all $k\geq 0$ are the digits which are zero starting from some finite rank. The sum-of-digits function $s_q$ is a well-studied object in number theory, it is defined for all $n \in \N$ by
$$
    s_{q}(n)=s_{q}\left(\sum_{k=0}^{\infty}n_{k}q^{k}\right)
    =\sum_{k=0}^{\infty}n_{k}.
$$
Although quite simple to define, the sum-of-digits function has raised many questions. 
A natural question that could be asked is how this function distributes itself within arithmetic progressions. In an influential paper, Gelfond~\cite{Gel} provided a new method to study the sum-of-digits function via exponential sums and product representations.
He obtained the following result concerning the distribution of the sum-of-digits function in arithmetic progressions. 

\begin{thmx}[Gelfond, 1967/68]
\label{Gelfond}
Let $q,b,m\geq 2$ be integers such that $(b,q-1)=1$. Then, for all $a \in \{0,\dots,b-1\}$ and for all $r \in \{0,\dots,m-1\}$, we have
\begin{equation*}
    |\{n<N\,:\, s_{q}(n) \equiv a \Mod{b},\ n \equiv r \Mod{m}\}|=\dfrac{N}{bm}+O(N^{\lambda}),
\end{equation*}
where 
\begin{equation*}
    \lambda=\dfrac{1}{2\log(q)}\log\dfrac{q\sin(\pi/2b)}{\sin(\pi/2bq)}<1.
\end{equation*}
\end{thmx}
\begin{rem}
Rebuilding on work of Morgenbesser, Shallit and Stoll~\cite{MorgShaStoll} it can be shown that there exist infinitely many $m\geq 1$ such that
$\min \{n\geq 1\,:\, s_{q}(n) \equiv 1 \pmod{b},\ n \equiv 0 \pmod{m}\}\geq c_b m^{1/(b-1)}$, for some $c_b>0$ only depending on $b$ (one may consider $m=q^{\ell(b-1)}+q^{\ell(b-2)}+\cdots + q^\ell+1$, for instance). 
Therefore, the implied constant cannot be independent of $m$ (at least for the Thue-Morse sequence). This dependence of the constant on the modulus $m$ is the object of our Bombieri-Vinogradov-type average result, see Theorem~\ref{Thm0}.
\end{rem}
In addition to his results in~\cite{Gel}, Gelfond proposed several open questions related to the sum-of-digits function. In particular, can we substitute in Theorem~\ref{Gelfond}, $s_q(n)$ by $s_q(p)$ for $p$ primes or by $s_q(P(n))$ where $P$ is a polynomial $P$ such that $P(n) \in \N$ for all $n \in \N$? 
These questions have been addressed in many recent works.

The most emblematic instance of a sequence related to the sum-of-digits function is the (Prouhet--)Thue--Morse sequence defined (for instance) in its multiplicative way by
\begin{equation*}
t(n)=(-1)^{s_{2}(n)}=\e\left(\dfrac{1}{2}s_{2}(n)\right),
\end{equation*}
where $s_{2}(n)$ is the sum of the binary  digits of the integer $n$, and  $\e(t)=e^{2\pi i t}$. The Thue--Morse sequence is a fundamental object in various areas of mathematics, and in particular in number theory and combinatorics on words (see \cite{AS99} for an overview). Mauduit and Rivat~\cite{MauRivGelf} solved Gelfond's problem on the distribution of $s_q(p)$, $p$ primes, in arithmetic progressions. In the proof of their breakthrough result, the authors introduced several new tools, and most notably an ingenious use of the van der Corput inequality paired with a ``carry propogation lemma'', which, simply put, states that in the addition of a large and a small integer expressed in base-$q$, the highly significant digits of the large number are rarely affected by the addition. In respect of Gelfond's problem, we cite the result of Drmota, Mauduit and Rivat~\cite{DMauRiv}, who showed that the Thue--Morse sequence is normal along squares.
\\Spiegelhofer~\cite{Spi} observed that an iterative use of the van der Corput inequality followed by several digits shiftings leads to Gowers norms. He then was able to use the estimate of Gowers norms for the Thue--Morse provided by Konieczny in~\cite{Kon} to conclude. Bounds on these norms have far-reaching applications (see, for instance,~\cite{Spi23+} for an application for the sum of digits of cubes).

\subsection{Results on the level of distribution of the sum of digits}
Fouvry and Mauduit \cite{FM96} started the investigation of the level of distribution of the Thue--Morse sequence and its generalizations along arithmetic progressions. For $q,b$ two integers such that $(b,q-1)=1$, let
\begin{equation*}
     A_q(y,z;a,b,r,m)=\Big|\{y \leq n<z\,:\, s_{q}(n)\equiv a \Mod{b},\;\; n\equiv r \Mod{m}\}\Big|.
\end{equation*}

The paper~\cite{FM96} is devoted to the base 2.
\begin{thmx}[Fouvry/Mauduit, 1996]\label{ThmFM2}
Let $C_b$ be a constant such that for all $\xi \in \{1/b,\dots,(b-1)/b\}$ we have 
\begin{equation*}
    \int_{0}^{1}\prod_{0 \leq n < N}|\cos(\pi(2^{n}t+\xi))|dt=O(C_b^N), \qquad N \rightarrow \infty.
\end{equation*}
Let
\begin{equation*}
    \gamma(\alpha)=1+\dfrac{\log\left(\beta(\alpha)\right)}{\log 2},
\end{equation*}
where
\begin{equation*}
    \beta(\alpha):= \sqrt{\max_{t \in \R}|\cos(\pi\left(t+\alpha\right))\cos(\pi(2t+\alpha))|}, 
\end{equation*}
and set
\begin{equation*}
    \gamma_{b}=\max\{\gamma(\alpha)\;:\; \e(\alpha)^b=1\;,\; \alpha \not \in \mathbb{Z}\}.
\end{equation*}
For $A \in \R$, we set $D=x^{1/2}\log(x)^{-A}$. Then, as $x\to +\infty$,
\begin{equation*}
    \sum_{1 \leq m \leq D}\max_{1 \leq z \leq x}\;\max_{0 \leq r < m}\Big|A_2(0,z;a,b,r,m)-\dfrac{z}{bm}\Big|=O\left(x^{1+\log(C_b)/\log(2)}D+x^{\gamma_{b}}\left(D^{3-2\gamma_b+2\log(C_b)/\log(2)}+\log x\right)\right).
\end{equation*}
\end{thmx}

Theorem~\ref{ThmFM2} allowed them to show that $0.5924$ is a level of distribution of the Thue--Morse sequence $t$. More precisely, they showed in~\cite{FM96} the following theorem:
\begin{thmx}[Fouvry/Mauduit, 1996]\label{ThmFM1}
Let $A \in \R$ and $D=x^{0.5924}$. There exists $C>0$ such that
\begin{equation*}
    \sum_{1 \leq m \leq D}\max_{1 \leq z \leq x}\;\max_{0 \leq r < m}\Big|A_2(0,z;0,2,r,m)-\dfrac{z}{2m}\Big| \leq Cx(\log 2x)^{-A},\qquad x\to \infty.
\end{equation*}
\end{thmx}
They also investigated the level of distribution of the sequence $(s_{q}(n))_{n \in \N}$ for general base of numeration $q$. In their paper~\cite[p.340]{FM962} they proved the following result.

\begin{thmx}[Fouvry/Mauduit, 1996]
Let $q\geq 2$, $a$, $b$ be integers such that $(b,q-1)=1$. Then, for all $x \geq 1$, for all $A \in \R$ and for all $\varepsilon>0$, we have
\begin{equation*}
    \sum_{1 \leq m \leq x^{\theta_{q}-\varepsilon}}\max_{1 \leq z \leq x}\;\max_{0 \leq r <m}\Big|\sum_{\substack{n<y\\s_{q}(n) \equiv a \Mod{b}\\n \equiv r \Mod{m}}}1-\dfrac{1}{m}\sum_{\substack{n<y\\s_{q}(n) \equiv a \Mod{b}}}1\Big|=O(x(\log 2x)^{-A}),\qquad x\to \infty,
\end{equation*}    
where $\theta_{q}$ is defined by
\begin{equation*}
    \theta_{q}=1-\dfrac{\log(M(q))}{\log(q)},
\end{equation*}
where
\[
M(q)=
\begin{cases}
\dfrac{1}{n}\sum\limits_{k=0}^{n-1}\cos\left(\frac{2k+1}{4n}\pi)\right)^{-1}, & \text{if } q=2n;\\
\dfrac{1}{2n+1}\left(1+2\sum\limits_{k=1}^{n}\cos\left(\frac{k}{2n+1}\pi \right)^{-1}\right), &\text{if } q=2n+1.
\end{cases}
\]  
Moreover, $\theta_q\rightarrow 1$ for $q\rightarrow \infty$.

\end{thmx}
 Spiegelhofer~\cite{Spi} improved largely Theorem \ref{ThmFM1} in 2020. He obtained that the Thue--Morse sequence has level of distribution equal to $1$, which can be seen as an optimal result. This means that for all $\varepsilon>0$ there exists $\eta>0$ such that
    \begin{equation}\label{introeq}
        \sum_{1 \leq m \leq x^{1-\varepsilon}}\max_{\substack{y,z \geq 0\\z-y \leq x}}\;\max_{0 \leq r < m}\Big|\sum_{\substack{y \leq n<z\\n \equiv r \Mod{m}}}(-1)^{s_{2}(n)}\Big|\ll_{\varepsilon} x^{1-\eta}.
    \end{equation}
Since $(-1)^{s_2(n)}=1-2 (s_2(n) \bmod 2)$ for all $n\geq 0$, the bound~\eqref{introeq} follows at once from ~\cite[Theorem 2.1]{Spi} recalled below 
\begin{thmx}[Spiegelhofer, 2020]\label{Spiegel20}
    Let $\varepsilon>0$. There exists $\eta>0$ such that
    \begin{equation*}
         \sum_{1 \leq m \leq x^{1-\varepsilon}}\max_{\substack{y,z \geq 0\\z-y \leq x}}\max_{0 \leq r < m}\Big|A(y,z;r,m)-\dfrac{z-y}{2m}\Big| \ll x^{1-\eta},\qquad x\to \infty,
    \end{equation*}
    where 
    \begin{equation*}
       A(y,z;r,m)=\Big|\{y \leq n<z\,:\, t(n)=0,\;\; n\equiv r \Mod{m}\}\Big|.
    \end{equation*}

\end{thmx}

The very remarkable part in this theorem is that $m$ can be almost as large as $x$ and there is a maximum over the residues $r$ modulo $m$. The proof of Spiegelhofer allows in principle to get a lower bound for $\eta$ as a function of $\varepsilon$. 
Explicit results existed before the work of Spiegelhofer. In 2014, Martin, Mauduit, and Rivat~\cite[Proposition 3]{Mart2014} determined such an estimate for sums of type II. 
Although Spiegelhofer did not provide an explicit value for $\eta$, as we will see, his method allows us to do so. The main challenge in obtaining an explicit value of $\eta$ is to have an effective version of the estimate of the Gowers norm associated with the generalization of the Thue--Morse sequence.

\begin{thmx}[Martin/Mauduit/Rivat, 2014]\label{MMR2014}
Let $\alpha \in \R/\Z$. Let $(a_{n})_{n \in \N}$ and $(b_{n})_{n \in \N}$ be sequences of complex numbers such that for all $n \geq 1$, we have $|a_{n}| \leq 1$ and $|b_{n}| \leq 1$. Let $x\geq 2$, $0<\varepsilon \leq 1/2$, $x^{\varepsilon} \leq M,N \leq x$ and $MN \leq x$.
We set
$$
\Theta_{q}:=\left(1-\dfrac{1}{q}\right)\left(1-\sqrt{1-\dfrac{2q-1}{3q(q-1)}}\right),
$$
$$
\eta_{q}:=\max\left(\dfrac{1}{2}-\dfrac{\log(4-2\sqrt{2})}{2\log 2},\dfrac{1}{2}+\dfrac{\log\left(1-\Theta_{q})\right)}{4\log 2}\right)
$$
and let $\gamma_{q} \in\mathbb{R}$ be such that
$$
q^{\gamma_{q}}:=2\max_{t \in \R}\sqrt{\Big|\dfrac{\sin\left(q(\alpha-qt)\pi\right)\sin\left(q(\alpha-t)\pi\right)}{\sin\left((\alpha-qt)\pi\right)\sin\left((\alpha-t)\pi\right)}\Big|}.
$$

Finally, set
\begin{equation}\label{xi2eps}
    \xi_{q,\varepsilon}:=\min\left(\dfrac{\varepsilon}{6},\dfrac{1}{20}\right)\min\left(\dfrac{1}{2}-\eta_{q},2(1-\gamma_{q})\right).
\end{equation}

Then
\begin{equation*}
    \sum_{M<m \leq 2M}\sum_{N<n \leq 2N}a_{m}b_{n}\e(\alpha s_{q}(mn)) \ll x^{1-\xi_{q,\varepsilon}}\log(x).
\end{equation*}
\end{thmx}

 The number $\xi_{q,\varepsilon}$ is entirely explicit. Spiegelhofer showed that Theorem \ref{MMR2014} allows to get a weak version of Theorem 1.1 in~\cite{Spi}:

For $0<\varepsilon<1/2$ and for $D=x^{\varepsilon}$ we have
\begin{equation*}
    \sum_{1 \leq m \leq D}\Big|\sum_{\substack{0 \leq n \leq x\\n \equiv 0 \Mod{m}}}(-1)^{s_{2}(n)}\Big|\leq Cx^{1-\xi'_{2,\varepsilon}},
\end{equation*}
where
\begin{equation*}
\label{xi1def}
\xi'_{2,\varepsilon}=\dfrac{1}{1+\varepsilon}\min\left(\dfrac{\varepsilon}{6},\dfrac{1}{20}\right)\min\left(\dfrac{1}{2}-\eta_{2},2(1-\gamma_{2})\right).  
\end{equation*}

Since $\alpha$ defined in Theorem~\ref{MMR2014} belongs to $\R/\Z$, Theorem~\ref{MMR2014} gives information for more general sequences than the Thue--Morse sequence (case $\alpha=1/2$).

\bigskip
The aim of the present article is to establish the distribution result for the base-$q$ generalization of the Thue--Morse sequence of the strength of Theorem \ref{Spiegel20}. 
\bigskip

Let $b$ be a nonnegative integer\footnote{Note that we do not impose any condition on $\gcd (\ell, b)$.
}  and $\ell$ be an integer such that 
\begin{equation}\label{hypob}
(b,q-1)=1 \mbox{ and } 0<\ell<b. 
\end{equation}
We consider the sequence $(t_{q}(n))_{n\geq 0}$ defined by
 \begin{equation}\label{gentqdef}
     t_{q}(n)=\e\left(\dfrac{\ell}{b}s_{q}(n)\right).
 \end{equation}

This generalization of the classical Thue--Morse sequence has been studied for a long time under various angles. The first appearance of this sequence was in the framework of $q$-multiplicative sequences in work of Bellman and Shapiro \cite{Bel} in 1948. Since then, many mathematicians have been interested in the properties of the sequence $(\e(\alpha s_{q}(n)))_{n \in \mathbb{N}}$ (for $\alpha \in \R)$. We mention Queff\'elec in 1979 (see \cite{Que}) or Coquet in the same year (see \cite{Coq}). Mauduit and Rivat \cite{MauRiv95} obtained a result of Gelfond type for $q$-multiplicative functions along sequences of the form $(\lfloor n^{c} \rfloor)_{n \in \mathbb{N}}$ for $c > 1$ (so-called Piatetski--Shapiro sequences). Their result has been sharpened by M\"{u}llner and Spiegelhofer and coauthors in a series of papers, see~\cite{MS17} or~\cite{Desou}.
More recent references can be found in the article of Spiegelhofer~\cite{Spi}. The main difficulties of generalization in base-$q$ lie particularly in handling the Gowers norm associated with our sequence. We require a new recurrence relation (see~\ref{Relsure}) and the weights involved in it are complex  without necessarily being real, which makes the proof more difficult since we can no longer use ergodic theorems on Markov chains as it was the case in the work of Konieczny~\cite[Corollary 2.4]{Kon}.

\subsection{Notation}\label{not}
This section gathers some notation that will be used in this article. First of all, truncations of the sum-of-digits function play a crucial part in recent works, it will be essential as well in our work. For $\alpha>0$ an integer, we write
\begin{equation*}
\som{\alpha}(n)=s_{q}(n'),
\end{equation*}
where $n' \equiv n \bmod {q^{\alpha}}$ and $0\le n' <{q^{\alpha}}$; moreover, for $\beta> \alpha$, we  write
\begin{equation*}
\somm{\alpha}{\beta}(n)=\som{\beta}(n)-\som{\alpha}(n).
\end{equation*}

Here is a list of more standard notation.
\begin{itemize}

    \item $\N$ denotes the set of the integers $\geq 1$. 
    \item $\mathbb{P}$ denotes the set of prime numbers.
    \item Unless stated otherwise, $p$ will denote a prime number.
    \item For $A \subset \R$ and $x \in \R$, $\mathbb{1}_{A}(x)=1$ if $x \in A$ and $0$ otherwise.
    \item For sets $A_i \subset \R$ for $1\leq i\leq N$, $\biguplus A_i$ is $\bigcup A_i$ and indicates that $A_i \cap A_j=\emptyset$ for all $i\neq j$. 
    \item For  $x \in \R$, $\{x\}=x-\lfloor x \rfloor$ and $\langle x \rangle=\Big\lfloor x+\dfrac{1}{2}\Big\rfloor$.
\item For $x \in \R$, $\|x\|=\min\limits_{n \in \Z}|x-n|$.
\item For $x>0$, $\log^{+}(x)=\max(1,\log(x))$.
\item For $x>0$ and $q>0$, $\log_{q}(x)=\dfrac{\log(x)}{\log(q)}$.
\item For two integers $a$ and $b$, the number  $(a,b)$ denotes the greatest common divisor of $a$ and $b$.
\item For $n \in \N$, $\omega(n)$ denotes the number of prime divisors of $n$ without multiplicity.
\item For $n \geq 2$, $P^{-}(n)$ is the smallest prime divisor of $n$.
\item For $n \geq 2$ and for $p \in \mathbb{P}$, $v_{p}(n)$ is the largest integer $s$ such that $p^{s} \mid n$.
\item For $n \in \N$, $u_{q}(n)=\som{1}(n)$.
\item For $w=\sum\limits_{i=0}^{k-1}w_{i}2^{i}$, where $w_{i} \in \{0,1\}$, we will write
\begin{equation*}
\mathbf{w}=(w_{0},\dots,w_{k-1})
\end{equation*}
as a shorthand (to emphasize the binary digit expansion of the integer $w$). We will sometimes, without further notice, use both the digits vector $\mathbf{w}$ and the represented integer $w$.
\item We write, as usual,
$\e(t)=e^{2\pi i t}$. 
\end{itemize}
 

\subsection{Main results}

The aim of this work is to generalize the result of Spiegelhofer~\cite{Spi} to base-$q$ and general modulus. 
Moreover, we want to find a value of $\eta$ that is almost optimal (as $\varepsilon$ approaches $0$) with respect to the implemented method. Let $m,b \geq 2$ be two integers and $r \in \{0,\dots,m-1\}$. Let $x, y, z$ be three real numbers such that $0\leq y<z$ with $z-y \leq x$. We define
\begin{equation}\label{defN}
    N_{y,z}(a,b;r,m)=\Big|\{y \leq n <z\ :\ s_{q}(n) \equiv a \Mod{b},\ n \equiv r \Mod{m} \}\Big|.
\end{equation}

\begin{thm}
\label{Thm0}
Let $0<\varepsilon<1$. Let $b,q \geq 2$ be two integers such that $(b,q-1)=1$. There exist a constant $C=C(\varepsilon,b,q)>0$ and an exponent $\eta=\eta(\varepsilon,b,q)>0$ such that

\begin{equation*}
 \sum_{1 \leq m \leq x^{1-\varepsilon}}\max_{\substack{y,z\\0\leq y<z\\z-y\leq x}}\max_{r \geq 0}\Big|N_{y,z}(a,b;r,m)-\dfrac{z-y}{bm}\Big|
 \leq C x^{1-\eta}.
\end{equation*}
An admissible value for $\eta$ is given for $\varepsilon<\dfrac{2}{3}(1-\lambda)$ (where $\lambda$ is defined in Theorem~\ref{Gelfond}) by
\begin{equation}
\label{value-eta}
  \eta\left(\varepsilon\right)=\dfrac{\varepsilon^{3}\min\left(1/4,3\log_q\left(P^{-}(q)\right)\right)}{7200\times 8^{1/\varepsilon}\left(\log(4q/\varepsilon)+5b\log(q)/\varepsilon\right)}\times \exp\left({-\dfrac{5}{\varepsilon}\left(\log(4q/\varepsilon)+\dfrac{5b\log(q)}{\varepsilon}\right)}\right).
\end{equation}
\end{thm}

Spiegelhofer \cite[p.2568, (i)]{Spi}
asked whether we can choose $D=x\log(x)^{-B}$ for some $B>0,$ and have $x\log(x)^{-A}$ as an error term. For our result, it might be possible to get an explicit $\varepsilon$ depending on $x$  tending to $0$ ($\varepsilon=1/(A\log\log x))$, for instance, for some constant $A$). However, our method gives $C$ as a function of $\varepsilon$. It would be interesting to study $C$ in Theorem~\ref{Thm0} and its degree of dependence in $\varepsilon$.

\begin{rem}
For $\varepsilon \rightarrow 0$, we have the following equivalent for $\eta$,
\begin{equation*}
    \eta \sim \dfrac{\varepsilon^{4}\exp(-\varepsilon^{-2}(25b\log(q)+o(1)))}{36000b\log(q)8^{\varepsilon^{-1}}}\min\left(1/4,3\log_q\left(P^{-}(q)\right)\right).
\end{equation*}

This is evidently a very small quantity. For $q=b=2$, we used a Python-program to calculate some approximate values for $\eta$ given by \eqref{value-eta}, see  Annexe~\ref{appendixeta}. 
\end{rem}

In order to prove Theorem~\ref{Thm0}, we will prove the following theorem, which is at the heart of the article of Spiegelhofer~\cite{Spi}.
\begin{thm}
\label{Thm1}
Let $ 0< \delta_{1}\leq \delta_{2}<1$ be real numbers. Let $b,q \geq 2$ be two integers such that $(b,q-1)=1$. Let $x>0$ be a real number and $D$ be an integer such that
\begin{equation*}
   x^{\delta_1} \leq D \leq x^{\delta_2}. 
\end{equation*}
Then there exist a constant $C=C(\delta_1,\delta_2,b,q)>0$ and an exponent $\eta=\eta(\delta_1,\delta_2,b,q)>
0$ such that 
\begin{equation*}
     \sum_{D<m\leq qD} \;\max_{\substack{y,z\\0\leq y<z\\z-y\leq x}}\;\max_{r \geq 0}\Big|N_{y,z}(a,b;r,m)-\dfrac{z-y}{bm}\Big| \leq Cx^{1-\eta}.
\end{equation*}
An admissible value for $\eta$ is given by
\begin{align}
\label{defetadd}
\eta&:=\dfrac{\delta_{1}(1-\delta_{2})^{2}\min\left(1/4,3\log_q\left(P^{-}(q)\right)\right)}{1800\times 8^{(1-\delta_{2})^{-1}}\left(\log(4q/(1-\delta_{2}))+5b\log(q)/(1-\delta_2)\right)} \nonumber\\
    &\quad \times \exp\left({-\dfrac{5}{1-\delta_{2}}\left(\log(4q(1-\delta_{2})^{-1})+\dfrac{5b\log(q)}{1-\delta_{2}}\right)}\right). \nonumber\\
\end{align}
\end{thm}
\begin{rem}
For $\delta_2 \rightarrow 1$, we have the following equivalent for $\eta$,
\begin{equation*}
    \eta \sim \dfrac{\delta_1(1-\delta_2)^{3}}{ 9000b\log (q) 8^{(1-\delta_2)^{-1}}}\times \exp\left(\dfrac{-25b\log (q)+o(1))}{(1-\delta_2)^{2}}\right)\times  \min\left(1/4,3\log_q\left(P^{-}(q)\right)\right).
\end{equation*}
\end{rem}

To prove this theorem, we will follow Spiegelhofer's approach  \cite{Spi}. A standard way to detect arithmetic progressions is to use exponential sums. In the following theorems we prove some cancellation on average over the modulus of the arithmetic progression. Theorem~\ref{Thm1} is a consequence of Theorem~\ref{Thm2} below, which might have applications in other problems as well.

\begin{thm}
\label{Thm2}
Let $\rho_{2}\geq\rho_{1}>0$. Let $b,q \geq 2$ be two integers such that $(b,q-1)=1$ and let $N, D>1$ be two integers 
such that
\begin{equation*}
    N^{\rho_{1}} \leq D \leq N^{\rho_{2}}.
\end{equation*}
Let $\ell \in \{1,\dots,b-1\}.$ and $\xi $ be a real number. We define $S_{0}(N,D,\xi$) by
\begin{equation*}
     S_{0}=S_{0}(N,D,\xi)=\sum\limits_{D \leq m <qD}\max_{a \geq 0}\Big|\sum\limits_{0 \leq n < N}\e\left(\dfrac{\ell}{b}s_q(nm+a)\right)\e(n\xi)\Big|.
\end{equation*}
Then there exist $C=C(b,q,\rho_{1},\rho_{2})>0$ and $\eta=\eta(b,q,\rho_1,\rho_2)>0$ such that
\begin{equation*}
\Big|\dfrac{S_{0}(N,D,\xi)}{D}\Big| \leq CN^{1-\eta}.
\end{equation*}
An admissible value for $\eta$ is given by
\begin{align*}
    \eta&=\dfrac{\rho_{1}\min\left(1/4,3\log_q\left(P^{-}(q)\right)\right)}{8^{1+\rho_{2}}(288\rho_{2}+300)(3\rho_{2}+2)\left(\log((3\rho_{2}+4)q)+b\log(q)(3\rho_{2}+5)\right)}\\
    &\qquad \times\exp\left({-(3\rho_{2}+5)\left(\log((3\rho_{2}+4)q)+(3\rho_{2}+5)b\log(q)\right)}\right).
\end{align*}
\end{thm}
The proof idea is to apply the van der Corput inequality (and variations thereof) a certain number of times in order to remove digits in the underlying quantities and to reduce the estimate of the exponential sums to estimates of Gowers norms for the sequence $(t_{q}(n))_{n\geq 1}$.
In this article, we prove the following property of Gowers norms for this sequence.
\begin{thm}
\label{NDGG}
Let $k\geq 3$ be an integer and $0<\ell<b$.
and set $$K=\left\lfloor\dfrac{\log(k)}{\log(q)}\right\rfloor+1.$$ Let $b,q$ be two integers such that $(b,q-1)=1$.
We define
\begin{equation*}
\eta_{0}:=\dfrac{1}{\log(q)(K+(k+1)b)q^{(k+1)(K+b(k+1))}}.
\end{equation*}
Then, as $\rho \rightarrow \infty$,
\begin{equation*}
    \dfrac{1}{q^{(k+1)\rho}}\sum_{\substack{0\leq n < q^{\rho}\\0\leq h_{0},\dots,h_{k-1}<q^{\rho}}}\e\left(\dfrac{\ell}{b}\sum_{\mathbf{w}=(w_{0},\dots,w_{k-1}) \in \{0,1\}^{k}}(-1)^{s_{2}(w)}s_{q}(n+\mathbf{w}\cdot \mathbf{h})\right) \ll q^{-\eta_{0} \rho}.
\end{equation*}
Moreover, we get the same formula for the function $\som{\rho}$: as $\rho \rightarrow +\infty$,
\begin{equation*}
    \dfrac{1}{q^{(k+1)\rho}}\sum_{\substack{0\leq n < q^{\rho}\\0\leq h_{0},\dots,h_{k-1}<q^{\rho}}}\e\left(\dfrac{\ell}{b}\sum_{\mathbf{w}=(w_{0},\dots,w_{k-1}) \in \{0,1\}^{k}}(-1)^{s_{2}(w)}\som{\rho}(n+\mathbf{w}\cdot \mathbf{h})\right) \ll q^{-\eta_{0} \rho}.
\end{equation*}
\end{thm}
This theorem generalizes the result obtained by Konieczny in his article~\cite{Kon} in 2019. As Konieczny commented in~\cite[Remark 2.5]{Kon}, the calculation of the spectral gap of the matrix $(p_{\mathbf{r_{0}},\mathbf{r_{1}}})_{(\mathbf{r_{0}},\mathbf{r_{1}}) \in \mathfrak{R}^{2}}$, where $p_{\mathbf{r_{0}},\mathbf{r_{1}}}$ is defined in \eqref{Ptrans} and $\mathfrak{R}$ is defined Lemma~\ref{GFini}, yields another admissible value for $q=b=2$. However, we will use a different approach, since $p_{\mathbf{r_{0}},\mathbf{r_{1}}}$ is not necessarily a real number for general $q$ and $b$. The case $q=b=2$ leads us to the Gowers norm estimate for the classical Thue--Morse sequence. We mention also the result of Byszewski, Konieczny and Müllner~[\cite{Bykomu}, Theorem A]. They showed that all automatic sequence orthogonal to periodic sequences are highly Gowers uniform (in the sense of [\cite{Bykomu}, (1), p.2]). Note that  $\eta_0$ is not explicit in their work (as it is in Konieczny's result in~\cite{Kon}), and the case $s_q \bmod b$ is not mentioned explicitly.

\vspace{0.2cm}

The structure of the article is as follows. In Sections \ref{AuxRes} and \ref{sectionqlemmas} we collect several lemmas of a technical nature that we use throughout our article. They concern classical (or standard) results on discrepancy estimates, various types of the van der Corput inequality, as well as base-$q$ generalizations of some lemmas of~\cite{Spi} regarding carry propagations, the cutting procedure of digits etc. We will proceed along the lines of~\cite{Spi} and will make frequent use of Farey fractions. The reader will find a short account on Farey fractions in Annexe \ref{Appendice}. Section \ref{NormesSec} is devoted to the Proof of Theorem~\ref{NDGG}, in Section~\ref{Theo2.5} we prove Theorem~\ref{Thm2}. Finally, in Sections \ref{secpreuveThm1} and \ref{secthm0} we give the proofs of Theorem~\ref{Thm1} and Theorem~\ref{Thm0}.

\section{Tools}\label{AuxRes}
In this section we collect various definitions and lemmas that will be used throughout the article. 
\subsection{Some technical lemmas}
The following lemma will be used at the end of the proof of Theorem~\ref{Thm2}. This inequality allows us to reduce it to Theorem~\ref{Thm1}.
\begin{lemme}
\label{Lemma37}
Let $x\leq  y \leq z$ be real numbers and $(a_{n})_{n \in \N} \in \C^{\N}$ for $x \leq n < z$. Then,
\begin{equation*}
    \Big|\sum_{x \leq n<y}a_{n}\Big| \leq \int_{0}^{1}\min(y-x + 1,\|\xi\|^{-1})\Big|\sum_{x \leq n<z}a_{n}\e(n\xi)\Big|d\xi.
\end{equation*}
\end{lemme}
\begin{proof}
We refer the reader to~\cite[Lemma 3.7]{MS17}.
\end{proof}

The following elementary lemma gives a bound on the number of solutions of a certain congruence. It will be used during the proof of Theorem $\ref{Thm2}$.

\begin{lemme}
\label{PropA}
Let $k \in \N^{*}$, $\rho \geq \gamma>0$ be integers. Let $q=p_{1}^{\alpha_{1}}\cdots p_{r}^{\alpha_{r}} \geq 2$ and $M \in \N^{*}$ such that for all $1 \leq i \leq r,$ $\alpha_i\geq 1$, $p_{i}^{\gamma} \nmid M$. Then, for all $0 \leq a<q^{\rho}$, we have
\begin{equation*}
    |\{h \in \{0,\dots,q^{\rho}-1\}\,:\, hM\equiv a \Mod{q^{\rho}}\}| \leq q^{\gamma}. 
\end{equation*}
\end{lemme}
\begin{proof}

If $(M,q^\rho)\nmid a$, then the congruence $hM\equiv a\Mod{q^\rho}$ has no solution $h$. Otherwise this congruence is equivalent to 
\[ 
h\frac{M}{(M,q^\rho)}\equiv \frac{a}{(M,q^\rho)}
\Mod{\frac{q^\rho}{(M,q^\rho)}}.
\]
Since 
 $\Big (\frac{M}{(M,q^\rho)}, \frac{q^\rho}{(M,q^\rho)}\Big )=1 $
 the number of solutions $h\in\{ 0,\ldots ,q^\rho -1\}$ is $(M,q^\rho)$. Since $p_{i}^{\gamma} \nmid M$ for all $1 \leq i \leq r$,  we have $v_{p_{i}}(M)< \gamma$. Moreover, $v_{p_i}(q^{\rho}) \geq \rho$. Therefore, $(M, q^{\rho}) \leq q^{\gamma}$.
 \end{proof}

\begin{prop}
\label{aplusb}
The following properties hold true:
\begin{enumerate}
    \item[(1)] Let $0 < \varepsilon \leq \dfrac{1}{2}$. We suppose that $a, b \in \mathbb{R}$ such that $\|a\| < \varepsilon$ and $\|b\| \geq \varepsilon$. Then
\begin{equation}
\label{petit1}
\lfloor a+b \rfloor = \left\lfloor a+\frac{1}{2}\right\rfloor + \lfloor b\rfloor.\\
\end{equation}
\item[(2)] Let $a \in \R$ and $n \in \N$. Then
\begin{equation}
\label{petit2}
\|na\| \leq n\|a\|.\\
\end{equation}
\item[(3)] Let $a \in \R$, $\varepsilon >0$ and $n \in \N$ such that $\|a\| < \varepsilon$ and $2n\varepsilon <1$. Then
\begin{equation}
\label{petit3}
\langle na \rangle=n\langle a \rangle.
\end{equation}
\end{enumerate}
\end{prop}

\begin{proof}
We refer the reader to~\cite[Lemma 4.1]{Spi}.
\end{proof}
\subsection{The discrepancy of a real number}
\begin{de}
\label{DefDisc}
Let $\alpha \in \R$ and $N \geq 1$ be an integer. We define the $N$-discrepancy of $\alpha$, denoted as $D_{N}(\alpha)$, by 
\begin{equation*}
D_{N}(\alpha)=\sup_{\substack{0 \leq x \leq 1\\y \in \R}}\Big|\dfrac{1}{N}\sum_{0 \leq n < N}\mathbb{1}_{[0,x[+y+\Z}(n\alpha)-x\Big|.
\end{equation*}
\end{de}
The notion of $N$-discrepancy appears naturally in the following result which is used by Müllner and Spiegelhofer~\cite[Lemma 3.3]{MS17}.
\begin{prop}
\label{PropUnNEUF}
Let $J$ be an interval in $\R$ containing $N$ integers. Let $\alpha, \beta \in \R$. Let $t, T, \ell, L$ be integers such that $0 \leq t < T$ and $0 \leq \ell < L$. Then
\begin{equation*}
\Big|\{n \in J\ : \dfrac{t}{T} \leq \{n\alpha+\beta\}< \dfrac{t+1}{T},\ \lfloor n\alpha+\beta\rfloor \equiv \ell \Mod{L}\}\Big| =\dfrac{N}{LT}+O\left(ND_{N}\left(\dfrac{\alpha}{L}\right)\right)
\end{equation*}
with an absolute implied constant.
\end{prop}


\subsection{The van der Corput inequality}

As in previous articles on the sum-of-digits function, the van der Corput inequality and its variations play a major role. The original version can for example be found in the book of Graham/Kolesnik~\cite[(2.3.4)]{Graham}.

\begin{lemme}
\label{VDCC}
Let $0\leq a<b$ be two integers and set $I=[a,b[$. Let $f\,:\, \N \rightarrow \C$ be a complex-valued function such that $f(n)=0$ for $n \notin I$.
Let $0<H \leq |I|$ be an integer. We have 
\begin{equation*}
\Big|\sum_{n \in I}\e(f(n))\Big|^{2} \leq 2\dfrac{|I|^{2}}{H}+4\dfrac{|I|}{H}\sum_{1 \leq h<H}\Big|\sum_{\substack{n \in I\\n+h \in I}}\e(f(n+h)-f(n))\Big|.
\end{equation*}
\end{lemme}
 Mauduit and Rivat~\cite[Lemme 17]{MAURIV} generalized the van der Corput inequality for general shifts, and showed that it can be used, in an elegant manner, to remove lower-placed digits. As we will see later, removing iteratively windows of digits via this inequality, allows to get a small window of digits only to consider in the end. For the purpose of handling better these iterations, we rewrite their inequality with a single sum over all  $h \in \{0,\dots, H-1\}$, without isolating the diagonal term ($h=0$).

\begin{lemme}
\label{VDCG}
Let $0\leq a<b$ be two integers and set $I=[a,b[$. Let $f\,:\, \N \rightarrow \C$ be a complex valued function such that $f(n)=0$ for $n \notin I$. Let $K>0$ be an integer.
Let $0<H \leq |I|$ be an integer. We have 
\begin{equation*}
\Big|\sum_{n \in I}\e(f(n))\Big|^{2} \leq 2 \dfrac{|I|+K(H-1)}{H}\sum_{0 \leq h<H}\Big|\sum_{\substack{n \in I\\n+Kh \in I}}\e(f(n+Kh)-f(n))\Big|.
\end{equation*}
\end{lemme}




\section{Base-$q$ generalization of several lemmas}\label{sectionqlemmas}

 In this section, we generalize to base-$q$ several lemmas used by Spiegelhofer in~\cite{Spi} for the base $2$. We refer the reader to the beginning of Section \ref{not} for the notation that we use in the following lemmas.

 The first lemma will be used when we begin to perform many digits shiftings, in order to remove a large portion of the digits of the integers and to be able to concentrate our attention to a small window of digits only (see~\cite[p.2575]{Spi}).
\begin{lemme}[Cutting out digits]
\label{diffab}
Let $\alpha, u_{0}, u_{1}\geq 0$ be integers. Let $\beta>\alpha.$ We have the equality
\begin{equation*}
\som{\beta}(u_{0}+q^{\alpha}u_{1})-\som{\beta}(u_{0})=\somm{\alpha}{\beta}(u_{0}+q^{\alpha}u_{1})-\somm{\alpha}{\beta}(u_{0}).
\end{equation*}
\end{lemme}

\begin{proof}
By the definition of $s_q^{\alpha ,\beta}$, we have
\[ 
\somm{\alpha}{\beta}(u_{0}+q^{\alpha}u_{1})-\somm{\alpha}{\beta}(u_{0})=
\som{\beta}(u_{0}+q^{\alpha}u_{1})-\som{\beta}(u_{0})-
(\som{\alpha}(u_{0}+q^{\alpha}u_{1})-\som{\alpha}(u_{0})),
\]
and we observe that $\som{\alpha}(u_{0}+q^{\alpha}u_{1})-\som{\alpha}(u_{0})=0$
since $u_0+q^\alpha u_1\equiv u_0\Mod{q^\alpha}$.
\end{proof}
The next Lemma allows us to justify the stability of the recursion formula when we work with $\som{\rho}$ instead of $s_q$ in the proof of~\ref{NDGG}.
\begin{lemme}
\label{recsrho}
Let $\rho \geq 2$. For all $n \in \N$, we have
\begin{equation*}
    \som{\rho}(n)=\som{\rho-1}\left(\left\lfloor\dfrac{n}{q}\right\rfloor\right)+u_q(n).
\end{equation*}
\end{lemme}

The following property is useful to transform the sum-of-digits function $s_q$ into a periodic function $\som{\lambda}$, at the expense of a manageable error. This transformation will occur during the first iteration of the van der Corput inequality. Mauduit and Rivat~\cite[Lemme 5, p.1607]{MauRivGelf} have used this method in their resolution of the problem of Gelfond regarding the distribution of $s_q$ along primes. It has since then be applied in various contexts and shown to be very useful. The following version is a generalization of~\cite[Lemma 4.5]{Spi}. Note that we will only use the case $\alpha,\beta \in \N$ but we state the more general (Beatty sequence) version since the proof is not more difficult and the analogy to~\cite{Spi} is more apparent in that way.

\begin{lemme}[Carry propagation lemma]
\label{PropaRetenue}
Let $r > 0$, $\lambda \geq 1$, $N\geq 1$ be integers and $\alpha,\beta \in\mathbb{R}$ with $\alpha > 0$ and $\beta \geq 0$. Let $I \subset [0,+\infty[$ be an interval in $\mathbb{R}$ containing exactly $N$ consecutive integers.
Then
\begin{equation}\label{setexeptions}
\Big|\Big\{n \in I\ :\ s_{q}(\lfloor\alpha(n+r)+\beta \rfloor)-s_{q}(\lfloor\alpha n+\beta \rfloor) \neq \som{\lambda}(\lfloor \alpha(n+r)+\beta \rfloor)-\som{\lambda}(\lfloor\alpha n+\beta \rfloor)\Big\}\Big| < r\left(\dfrac{N\alpha}{q^{\lambda}}+2\right).
\end{equation}
\end{lemme}

\begin{proof}
If $r\alpha \geq q^{\lambda}$ then $r\left(\dfrac{N\alpha}{q^{\lambda}}+2\right)\geq N$ and the result follows immediately. We may therefore suppose that $r\alpha < q^{\lambda}$. Let $n \in I$ such that $n\alpha +\beta \in [kq^{\lambda},(k+1)q^{\lambda}[$ for some $k \in \mathbb{N}$ that we will make precise later.
\\If $n\alpha +\beta +r\alpha \in [kq^{\lambda},(k+1)q^{\lambda}[$, then $n$ is not in the set involved on the left hand side of~\eqref{setexeptions}.
We therefore only need to consider the integers $n$ such that
\begin{equation*}
n\alpha +\beta +r\alpha \geq (k+1)q^{\lambda}.
\end{equation*}
This second condition (combined with the condition $n\alpha +\beta \in [kq^{\lambda},(k+1)q^{\lambda}]$) allows us to estimate $n$. We obtain
\begin{equation*}
\dfrac{(k+1)q^{\lambda}-\beta-r\alpha}{\alpha} \leq n < \dfrac{(k+1)q^{\lambda}-\beta}{\alpha},
\end{equation*}
and the number of such integers is therefore bounded by
\begin{equation*}
\dfrac{(k+1)q^{\lambda}-\beta}{\alpha}-\dfrac{(k+1)q^{\lambda}-\beta-r\alpha}{\alpha}=r.
\end{equation*}
Finally, we need to bound the number of possible $k$. Write $I=[a_{1},a_{2}[$. Since $n \leq a_{2}$, we have
\begin{equation*}
a_{2}\alpha+\beta+r\alpha \geq (k+1)q^{\lambda},
\end{equation*}
which implies
\begin{equation*}
k \leq \dfrac{a_{2}\alpha+\beta+r\alpha}{q^{\lambda}}-1.
\end{equation*}
Moreover, we have
\begin{equation*}
(k+1)q^{\lambda}>a_{1}\alpha+\beta,
\end{equation*}
and
\begin{equation*}
k>\dfrac{a_{1}\alpha+\beta}{q^{\lambda}}-1.
\end{equation*}
Therefore, the number of possible $k$ is bounded by
\begin{equation*}
\dfrac{(a_{2}-a_{1})\alpha+r\alpha}{q^{\lambda}} \leq \dfrac{(N+1)\alpha}{q^{\lambda}}+1 < \dfrac{N\alpha}{q^{\lambda}}+2,
\end{equation*}
where in the last step we used $\alpha\leq r\alpha < q^{\lambda}$. This concludes the proof.

\end{proof}
The following lemma provides a bound on the number of $n<N$ such that $\|n\alpha+\beta\|$ is close to an integer. It is the analogue in base-$q$  of the arguments given by Spiegelhofer~\cite[p.2575]{Spi}. This lemma will be useful to us during the digit-shifting phase.
\begin{lemme}
\label{DisPetitN}
Let $N>1$ be an integer, $\sigma>1$ and $H \in \N$ such that $H<q^{\sigma-1}$. Moreover, let $\alpha, \beta \in \R$. We have 
\begin{equation*}
|\{n \in \{0,\dots,N-1\}\ :\  \|n\alpha+\beta\|<H/q^{\sigma}\}| \leq ND_{N}(\alpha)+\dfrac{2HN}{q^{\sigma}}.
\end{equation*}
\end{lemme}
\begin{proof}
This follows by elementary interval translations and the definition of the discrepancy, i.e.
\begin{align*}
    |\{n \in \{0,\dots,N-1\}\ :\  \|n\alpha+\beta\|<H/q^{\sigma}\}| &= |\{n \in \{0,\dots,N-1\}\ :\ n\alpha+\beta \in (\Z+[-H/q^{\sigma},H/q^{\sigma}])\}|,\\
    &= |\{n \in \{0,\dots,N-1\}\ :\ n\alpha \in (\Z+[0,2H/q^{\sigma}]-\beta-H/q^{\sigma})\}|\\
    &\leq \dfrac{2HN}{q^{\sigma}}+ND_{N}(\alpha).
\end{align*}
\end{proof}

As each digit shift introduces an error term involving the discrepancy, and since we perform $(k-1)$ such shifts, the following lemma will be used to estimate the overall error generated by multiple digit shifts. The proof is available in the article of Müllner and Spiegelhofer~\cite[Lemma 3.4, p.704]{MS17} for $q=2$, but holds for each $q \geq 2$ (with the obvious modifications in the proof). 
\begin{prop}
\label{DiscRes}
Let $q \geq 2$. There exists an absolute constant C, such that for all $m, n, N\in \mathbb{N}$,
\begin{equation*}
\sum_{0 \leq d<q^{m}}D_{N}\left(\dfrac{d}{q^{m}}\right) \leq C\; \dfrac{N+q^{m}}{N}(\log^{+}N)^{2}.
\end{equation*}
\end{prop}

 The following lemma is at the heart of the approximation process when we discard digits in the iterations of the van der Corput inequality, see~\cite[Lemma 5.1]{Spi} . We use standard notation concerning Farey fractions, see Annexe~\ref{appendix1}. 
\begin{lemme}
\label{CardA}
Let $\mu, \nu, \sigma, \gamma$ be integers with $\nu+1 \geq 3\gamma$ and $k \geq 3$ be an integer.

We define, for $m \in \{0, \dots, q^{\nu+1}-1\}$,
\begin{equation*}
\mathfrak{M}_{1}=P_{q^{\sigma}}\left(\dfrac{P_{q^{2\mu+2\sigma}}(m/q^{2\mu})}{q^{(k-2)\mu}}\right).
\end{equation*}
For $1<i<k-1$ we set
\begin{equation*}
\mathfrak{M}_{i}=P_{q^{\sigma}}\left(\dfrac{P_{q^{\mu+2\sigma}}(m/q^{(i+1)\mu})}{q^{(k-i-1)\mu}}\right)
\end{equation*}
and
\begin{equation*}
\mathfrak{M}_{k-1}=P_{q^{\sigma}}\left(\dfrac{P_{q^{\mu+\sigma}}(m/q^{k\mu})}{q^{k\mu}}\right),
\end{equation*}
where $P_{n}(\alpha)$  is the numerator of the Farey fraction of order $n$ ``closest'' to $\alpha$, according to Definition~\ref{DefFarey}.
Let
\begin{equation*}
A=\{m \in \{0,\dots,q^{\nu+1}-1\}\,:\,\exists i \in \{1,\dots,k-1\},\;\exists p\mid q\,,\;p^{3\gamma}| \,\mathfrak{M}_{i}\}.
\end{equation*}
Then
\begin{equation}\label{Aestimate}
|A|=O_{q}(q^{\nu+1-3\gamma\log_{q}(P^{-}(q))}).
\end{equation}
\end{lemme}

\begin{proof}
To make the proof easier to read, we will work only with $\mathfrak{M}_{1}$, the other cases ($2 \leq i \leq k-1$) are handled almost similarly. We refer the reader to the article by Spiegelhofer~\cite[Lemma 5.1]{Spi} to see the slight changes that appear for the cases $2\leq i\leq k-1$. We first deal with the case where $(m,q)=1$ (Case 1). Some minor changes and estimates will allow us to handle the general case (Case 2).

\vspace{0.3cm}
\underline{Case 1}: $(m,q)=1$.

As a first step, we establish the following formula for $\mathfrak{M}_{1}$:
\begin{equation}\label{frakM1formula}
\mathfrak{M}_{1}=P_{q^{\sigma}}(m_{0}/q^{(k-2)\mu})+m_{1}Q_{q^{\sigma}}(m_{0}/q^{(k-2)\mu}),
\end{equation}
with $m_0$ defined below.
To begin with, we remark that $\dfrac{m}{q^{2\mu}} \in \mathcal{F}_{q^{2\mu+2\sigma}}$, thus
\begin{equation*}
\dfrac{m}{q^{2\mu}}=\dfrac{P_{q^{2\mu+2\sigma}}(m/q^{2\mu})}{Q_{q^{2\mu+2\sigma}}(m/q^{2\mu})}.
\end{equation*}
Moreover, since $(m,q)=1$, we have
\begin{equation}\label{simplifyM1}
P_{q^{2\mu+2\sigma}}(m/q^{2\mu})=m.
\end{equation}
By performing the Euclidean division of $m$ by $q^{(k-2)\mu}$ and using the condition $m< q^{\nu+1}$, we get
\begin{equation*}
m=m_{0}+q^{(k-2)\mu}m_{1},
\end{equation*}
with
\begin{align*}
&m_{0} \in \{0,\dots,q^{(k-2)\mu}-1\},\\
&m_{1} \in \{0,\dots,q^{\nu+1-(k-2)\mu}-1\}.
\end{align*}

We write
\begin{equation*}
\dfrac{P}{Q}\leq \dfrac{m_{0}}{q^{(k-2)\mu}} \leq \dfrac{P'}{Q'},
\end{equation*}
where $\dfrac{P}{Q}$ and $\dfrac{P'}{Q'}$ are neighbours in $\mathcal{F}_{q^{\sigma}}$.
Thus, by adding $m_{1}$ to each term in the double inequality, we have
\begin{equation}
\label{Doubleineq}
\dfrac{P+m_{1}Q}{Q}\leq \dfrac{m}{q^{(k-2)\mu}}\leq \dfrac{P'+m_{1}Q'}{Q'}.
\end{equation}
Since $(P,Q)=1$ and $(P',Q')=1$, all the fractions are irreducible. Also, since $\dfrac{P}{Q}$ and $\dfrac{P'}{Q'}$ are neighbours in $\mathcal{F}_{q^{\sigma}}$, then, by translation, $(P+m_{1}Q)/Q$ and $(P'+m_{1}Q')/Q'$ are still neighbours in $\mathcal{F}_{q^{\sigma}}$.

We distinguish two cases according to the place of $m_0/q^{(k-2)\mu}$ with respect to the midpoint of the Farey interval $[P/Q, P'/Q'[$. First we suppose that
\begin{equation}
\label{midpointineq}
    \dfrac{m_{0}}{q^{(k-2)\mu}} \leq \dfrac{P+P'}{Q+Q'}.
\end{equation}
Then, by Definition~\ref{DefFarey}, we have 
\begin{equation*}
\dfrac{P}{Q}=\dfrac{P_{q^{\sigma}}(m_{0}/q^{(k-2)\mu})}{Q_{q^{\sigma}}(m_{0}/q^{(k-2)\mu})}.
\end{equation*}
Moreover, since we have
\begin{align*}
\dfrac{m}{q^{(k-2)\mu}}&=\dfrac{m_{0}}{q^{(k-2)\mu}}+m_{1}\\
&\leq \dfrac{P+P'}{Q+Q'}+m_{1}\\
&\leq \dfrac{(P+m_{1}Q)+(P'+m_{1}Q')}{Q+Q'},
\end{align*}
the number $\dfrac{m}{q^{(k-2)\mu}}$ is smaller than the midpoint of $\dfrac{P+m_{1}Q}{Q}$ and $\dfrac{P'+m_{1}Q'}{Q'}$.
This implies by~\eqref{simplifyM1} and~\eqref{Doubleineq} that
\begin{align*}
    \mathfrak{M}_{1}&=P_{q^{\sigma}}\left(\dfrac{P_{q^{2\mu+2\sigma}}(m/q^{2\mu})}{q^{(k-2)\mu}}\right)\\
    &=P+m_{1}Q\\
    &=P_{q^{\sigma}}(m_{0}/q^{(k-2)\mu})+m_{1}Q_{q^{\sigma}}(m_{0}/q^{(k-2)\mu}),
\end{align*}
as announced in~\eqref{frakM1formula}. For the second case, we suppose now that
\begin{equation}
\label{midpointineq2}
    \dfrac{m_{0}}{q^{(k-2)\mu}}>\dfrac{P+P'}{Q+Q'}.
\end{equation}
A similar calculation as above shows that
\begin{align*}
    \mathfrak{M}_{1}&=P_{q^{\sigma}}\left(\dfrac{P_{q^{2\mu+2\sigma}}(m/q^{2\mu})}{q^{(k-2)\mu}}\right)\\
    &=P'+m_{1}Q'\\
    &=P_{q^{\sigma}}(m_{0}/q^{(k-2)\mu})+m_{1}Q_{q^{\sigma}}(m_{0}/q^{(k-2)\mu}).
\end{align*}
\\In both cases we arrive at the same conclusion, identity~\eqref{frakM1formula}.
\\We now turn our attention to~\eqref{Aestimate}. We suppose that there exists $p|q$ such that $p^{3\gamma} \mid \mathfrak{M}_{1}$. We observe that 
$p \nmid Q_{q^{\sigma}}(m_{0}/q^{(k-2)\mu})$
since otherwise, by~\eqref{frakM1formula}, $p $ would divide $P_{q^{\sigma}}(m_{0}/q^{(k-2)\mu})$ which contradicts $$(Q_{q^{\sigma}}(m_{0}/q^{(k-2)\mu}),P_{q^{\sigma}}(m_{0}/q^{(k-2)\mu}))=1.$$

For a fixed number $m_{0} \in \{0,\dots,q^{(k-2)\mu}-1\}$, the number of $m_{1} \in \{0, \dots, q^{\nu+1-(k-2)\mu}-1\}$ such that 
\begin{equation*}
    m_{1}Q_{q^{\sigma}}(m_{0}/q^{(k-2)\mu}) \equiv -P_{q^{\sigma}}(m_{0}/q^{(k-2)\mu}) \Mod{p^{3\gamma}}
\end{equation*}
is bounded by $q^{\nu+1-(k-2)\mu}/{p^{3\gamma}}.$ Thus, the number of $m \in \{0, \dots, q^{\nu+1}-1\}$ such that $p^{3\gamma} \mid \mathfrak{M}_{1}$ is at most
${q^{\nu+1}}/{p^{3\gamma}}$.

\medskip

\underline{Case 2:} $(m,q)>1$.
\\The only difference here with respect to the former case lies in the beginning of the argument. The fraction ${m}/{q^{2\mu}}$ is no longer irreducible but it suffices to reduce it to ${m'}/{q'}$ (where $(m',q')=1$) and to proceed exactly in the same way as above with the previous argument with $q',m'$ in place of $q,m$.

\bigskip

We now collect all the information to estimate the cardinality of the set $A$ of the statement. We get
\begin{align*}
    |A| &\leq \sum_{\substack{p \in \mathbb{P} \\p|q}}|\{m \in \{0,\dots,q^{\nu+1}-1\}\,: \quad \exists i \in \{1,\dots,k-1\},\quad p^{3\gamma}| \,\mathfrak{M}_{i}\}|\\
    &\leq q^{\nu+1}\sum_{\substack{p \in \mathbb{P} \\p|q}}p^{-3\gamma}\\
    &\leq \omega(q)q^{\nu+1}P^{-}(q)^{-3\gamma}.
\end{align*}
\end{proof}

\section{Proof of Theorem~\ref{NDGG}}\label{NormesSec}

\subsection{Recursion formula and setting up the graph}

Throughout this section, in view of a further use of Lemma~\ref{CardA}, we use $k \geq 3$ (The arguments of this section still remain valid for $k=2$). Recall that we use throughout the article the assumption~\eqref{hypob},
$$(b,q-1)=1 \mbox{ and } 0<\ell<b. $$For $w \in \{0,\dots,2^{k}-1\}$ we write $w=\sum_{0\leq i\leq k-1 }{w_i} 2^i $ with $w_i\in\{0,1\}$ and $\mathbf{w}=(w_0,\ldots,w_{k-1})$ the vector of the binary digits of the integer $w$ (we will often switch from the vector representation to the integer without giving further indication.) We denote
\begin{equation}\label{tqshiftw}
    t_{q}^{(w)}(n)=\e\left((-1)^{s_{2}(w)}\dfrac{\ell}{b}s_{q}(n)\right).
\end{equation}
We observe that~\eqref{tqshiftw} is a generalization of~\eqref{gentqdef}. The factor $(-1)^{s_{2}(w)}$ appears naturally in the context of Gowers norms, as we will see later.
Let $\mathbf{r_{0}}=(r_{0,w})_{w \in \{0,\dots,2^{k}-1\}}\in \Z^{2^{k}}$. For $\rho \geq 0$, we consider 
\begin{equation}\label{Arhor0}
    A(\rho,\mathbf{r_{0}})=\dfrac{1}{q^{(k+1)\rho}}\sum_{\substack{0 \leq n<q^{\rho}\\\mathbf{h}=(h_{0},\ldots, h_{k-1}) \in \{0,\dots,q^{\rho}-1\}^{k}}}\prod_{w=0}^{2^{k}-1}t_{q}^{(w)}(n+\mathbf{w}\cdot \mathbf{h}+r_{0,w}),
\end{equation}
where $\mathbf{w}\cdot\mathbf{h}$ is the usual dot-product defined by
\begin{equation*}
    \mathbf{w}\cdot \mathbf{h}=\sum_{i=0}^{k-1}w_{i}h_{i}.
\end{equation*}

In the first step we establish a recurrence formula for $A(\rho,\mathbf{r_0})$ based on the $q$-multiplicativity of $(t_q^{(w)}(n))_{n \in \N}$. This recurrence relation  can be interpreted as a recurrence relation of the type
\begin{equation*}
X_{\rho}=MX_{\rho-1},
\end{equation*}
where $X_{\rho}$ is a vector and $M$ a  matrix such that $|M|$ is stochastic.
This will lead us to examine the underlying ``probabilistic graph'' related to $M$. For the first step, we follow the argument given by Konieczny in the proof of~\cite[Lemma 2.1, p.1902]{Kon}.

\begin{lemme}
\label{Relsure}
Let $k\geq 3$ be an integer. For any fixed $\mathbf{r_{0}} \in \mathbb{Z}^{2^k}$, the sequence $(A(\rho,\mathbf{r_{0}}))_{\rho\geq 1}$ satisfies the recursion formula
\begin{align*}
A(\rho,\mathbf{r_{0}})=&\dfrac{1}{q^{k+1}}\;\e\left(\dfrac{\ell}{b}\sum\limits_{w=0}^{2^{k}-1}(-1)^{s_{2}(w)}r_{0,w}\right)\\
&\times \sum_{\substack{{\mathbf{e}=(e_0,\ldots,e_k)\in\{0,\ldots,q-1\}^{k+1}}}}A(\rho-1,\mathbf{\delta(r_{0};e)})\;\e\left(-\dfrac{q\ell}{b}\sum\limits_{w=0}^{2^{k}-1}(-1)^{s_{2}(w)}\delta(\mathbf{r_{0}};\mathbf{e})_{w}\right),
\end{align*}
where $\mathbf{\delta(r_{0};e)}=(\mathbf{\delta(r_{0};e)}_{w})_{w\in \{0,\dots,2^{k}-1\}}$ $\in \mathbb{Z}^{2^k}$ is the vector defined by 
\begin{equation}\label{Defdelta}
\mathbf{\delta(r_{0};e)}_{w}=\Big\lfloor\dfrac{r_{0,w}+(1,\mathbf{w})\cdot \mathbf{e}}{q}\Big\rfloor,
\end{equation}
for all $w \in \{0,\ldots,2^{k}-1\}$, with the shorthand notation $(1,\mathbf{w})\cdot \mathbf{e}=e_0+\sum\limits_{i=1}^k w_{i-1} e_i$.
\end{lemme}

\begin{proof}
Let $0\leq n< q^{\rho}$ and $h=(h_{0},\ldots,h_{k-1})$ with 
$0\leq h_{i}<q^{\rho}$ for $ 0 \leq i \leq k-1$. Using the Euclidean division, there exists a unique $(2k+2)$-tuple of integers $(n',h_{0}',\ldots,h_{k-1}',e_{0}',\ldots,e_{k}')$ satisfying $0 \leq n'<q^{\rho-1}$, $0 \leq h'_i<q^{\rho-1}$ ($1 \leq i \leq k-1)$ and $0 \leq e'_i<q $ ($0 \leq i \leq k$) and such that
\begin{align*}
    n&=qn'+e_{0},\\
    h_{0}&=qh'_{0}+e_{1},\\
    &\vdots\\
    h_{k-1}&=qh'_{k-1}+e_{k}.
\end{align*}
Summing up, for all $\mathbf{w}=(w_{0},\dots,w_{k-1})$, we have
\begin{align*}
    n+\mathbf{w}\cdot\mathbf{h}=q(n'+\mathbf{w}\cdot\mathbf{h'})+(1,\mathbf{w})\cdot \mathbf{e}.
\end{align*}
By the uniqueness of the parameters involved in the Euclidean division, we then observe that
\begin{align*}
  Q&:=\{n+\mathbf{w}\cdot\mathbf{h}\,|\, w \in \{0,\ldots,2^{k}-1\}\\ 
  &=\{q(n'+\mathbf{w}\cdot \mathbf{h'})+(1,\mathbf{w})\cdot \mathbf{e}\,|\, w \in \{0,\dots,2^{k}-1\}\}  
\end{align*}
is in bijection with the cube
\begin{equation*}
Q'=\{n'+\mathbf{w}\cdot\mathbf{h'}\,|\, w \in \{0,\dots,2^{k}-1\} \}. 
\end{equation*}

By the definition of $A(\rho,\mathbf{r_{0}})$ (see~\eqref{Arhor0}),
\begin{equation}
\label{arhofor}
    A(\rho,\mathbf{r_{0}})=\dfrac{1}{q^{(k+1)\rho}}\sum_{\substack{0 \leq n'<q^{\rho-1}\\\mathbf{h'}=(h'_{0},\dots, h'_{k-1}) \in \{0,\dots,q^{\rho-1}-1\}^{k}\\
    \mathbf{e}=(e_0,\ldots,e_k)\in\{0,\ldots,q-1\}^{k+1}}}\prod_{w=0}^{2^{k}-1}t_{q}^{(w)}(q(n'+\mathbf{w}\cdot\mathbf{h'})+(1,\mathbf{w})\cdot \mathbf{e}+r_{0,w}).
\end{equation}
For all $0\leq n'<q^{\rho-1}$, $\mathbf{h'} \in \{0,\dots,q^{\rho-1}-1\}^{k}$ and $\mathbf{e} \in \{0,\dots,q-1\}^{k+1}$, we recall that
\begin{equation*}
 t_{q}^{(w)}(q(n'+\mathbf{w}\cdot \mathbf{h'})+(1,\mathbf{w})\cdot \mathbf{e}+r_{0,w})=\e\left((-1)^{s_{2}(w)}\dfrac{\ell}{b}s_{q}(q(n'+\mathbf{w}\cdot \mathbf{h'})+(1,\mathbf{w})\cdot \mathbf{e}+r_{0,w})\right).
\end{equation*}
Since for all $n \in \N$ we have
\begin{equation}
\label{EquaFonction}
    s_{q}(n)=s_{q}\left(\Big\lfloor\dfrac{n}{q}\Big \rfloor\right)+u_{q}(n),
\end{equation}
where $u_{q}(n)$ is the unit place digit of $n$ in base-$q$, 
we obtain
\begin{align}
\label{tqfor}
t_{q}^{(w)}(q(n'+\mathbf{w}\cdot \mathbf{h'})+&(1,\mathbf{w})\cdot \mathbf{e}+r_{0,w})=\\ \nonumber
&\e\left((-1)^{s_{2}(w)}\dfrac{\ell}{b}(s_{q}(n'+\mathbf{w}\cdot \mathbf{h'}+\mathbf{\delta(r_{0};e)}_{w})+u_{q}((1,\mathbf{w})\cdot \mathbf{e}+r_{0,w}))\right),
\end{align}
where
\begin{equation*}
\mathbf{\delta(r_{0};e)}_{w}=\Big\lfloor\dfrac{r_{0,w}+(1,\mathbf{w})\cdot \mathbf{e}}{q}\Big\rfloor.
\end{equation*}

Inserting the identity~(\ref{tqfor}) into~\eqref{arhofor}, we get 
\begin{align*}
    A(\rho,\mathbf{r_{0}})=\dfrac{1}{q^{(k+1)\rho}}
    \sum_{\substack{0 \leq n'<q^{\rho-1}\\\mathbf{h'}=(h'_{0},\dots, h'_{k-1}) \in \{0,\dots,q^{\rho-1}-1\}^{k}\\
    \mathbf{e}=(e_0,\ldots,e_k)\in\{0,\ldots,q-1\}^{k+1}}}&\prod_{w=0}^{2^{k}-1}\e\left((-1)^{s_{2}(w)}\dfrac{\ell}{b} u_{q}((1,\mathbf{w})\cdot \mathbf{e}+r_{0,w})\right)\\
    &\times \prod_{w=0}^{2^{k}-1} t_{q}^{(w)}(n'+\mathbf{w}\cdot\mathbf{h'}+\mathbf{\delta(r_{0};e)}_{w}).
\end{align*}
Since the first product is independent of $n'$ and $\mathbf{h'}$, it follows that
\begin{equation}\label{Afirstformula}
A(\rho,\mathbf{r_{0}})=\dfrac{1}{q^{k+1}}\sum \limits _{\mathbf{e}=(e_0,\ldots,e_k)\in\{0,\ldots,q-1\}^{k+1}}\e\left(\dfrac{\ell}{b}S(\mathbf{r_{0}},\mathbf{e})\right)A(\rho-1,\mathbf{\delta(r_{0};e)}),
\end{equation}
where
\begin{equation*}
    S(\mathbf{r_{0}},\mathbf{e})=\sum_{w=0}^{2^{k}-1}(-1)^{s_{2}(w)}u_{q}((1,\mathbf{w})\cdot\mathbf{e}+r_{0,w}).
\end{equation*}
We note that~Lemma~\ref{recsrho} provides the same recursion formula \eqref{Afirstformula} when we handle $\som{\rho}$ instead of $s_q$.
Until the end of the proof, we will only consider $s_q$ knowing that we can replace any occurrence of $s_q$ by $\som{\rho}$.
\\We now rewrite the sum $S(\mathbf{r_0},\mathbf{e})$ in a more convenient way. By~\eqref{Defdelta},
\begin{equation*}
    u_{q}((1,\mathbf{w})\cdot \mathbf{e}+r_{0,w})=(1,\mathbf{w})\cdot \mathbf{e}+r_{0,w}-q\mathbf{\delta(r_{0};e)}_{w}
\end{equation*}
and therefore
\begin{align*}
    S(\mathbf{r_{0}},\mathbf{e})&=\sum_{w=0}^{2^{k}-1}(-1)^{s_{2}(w)}((1,\mathbf{w})\cdot \mathbf{e}+r_{0,w}-q\mathbf{\delta(r_{0};e)}_{w})\\
    &=\sum\limits_{w=0}^{2^{k}-1}(-1)^{s_{2}(w)}(r_{0,w}-q\mathbf{\delta(r_{0};e)}_{w})-\sum_{\substack{w=0\\2 \nmid s_{2}(w)}}^{2^{k}-1}(1,\mathbf{w})\cdot \mathbf{e}+\sum_{\substack{w=0\\2\mid s_{2}(w)}}^{2^{k}-1}(1,\mathbf{w})\cdot \mathbf{e}.
\end{align*}
We claim that
\begin{equation*}
 \sum_{\substack{w=0\\2\mid s_{2}(w)}}^{2^{k}-1}(1,\mathbf{w})\cdot \mathbf{e}=\sum_{\substack{w=0\\2\nmid s_{2}(w)}}^{2^{k}-1}(1,\mathbf{w})\cdot \mathbf{e}.
\end{equation*}
To see this, we write
\begin{align}
    \sum_{\substack{w=0\\2\nmid s_{2}(w)}}^{2^{k}-1}(1,\mathbf{w})\cdot \mathbf{e}&=e_{0}\,|\{0\leq w<2^{k} 
 \,:\, 2\nmid s_{2}(w)\}|+\sum_{\substack{\mathbf{w}=(w_{0},\ldots,w_{k-1})\in \{0,1\}^{k}\\2\nmid\sum\limits_{i=0}^{k-1}  w_{i}}}\;\sum_{r=1}^{k}w_{r-1}e_{r}\nonumber\\
    &=e_{0}\,|\{0\leq w<2^{k}\,:\, 2\nmid s_{2}(w)\}|+\sum_{r=1}^{k}e_{r}\sum_{\substack{\mathbf{w}=(w_{0},\ldots,w_{k-1})\in \{0,1\}^{k}\\2\nmid \sum\limits_{i=0}^{k-1} w_{i}}}w_{r-1}\nonumber\\
    &=e_{0}\,|\{0\leq w<2^{k}\,:\, 2\nmid s_{2}(w)\}|+\sum_{r=1}^{k}e_{r}\sum_{\substack{w_{0},\ldots,w_{k-1}\in \{0,1\}\\w_{r-1}=1\\2\mid \sum\limits_{i \neq r-1} w_{i}}}1.\label{sumsum1}
\end{align}
Analogously, we have
\begin{equation}\label{sumsum2}
    \sum_{\substack{w=0\\2\mid s_{2}(w)}}^{2^{k}-1}(1,\mathbf{w})\cdot \mathbf{e}=e_{0}|\;\{0\leq w<2^{k}\,:\, 2\mid s_{2}(w)\}|+\sum_{r=1}^{k}e_{r}\sum_{\substack{w_{0},\ldots,w_{k-1}\in \{0,1\}\\w_{r-1}=1\\2\nmid\sum\limits_{i \neq r-1} w_{i}}}1.
\end{equation}
We notice that the sets $\{0\leq w<2^{k}\,:\, 2\mid s_{2}(w)\}$ and $\{0\leq w<2^{k}\,:\,2\nmid s_{2}(w)\}$ are in bijection. Indeed, when we switch one fixed binary digit of $w$, we pass from one set to the other (or, equivalently, we can use $\sum_{0\leq j<2^{k}}(-1)^{s_{2}(j)}=0$).
A similar digit switching argument (noticing that $k\geq 2$) 
works to show that for fixed $r$ both inner sums in~\eqref{sumsum1} and~\eqref{sumsum2} are equal.

We therefore get
\begin{equation*}
 \sum_{\substack{w=0\\2\mid s_{2}(w)}}^{2^{k}-1}(1,\mathbf{w})\cdot \mathbf{e}=\sum_{\substack{w=0\\2\nmid s_{2}(w)}}^{2^{k}-1}(1,\mathbf{w})\cdot \mathbf{e}
\end{equation*}
and
\begin{equation*}
    S(\mathbf{r_{0}},\mathbf{e})=\sum\limits_{w=0}^{2^{k}-1}(-1)^{s_{2}(w)}(r_{0,w}-q\mathbf{\delta(r_{0};e)}_{w}).
\end{equation*}
Plugging this into~\eqref{Afirstformula} we get
the stated recursion formula.
\end{proof}
\begin{rem}
For the case $b=q=2$, which corresponds to the Thue--Morse sequence, the recursion formula in Lemma \ref{Relsure} becomes
\begin{equation*}
    A(\rho,\mathbf{r_{0}})=\dfrac{(-1)^{|\mathbf{r_{0}}|}}{2^{k+1}}\sum_{\substack{{\mathbf{e}=(e_0,\ldots,e_{k-1})\in\{0,1\}^{k+1}}}}A(\rho-1,\mathbf{\delta(r_{0};e)}),
\end{equation*}
where $|\mathbf{r_0}|:=\sum_{w \in \{0,\dots,2^{k-1}\}}r_{0,w}$.
This is the same formula as in the article of Konieczny~\cite[(2.1)]{Kon} and of Spiegelhofer~\cite[(5.13)]{Spi}. We note that in our more general case, the corresponding formula is more involved since there are additional multiplicative weights ($b$th roots of unity) attached to the sum and the arguments. Somewhat surprisingly, these $b$th roots of unity are connected to the binary digit sum of the summation index $w$.
\end{rem}
In order to have a summation ranging over the whole set $\mathbb{Z}^{2^k}$, we
gather the $\mathbf{e} \in \{0,\dots,q-1\}^{k+1}$ such that $\mathbf{r_{1}}=\mathbf{\delta(r_{0};e)}$. The formula in Lemma~\ref{Relsure} becomes a recursion with normalized weights attached to its terms.
\begin{lemme}
\label{Relsurr}
For $\mathbf{r_{0}} \in \mathbb{Z}^{2^k}$, the sequence $(A(\rho,\mathbf{r_{0}}))_{\rho>0}$ satisfies the recursion formula
\begin{equation}\label{recforumlafinal}
    A(\rho,\mathbf{r_{0}})=\sum\limits_{\mathbf{r_{1}} \in \mathbb{Z}^{2^k}}A(\rho-1,\mathbf{r_{1}})p_{\mathbf{r_{0}},\mathbf{r_{1}}},
\end{equation}
where

\begin{equation}
\label{Ptrans}
p_{\mathbf{r_{0}},\mathbf{r_{1}}}=\e\left(\dfrac{\ell}{b}\left(q\Tilde{S}(\mathbf{r_{1}})-\Tilde{S}(\mathbf{r_{0}})\right)\right)\dfrac{|\{\mathbf{e} \in \{0,\ldots,q-1\}^{k+1}\,:\, \mathbf{\delta(r_{0};e)}=\mathbf{r_{1}}\}|}{q^{k+1}},
\end{equation}
with
\begin{equation*}
    \Tilde{S}(\mathbf{r})=-\sum\limits_{w=0}^{2^{k}-1}(-1)^{s_{2}(w)}r_{w}
\end{equation*}
for 
\begin{equation*}
    \mathbf{r}=(r_{w})_{w \in \{0,\ldots,2^{k}-1\}} \in \mathbb{Z}^{2^k} .
\end{equation*}
\end{lemme}

The vector $\mathbf{r}=(r_{w})_{w \in \{0,1\}^{k}}$ such that $r_{w}=0$ for all $w \in \{0,\ldots,2^{k}-1\}$ will be written as  $$\mathbf{0}:=(0,0,\ldots,0)\in\mathbb{Z}^{2^k}.$$
The recursion formula~\eqref{recforumlafinal} and the weights that can be seen as transition complex random variables, suggests to consider a directed graph modelled in the set $\Z^{2^{k}}$ consisting of the vectors $\mathbf{r}$ as vertices and being connected to each other whenever a specific condition on the associated weights is verified: there is an edge between $\mathbf{r_{0}}$ and $\mathbf{r_{1}}$ if and only if there exists $\mathbf{e} \in \{0,\dots,q-1\}^{k+1}$ such that
\begin{equation*}
 \mathbf{\delta(r_{0};e)}=\mathbf{r_{1}}.   
\end{equation*}
A vertex $\mathbf{r_{1}}=(r_{1,w})_{w}$ is said to be reachable from a vertex $\mathbf{r_{0}}$ if there exists $d>0$ and $\mathbf{r^{(1)}},\dots,\mathbf{r^{(d)}} \in \Z^{2^{k}}$ such that
\begin{equation*}
    p_{\mathbf{r_{0}},\mathbf{r^{(1)}}}\times p_{\mathbf{r^{(1)}},\mathbf{r^{(2)}}}\times \cdots \times p_{\mathbf{r^{(d-1)}},\mathbf{r^{(d)}}}\times p_{\mathbf{r^{(d)}},\mathbf{r_{1}}} \neq 0.
\end{equation*}
We write
\begin{equation*}
    \mathbf{r_{0}} \rightarrow \mathbf{r^{(1)}} \rightarrow \mathbf{r^{(2)}} \rightarrow \cdots \rightarrow \mathbf{r^{(d)}} \rightarrow \mathbf{r_{1}},
\end{equation*}
whenever $\mathbf{r_{1}}$ is reachable from $\mathbf{r_{0}}$ via $\mathbf{r^{(1)}}, \mathbf{r^{(2)}} ,\dots, \mathbf{r^{(d)}}$. We now turn our attention to the graph built by the elements $\mathbf{r_{1}} \in \Z^{2^{k}}$ reachable from $\mathbf{0}$. The following result shows that this is a finite subgraph in above graph.
\begin{lemme}
\label{GFini}
Let $\mathfrak{R}$ be the set of elements $\mathbf{r} \in \mathbb{Z}^{2^k}$ reachable from $\mathbf{0}$. Then $\mathfrak{R}$ is finite and for $\mathbf{r}=(r_{w})_{w \in \{0,\dots,2^{k}-1\}}$ reachable from $\mathbf{0}$ and for $0 \leq w<2^{k}$ we have
\begin{equation*}
    0 \leq r_{w} \leq s_{2}(w).
\end{equation*}
 Furthermore,
\begin{equation*}
    |\mathfrak{R}|\leq \prod_{i=0}^{k}(i+1)^{\binom{k}{i}}.
\end{equation*}
\end{lemme}
\begin{proof}
We suppose that
\begin{equation*}
    \mathbf{0} \rightarrow \mathbf{r^{(1)}} \rightarrow \mathbf{r^{(2)}} \rightarrow \cdots \rightarrow \mathbf{r^{(d)}} \rightarrow \mathbf{r_{1}}.
\end{equation*}
Since $p_{\mathbf{r^{(d)}},\mathbf{r_{1}}}\neq 0$,  there exists an element $\mathbf{e^{(d)}} \in \{0,\dots,q-1\}^{k+1}$ such that for all $w \in \{0,\dots,2^{k}-1\},$ 
\begin{equation*}
\delta(\mathbf{r^{(d)}};\mathbf{e^{(d)}})_{w}
=\Bigg\lfloor \dfrac{r^{(d)}_{w}+(1,\mathbf{w})\cdot \mathbf{e^{(d)}}}{q}\Bigg\rfloor=r_{1,w}.
\end{equation*}
 In the same way, when $d \geq 2$ and if $p_{\mathbf{r^{(d-1)}},\mathbf{r^{(d)}}}\neq 0$, we deduce the existence of $\mathbf{e^{(d-1)}} \in \{0,\dots,q-1\}^{k+1}$ such that
\begin{equation*}
r_{1,w}=\Bigg\lfloor \dfrac{ \Big\lfloor \dfrac{r^{(d-1)}_{w}+(1,\mathbf{w})\cdot \mathbf{e^{(d-1)}}}{q}\Big\rfloor+(1,\mathbf{w})\cdot\mathbf{e^{(d)}}}{q}\Bigg\rfloor.
\end{equation*}
For all $\mathbf{e} \in \{0,\dots,q-1\}^{k+1}$ and all $0\leq w<2^{k}$ we have the simple bound
\begin{equation*}
(1,\mathbf{w})\cdot \mathbf{e} \leq (q-1)(1+s_{2}(w)).
\end{equation*}
Iterating the process above, we will reach the point $\mathbf{0}$, after a finite number of steps. We therefore have the bound
\begin{equation*}
    r_{1,w} < (q-1)(1+s_{2}(w))\sum_{r=1}^{\infty}\dfrac{1}{q^{r}}=1+s_2(w)
\end{equation*}
and we get $r_{1,w}\leq s_2(w)$. Moreover, we notice that for all $w \in \{0,\dots,2^{k}-1\}$, we have
\begin{equation*}
    r^{(1)}_{w}=\Big\lfloor\dfrac{(1,\mathbf{w})\cdot \mathbf{e^{(0)}}}{q}\Big\rfloor \geq 0,
\end{equation*}
and therefore $r^{(2)}_{w}, \ldots, r^{(d)}_{w}, r_{1,w} \geq 0$ for all $w \in \{0,\dots,2^{k}-1\}$. This implies that the set $\mathfrak{R}$ of the vertices reachable from $\mathbf{0}$ is finite and for all $\mathbf{r}=(r_{w})_{w} \in \mathfrak{R}$ we have
\begin{equation*}
    0 \leq r_{w} \leq s_{2}(w).
\end{equation*}
For the cardinality of $\mathfrak{R}$, we observe that 
for any $\mathbf{r}=(r_{w})_{w \in \{0,\dots,2^{k}-1\}} \in \mathfrak{R}$ and any fixed $ 0\leq i \leq k$, there are $\binom{k}{i}$ choices of elements $w$ such that $s_{2}(w)=i$. Thus, there are $(i+1)^{\binom{k}{i}}$ possible choices for the coordinates of $\mathbf{r}$ indexed by a number $w$ such that $s_{2}(w)=i$, leading to the stated bound. 
\end{proof}
In addition to the finiteness, our graph is strongly connected. To prove this, we introduce a new notation.
\\For $\mathbf{e} \in \{0,\dots,q-1\}^{k+1}$, we define the function $\delta_{\mathbf{e}}\,:\, \Z^{2^{k}}\rightarrow \Z^{2^{k}}$ by
\begin{equation*}
    \delta_{\mathbf{e}}(\mathbf{r})=\delta(\mathbf{r};\mathbf{e}).
\end{equation*}
\begin{lemme}
\label{FreeFall}
For $K:=\left\lfloor\dfrac{\log(k)}{\log(q)}\right\rfloor+1$ we have for all $\mathbf{r} \in \mathfrak{R}$:
\begin{equation*}
\underbrace{(\delta_{\mathbf{0}} \circ \delta_{\mathbf{0}} \circ \dots \circ \delta_{\mathbf{0}})}_{K\, \text{times}}(\mathbf{r})=\mathbf{0}.
\end{equation*}
\end{lemme}
\begin{proof}
Let $\mathbf{r} \in \mathfrak{R}$ and $w \in \{0,\dots,2^{k}-1\}$. Then
$$
    0\leq \mathbf{\delta_{0}(\mathbf{r})}_{w}=\Big\lfloor\dfrac{r_{w}}{q}\Big\rfloor
    \leq \dfrac{r_{w}}{q}.
$$
Thus,
\begin{equation*}
    0\leq \mathbf{\delta_{0}}(\mathbf{\delta_{0}}(\mathbf{r}))_{w} \leq \dfrac{r_{w}}{q^{2}}.
\end{equation*}
By induction on the number of iterations of $\mathbf{\delta_{0}}$, and using the fact that $r_{w} \leq s_{2}(w) \leq k$, we have
\begin{equation*}
    0 \leq \underbrace{(\delta_{\mathbf{0}} \circ \delta_{\mathbf{0}} \circ \dots \circ \delta_{\mathbf{0}})}_{K\, \text{times}}(\mathbf{r})_{w} \leq \dfrac{k}{q^{K}}.
\end{equation*}
Since $k<q^{K}$ and $\underbrace{(\delta_{\mathbf{0}} \circ \delta_{\mathbf{0}} \circ \dots \circ \delta_{\mathbf{0}})}_{K\, \text{times}}(\mathbf{r})_{w} \in \N$, we have
$
    \underbrace{(\delta_{\mathbf{0}} \circ \delta_{\mathbf{0}} \circ \dots \circ \delta_{\mathbf{0}})}_{K\, \text{times}}(\mathbf{r})_{w}=0.
$
\end{proof}
Lemma~\ref{FreeFall} assures us that after applying $K$ times the operator $\mathbf{\delta_{0}}$  to any vertex of the graph, we inevitably end up at vertex $\mathbf{0}$ (this is sometimes also described as the property that the graph is \textit{synchronizing}).
\begin{prop}
\label{ForteCo}
Let $\mathbf{r_{0}},\mathbf{r_{1}} \in \mathfrak{R}$. Then, $\mathbf{r_{1}}$ is reachable from $\mathbf{r_{0}}$. 
\end{prop}
\begin{proof}
Lemma~\ref{FreeFall} allows us to reach $\mathbf{0}$ in a finite number of steps. Then, the definition of the graph enables us to reach $\mathbf{r_{1}}$ in a finite number of steps from $\mathbf{0}$. We can thus concatenate the two previous paths and connect $\mathbf{r_{0}}$ to $\mathbf{r_{1}}$.
\end{proof}
We can now interpret  formula~\eqref{recforumlafinal} in Lemma~\ref{Relsurr} in a recursive way in our graph setting. For $1\leq j\leq \rho$, and for $\mathbf{r_{0}} \in \mathfrak{R}$, we have
\begin{equation}
    A(\rho,\mathbf{r_{0}})=\sum\limits_{\mathbf{r_{j}} \in \mathfrak{R}}A(\rho-j,\mathbf{r_{j}})p^{(j)}_{\mathbf{r_{0}},\mathbf{r_{j}}},
\end{equation}
where $p^{(j)}_{\mathbf{r_{0}},\mathbf{r_{j}}}$ denotes the sum of all weights of paths connecting $\mathbf{r_{0}}$ to $\mathbf{r_{j}}$ of length $j$. Recall that the weight $p(\gamma)$ of a path $\gamma$ connecting $\mathbf{r_{0}}$ to $\mathbf{r_{j}}$,
\begin{equation*}
 \gamma\ :\ \mathbf{r_{0}} \rightarrow \mathbf{r^{(1)}} \rightarrow \mathbf{r^{(2)}} \rightarrow \cdots \rightarrow \mathbf{r^{(d)}} \rightarrow \mathbf{r_{j}}
\end{equation*}
is defined by
\begin{equation*}
    p(\gamma)=p_{\mathbf{r_{0}},\mathbf{r^{(1)}}}\times p_{\mathbf{r^{(1)}},\mathbf{r^{(2)}}}\times \cdots \times p_{\mathbf{r^{(d-1)}},\mathbf{r^{(d)}}}\times p_{\mathbf{r^{(d)}},\mathbf{r_{j}}}.
\end{equation*}
Since $\mathfrak{R}$ is finite, the quantity $v_{\rho}=\max\limits_{\mathbf{r}\in \mathfrak{R}}(|A(\rho,\mathbf{r})|)<+\infty$ is well-defined. We obtain, for all $1\leq j\leq \rho$, 
\begin{equation}
\label{veroinneq}
v_{\rho} \leq v_{\rho-j}\max\limits_{\mathbf{r_{0}}\in \mathfrak{R}}\left(\sum\limits_{\mathbf{r_{j}}\in \mathfrak{R}}|p^{(j)}_{\mathbf{r_{0}},\mathbf{r_{j}}}|\right).
\end{equation}
In order to proceed with the proof of Theorem~\ref{NDGG}, we will show in the next section that there exists an integer $j \geq 1$ such that
\begin{equation}
\label{hyp}
\sum\limits_{\mathbf{r_{j}}\in \mathfrak{R}}|p^{(j)}_{\mathbf{r_{0}},\mathbf{r_{j}}}|<1
\end{equation}
for all $\mathbf{r_{0}} \in \mathfrak{R}$.

\subsection{An effective value of the exponent $\eta_0$ in Theorem~\ref{NDGG}}\label{effectiveeta0}

The aim of this section is to prove~(\ref{hyp}), and, more precisely, to provide an effective value for $\eta_{0}$. This allows us to get the effective version of Theorem~\ref{NDGG} with an explicit value for $\eta_{0}$ in the end. 
We will show that there exists $j \geq 1$ such that  
\begin{equation*}
M=\max\limits_{\mathbf{r_{0}}\in \mathfrak{R}}\left(\sum\limits_{\mathbf{r_{j}}\in \mathfrak{R}}|p^{(j)}_{\mathbf{r_{0}},\mathbf{r_{j}}}|\right)<1
\end{equation*}
with an effective bound on $M$ by a constant $<1$. Let $k$ and $j$ be two integers. Let $K$ be as defined in Lemma~\ref{FreeFall}. We suppose that $j \geq K+b(k+1)$.
We use two paths connecting $\mathbf{0}$ to $\mathbf{0}$ of length $(k+1)$ with different weights to get a saving.

\medskip

 We recall that a path of length $(k+1)$ connecting a point $\mathbf{r^{(0)}}$ to $\mathbf{r^{(k+1)}}$ by means of $\delta$ (which is defined in~(\ref{Defdelta})) is specified by $(k+1)$ vectors $\mathbf{e^{(0)}},\ldots,\mathbf{e^{(k)}}\in \{0,\ldots,q-1\}^{k+1}$ and $k$ vectors $\mathbf{r^{(1)}},\ldots,\mathbf{r^{(k)}}\in  \mathbb{Z}^{2^k}$ such that $\mathbf{r^{(1)}}=\delta(\mathbf{r^{0}};\mathbf{e^{(0)}})$, $\dots$, $\mathbf{r^{(k+1)}}=\delta(\mathbf{r^{(k)}};\mathbf{e^{(k)}})$.
This could be written as follows
\begin{equation*}
\mathbf{r^{(0)}} \xrightarrow[]{\mathbf{e^{(0)}}} \mathbf{r^{(1)}} \xrightarrow[]{\mathbf{e^{(1)}}} \cdots \xrightarrow[]{\mathbf{e^{(k-2)}}}\mathbf{r^{(k-1)}} \xrightarrow[]{\mathbf{e^{(k-1)}}} \mathbf{r^{(k)}}\xrightarrow[]{\mathbf{e^{(k)}}} \mathbf{r^{(k+1)}}.
\end{equation*}

\subsubsection{Trivial path and Konieczny's path}

As for the \textit{first path (trivial path)}, we take $(k+1)$ (\textit{little}) loops at $\mathbf{0}$, i.e. at each step $1 \leq t \leq k+1$ we choose $\mathbf{e^{(t)}}=\mathbf{0}$,

\begin{equation}\label{0path}
\mathbf{0} \xrightarrow[]{\mathbf{0}} \mathbf{0} \xrightarrow[]{\mathbf{0}} \cdots\xrightarrow[]{\mathbf{0}} \mathbf{0} \xrightarrow[]{\mathbf{0}} \mathbf{0}\xrightarrow[]{\mathbf{0}} \mathbf{0}.
\end{equation}
For the \textit{second path (Konieczny's path)}, we proceed via a \textit{great} loop. For that purpose, we adapt the construction given in the proof of Proposition 2.3 by Konieczny~\cite{Kon}. We set $\mathbf{r^{(0)}}=\mathbf{0}=\mathbf{r^{(k+1)}}$. For $1 \leq t \leq k$, we define $\mathbf{r^{(t)}}=(r^{(t)}_w)_w$ by
\begin{equation*}
\forall w \in \{0,\dots,2^{k}-1\},\qquad 
r^{(t)}_{w}=
\begin{cases}
    1, & \text{if }\,  w_{0}=\cdots=w_{t-1}=1; \\
    0, & \text{otherwise.}\,
\end{cases}
\end{equation*}
\underline{Case $t=0$}. We take $\mathbf{e_{Ko}^{(0)}}=(q-1,1,0,\dots,0)$. We have
\begin{equation*}
\mathbf{\delta(r^{(0)};e_{Ko}^{(0)})}=\mathbf{r^{(1)}}\\.
\end{equation*}
To prove this, we observe that for $w \in \{0,\dots,2^{k}-1\}$ we have
\begin{equation*}
 \delta(\mathbf{r^{(0)}};\mathbf{e_{Ko}^{(0)}})_{w}=\Big\lfloor\dfrac{0+q-1+w_{0}}{q}\Big\rfloor.
\end{equation*}
This is equal to $1$ if $w_{0}=1$, and $0$ if $w_{0}=0$. Hence, in this case,
\begin{equation*}
    \delta(\mathbf{r^{(0)};e_{Ko}^{(0)}})_{w}=r^{(1)}_{w}.
\end{equation*}
The transition weight from $\mathbf{r_{0}}$ to $\mathbf{r_{1}}$ is given by
\begin{equation*}
\e\left(\dfrac{\ell}{b}\left(q\Tilde{S}(\mathbf{r^{(1)}})-\Tilde{S}(\mathbf{r^{(0)}})\right)\right), 
\end{equation*}
where
$$
    \Tilde{S}(\mathbf{r^{(1)}})=-\sum_{w=0}^{2^{k}-1}(-1)^{s_{2}(w)}r_{w}^{(1)}=-\sum_{\substack{w=0\\w_{0}=1}}^{2^{k}-1}(-1)^{s_{2}(w)}
    =0$$
and
$$
    \Tilde{S}(\mathbf{r^{(0)}})=-\sum_{w=0}^{2^{k}-1}(-1)^{s_{2}(w)}r_{w}^{(0)}
    =0.$$
\underline{Case $1 \leq t <k-1$.} We choose $\mathbf{e_{Ko}^{(t)}}=(q-2,0,\dots,0,1,0,\dots,0)$, where $1$ is at the $(t+1)$ position.
For $w \in \{0,\dots,2^{k}-1\}$, we have 
$$
\delta(\mathbf{r^{(t)};e_{Ko}^{(t)}})_{w}=\Big\lfloor \dfrac{r^{(t)}_{w}+(1,\mathbf{w})\cdot \mathbf{e_{Ko}^{(t)}}}{q}\Big\rfloor
=\Big\lfloor\dfrac{r^{(t)}_{w}+q-2+w_{t}}{q}\Big\rfloor.
$$
Thus, $\delta(\mathbf{r^{(t)};e_{Ko}^{(t)}})_{w}=1$ if and only if $r_{w}^{(t)}=1$ and $w_{t}=1$. By definition of $r_{w}^{(t)}$,
\begin{equation*}
  \delta(\mathbf{r^{(t)};e_{Ko}^{(t)}})_{w}=1 \iff w_{0}=\dots=w_{t}=1. 
\end{equation*}
We deduce that
\begin{equation*}
    \mathbf{\delta(\mathbf{r^{(t)};e_{Ko}^{(t)}})}=\mathbf{r^{(t+1)}}.
\end{equation*}
According to~(\ref{Ptrans}), the weight attached to the transition from $\mathbf{r^{(t)}}$ to $\mathbf{r^{(t+1)}}$ is given by
\begin{equation*}
\e\left(\dfrac{\ell}{b}\left(q\Tilde{S}(\mathbf{r^{(t+1)}})-\Tilde{S}(\mathbf{r^{(t)}})\right)\right), 
\end{equation*}
where
$$
    \Tilde{S}(\mathbf{r^{(t+1)}})=-\sum_{w=0}^{2^{k}-1}(-1)^{s_{2}(w)}r_{w}^{(t+1)}=-\sum_{\substack{w=0\\w_{0}=\cdots=w_{t}=1}}^{2^{k}-1}(-1)^{s_{2}(w)}
    =0
$$
and
$$
    \Tilde{S}(\mathbf{r^{(t)}})=-\sum_{w=0}^{2^{k}-1}(-1)^{s_{2}(w)}r_{w}^{(t)}=-\sum_{\substack{w=0\\w_{0}=\cdots=w_{t-1}=1}}^{2^{k}-1}(-1)^{s_{2}(w)}
    =0.$$
Hence,
\begin{equation*}
\e\left(\dfrac{\ell}{b}\left(q\Tilde{S}(\mathbf{r^{(t+1)}})-\Tilde{S}(\mathbf{r^{(t)}})\right)\right)=1.
\end{equation*}
\underline{Let $t=k-1$.} We set $\mathbf{e_{Ko}^{(k-1)}}=(q-2,0,\dots,0,0,1)$. 

For $w \in \{0,\dots,2^{k}-1\}$, we have 
$$
\delta(\mathbf{r^{(t)};e_{Ko}^{(t)}})_{w}=\Big\lfloor \dfrac{r^{(k-1)}_{w}+(1,\mathbf{w})\cdot \mathbf{e_{Ko}^{(k-1)}}}{q}\Big\rfloor
=\Big\lfloor\dfrac{r^{(k-1)}_{w}+q-2+w_{k-1}}{q}\Big\rfloor.$$
Thus, $\delta(\mathbf{r^{(k-1)};e_{Ko}^{(k-1)}})_{w}=1$ if and only if $r_{w}^{(k-1)}=1$ and $w_{k-1}=1$. By definition of $r_{w}^{(k-1)}$ we have
\begin{equation*}
  \delta(\mathbf{r^{(k-1)};e_{Ko}^{(k-1)}})_{w}=1 \quad \iff \quad w_{0}=\dots=w_{k-1}=1. 
\end{equation*}
We deduce that
\begin{equation*}
    \mathbf{\delta(\mathbf{r^{(k-1)};e_{Ko}^{(k-1)}})}=\mathbf{r^{(k)}}.
\end{equation*}
Here, we have a change regarding the transition weights. 
While

$$
    \Tilde{S}(\mathbf{r^{(k-1)}})=-\sum_{w=0}^{2^{k}-1}(-1)^{s_{2}(w)}r_{w}^{(t)}=\sum_{\substack{w=0\\w_{0}=\cdots=w_{k-2}=1}}^{2^{k}-1}(-1)^{s_{2}(w)}
    =0$$
we here have
$$
    \Tilde{S}(\mathbf{r^{(k)}})=-\sum_{w=0}^{2^{k}-1}(-1)^{s_{2}(w)}r_{w}^{(t)}=-\sum_{\substack{w=0\\w_{0}=\cdots=w_{k-1}=1}}^{2^{k}-1}(-1)^{s_{2}(w)}
    =-(-1)^{k}
    =(-1)^{k+1}.$$
Thus, the weight attached to the transition from $\mathbf{r^{(k-1)}}$ to $\mathbf{r^{(k)}}$ equals
\begin{equation*}
\e\left(\dfrac{\ell}{b}\left(q\Tilde{S}(\mathbf{r^{(k)}})-\Tilde{S}(\mathbf{r^{(k-1)}})\right)\right)=\e\left((-1)^{k+1}\dfrac{q\ell}{b}\right).
\end{equation*}
\\\underline{Let $t=k$.} Here we set $\mathbf{e_{Ko}^{(k)}}=\mathbf{0}$. Thus,
$$
\delta(\mathbf{r^{(k)}};\mathbf{e_{Ko}^{(k)}})=\mathbf{0}=\mathbf{r^{(k+1)}}.$$
The weight attached to the transition from $\mathbf{r^{(k)}}$ to $\mathbf{r^{(k+1)}}$ is given by
\begin{equation*}
\e\left(\dfrac{\ell}{b}\left(q\Tilde{S}(\mathbf{r^{(k+1)}})-\Tilde{S}(\mathbf{r^{(k)}})\right)\right)=\e\left((-1)^{k}\dfrac{\ell}{b}\right).
\end{equation*}
Note that the weights $\tilde{S}$ only depend on $\mathbf{r^{(t)}}$ (not on ${\bf e}$). We therefore have 
\begin{align}
    p_{\mathbf{r^{(t)}},\mathbf{r^{(t+1)}}}&=\dfrac{|\{\mathbf{e} \in \{0,\dots,q-1\}^{k+1}\,:\, \mathbf{\delta(r^{(t)};e)}=\mathbf{r^{(t+1)}}\}|}{q^{k+1}},\qquad\qquad 0 \leq t<k,\nonumber\\ 
\label{Calcpoids}
p_{\mathbf{r^{(k-1)}},\mathbf{r^{(k)}}}&=\e\left((-1)^{k+1}\dfrac{q\ell}{b}\right)\dfrac{|\{\mathbf{e} \in \{0,\dots,q-1\}^{k+1}\,:\, \mathbf{\delta(r^{(k-1)};e)}=\mathbf{r^{(k)}}\}|}{q^{k+1}},\\
p_{\mathbf{r^{(k)}},\mathbf{r^{(k+1)}}}&=\e\left((-1)^{k}\dfrac{\ell}{b}\right)\dfrac{|\{\mathbf{e} \in \{0,\dots,q-1\}^{k+1}\,:\, \mathbf{\delta(r^{(k)};e)}=\mathbf{r^{(k+1)}}\}|}{q^{k+1}}.\nonumber
\end{align}

\subsubsection{Coding the paths by words}

Let $\mathbf{r_{0}},\mathbf{r_{j}} \in \mathfrak{R}$. Consider the set of all paths of length $j$ connecting $\mathbf{r_{0}}$ to $\mathbf{r_{j}}$. Proposition~\ref{ForteCo} shows that this set is non-empty. A path $\gamma$ belonging to this set can be described by words $\omega_{\gamma}$ of length $j$ over the alphabet $\{0,1,\ldots, q^{k+1}-1\}$ where we identify $\mathbf{e^{(i)}}\in \{0,\ldots,q-1\}^{k+1}$ to the letter $\sum_{j=0}^{k+1} e^{(i)}_j q^j$. Note that there might be several different words that describe the same path. 
We illustrate the word by
\begin{center}
\begin{tikzpicture}
    \node at (0.55,0.55){$\omega_{\gamma}\,:\,$};
\end{tikzpicture}
\begin{tikzpicture}
    \draw (0,0) rectangle (1.1,1.1); 
    \node  at (0.55,0.55) {$\mathbf{e^{(0)}}$};
\end{tikzpicture}
\begin{tikzpicture}
    \draw (0,0) rectangle (1.1,1.1); 
    \node  at (0.55,0.55) {$\mathbf{e^{(1)}}$};
\end{tikzpicture}
\begin{tikzpicture} 
    \node at (0.55,0.55) {$\cdots$};
\end{tikzpicture}
\begin{tikzpicture}
    \draw (0,0) rectangle (1.1,1.1); 
    \node  at (0.55,0.55) {$\mathbf{e^{(j)}}$};
\end{tikzpicture}
\end{center}
which relates to the path
\begin{equation*}
\mathbf{r_{0}} \longrightarrow \mathbf{r_{1}} \longrightarrow \dots \longrightarrow \mathbf{r_{j-1}}\longrightarrow \mathbf{r_{j}},
\end{equation*}
via $$\mathbf{r_{1}}=\mathbf{\delta(r_{0};e^{(0)})}, \quad \dots,\quad  \mathbf{r_{j}}=\mathbf{\delta(r_{j-1};e^{(j-1)})}.$$ We now define the concatenation of two paths.
\begin{de}
Let $\omega$ and $\omega'$ be two words of length $\alpha$ and $\beta$, respectively.
\begin{center}
\begin{tikzpicture}
    \node at (0.55,0.55){$\omega\,:\,$};
\end{tikzpicture}
\begin{tikzpicture}
    \draw (0,0) rectangle (1.1,1.1); 
    \node  at (0.55,0.55) {$\mathbf{e^{(0)}}$};
\end{tikzpicture}
\begin{tikzpicture}
    \draw (0,0) rectangle (1.1,1.1); 
    \node  at (0.55,0.55) {$\mathbf{e^{(1)}}$};
\end{tikzpicture}
\begin{tikzpicture} 
    \node at (0.55,0.55) {$\cdots$};
\end{tikzpicture}
\begin{tikzpicture}
    \draw (0,0) rectangle (1.1,1.1); 
    \node  at (0.55,0.55) {$\mathbf{e^{(\alpha)}}$};
\end{tikzpicture}
\end{center}
\begin{center}
\begin{tikzpicture}
    \node at (0.55,0.55){$\omega'\,:\,$};
\end{tikzpicture}
\begin{tikzpicture}
    \draw (0,0) rectangle (1.1,1.1); 
    \node  at (0.55,0.55) {$\mathbf{e'^{(0)}}$};
\end{tikzpicture}
\begin{tikzpicture}
    \draw (0,0) rectangle (1.1,1.1); 
    \node  at (0.55,0.55) {$\mathbf{e'^{(1)}}$};
\end{tikzpicture}
\begin{tikzpicture} 
    \node at (0.55,0.55) {$\cdots$};
\end{tikzpicture}
\begin{tikzpicture}
    \draw (0,0) rectangle (1.1,1.1); 
    \node  at (0.55,0.55) {$\mathbf{e'^{(\beta)}}$};
\end{tikzpicture}
\end{center}
We define $\omega \circ \omega'$ by
\begin{center}
\begin{tikzpicture}
    \node at (0.55,0.55){$\omega\circ \omega'\,:\,$};
\end{tikzpicture}
\begin{tikzpicture}
    \draw (0,0) rectangle (1.1,1.1); 
    \node  at (0.55,0.55) {$\mathbf{e^{(0)}}$};
\end{tikzpicture}
\begin{tikzpicture}
    \draw (0,0) rectangle (1.1,1.1); 
    \node  at (0.55,0.55) {$\mathbf{e^{(1)}}$};
\end{tikzpicture}
\begin{tikzpicture} 
    \node at (0.55,0.55) {$\cdots$};
\end{tikzpicture}
\begin{tikzpicture}
    \draw (0,0) rectangle (1.1,1.1); 
    \node  at (0.55,0.55) {$\mathbf{e^{(\alpha)}}$};
\end{tikzpicture}
\begin{tikzpicture}
    \draw (0,0) rectangle (1.1,1.1); 
    \node  at (0.55,0.55) {$\mathbf{e'^{(0)}}$};
\end{tikzpicture}
\begin{tikzpicture}
    \draw (0,0) rectangle (1.1,1.1); 
    \node  at (0.55,0.55) {$\mathbf{e'^{(1)}}$};
\end{tikzpicture}
\begin{tikzpicture} 
    \node at (0.55,0.55) {$\cdots$};
\end{tikzpicture}
\begin{tikzpicture}
    \draw (0,0) rectangle (1.1,1.1); 
    \node  at (0.55,0.55) {$\mathbf{e'^{(\beta)}}$};
\end{tikzpicture}
\end{center}
In a similar manner, we define, for a word $\omega$ and for $n \in \N$, the word $\omega^{n}$ by
\begin{equation*}
\omega^{n}=
\begin{cases}
    \omega\circ \omega, & \text{if }\,  n=2; \\
    \omega^{n-1}\circ \omega, & \text{if}\, n\geq3.
\end{cases}
\end{equation*}
\end{de}

Let us define some words which will be useful in the argument. First, we define the trivial word of length $1$ $\omega_{0}$ by 
\begin{center}
\begin{tikzpicture}
    \node at (0.55,0.55) {$\omega_{0}\,:\,$};
\end{tikzpicture}
\begin{tikzpicture}
    \draw (0,0) rectangle (1.1,1.1); 
    \node  at (0.55,0.55){$\mathbf{0}$};
\end{tikzpicture}
\end{center}
When $\mathbf{r_0}=\mathbf{0}$ then this word is related to the \textit{first} (trivial) path from $\mathbf{0}$ to itself, and $\omega_{0}^{k+1}$
encodes the path consisting of $(k+1)$ (\textit{little}) loops at $\mathbf{0}$.
Similarly, we define the word $\omega_{1}$ with respect to the \textit{second} path, which is linked to Konieczny's path~\cite{Kon} using the previous definition of $\mathbf{e_{Ko}^{(t)}}$:
\begin{center}
\begin{tikzpicture}
    \node at (0.55,0.55) {$\omega_{1}\,:\,$};
\end{tikzpicture}
\begin{tikzpicture}
    \draw (0,0) rectangle (1.1,1.1); 
    \node  at (0.55,0.55){$\mathbf{e_{Ko}^{(0)}}$};
\end{tikzpicture}
\begin{tikzpicture}
    \draw (0,0) rectangle (1.1,1.1); 
    \node  at (0.55,0.55){$\mathbf{e_{Ko}^{(1)}}$};
\end{tikzpicture}
\begin{tikzpicture} 
    \node at (0.55,0.55) {$\cdots$};
\end{tikzpicture}
\begin{tikzpicture}
    \draw (0,0) rectangle (1.1,1.1); 
    \node  at (0.55,0.55) {$\mathbf{e_{Ko}^{(k)}}$};
\end{tikzpicture}
\end{center}
Let us define the fundamental word of length $(K+b(k+1))$
\begin{center}
\begin{tikzpicture}
    \node at (0.55,0.55) {$\omega_{0}^{K+b(k+1)}\,:\,$};
\end{tikzpicture}
\begin{tikzpicture}
    \draw (0,0) rectangle (1.1,1.1); 
    \node  at (0.55,0.55){$\mathbf{0}$};
\end{tikzpicture}
\begin{tikzpicture} 
    \node at (0.55,0.55) {$\cdots$};
\end{tikzpicture}
\begin{tikzpicture}
    \draw (0,0) rectangle (1.1,1.1); 
    \node  at (0.55,0.55) {$\mathbf{0}$};
\end{tikzpicture}
\begin{tikzpicture}
    \draw (0,0) rectangle (1.1,1.1); 
    \node  at (0.55,0.55) {$\mathbf{0}$};
\end{tikzpicture}
\begin{tikzpicture}
     \draw[line width=2pt] (0.55,0)--(0.55,1.1);
\end{tikzpicture}
\begin{tikzpicture}
    \draw (0,0) rectangle (1.1,1.1); 
    \node  at (0.55,0.55){$\mathbf{0}$};
\end{tikzpicture}
\begin{tikzpicture} 
    \node at (0.55,0.55) {$\cdots$};
\end{tikzpicture}
\begin{tikzpicture}
    \draw (0,0) rectangle (1.1,1.1); 
    \node  at (0.55,0.55) {$\mathbf{0}$};
\end{tikzpicture}
\end{center}
The first $K$ $\mathbf{0}$'s appearing on the left side in $\omega_{0}^{K+b(k+1)}$ (until the bold vertical line) indicate the ``free fall'' of our path, i.e. the path leading to the vertex $\mathbf{0}$. The $b(k+1)$ other $\mathbf{0}$'s indicate the \textit{first} (trivial) path from $\mathbf{0}$ to itself, defined in~\eqref{0path}.
\vspace{0.1cm}

\subsubsection{Partitioning of the set of the words}
The set $W^{*}$ of words of length $j\geq K+b(k+1)$ which \textit{do not contain} the subword $\omega_{0}^{K+b(k+1)}$  satisfies the inequality
\begin{equation}
\label{cardW}
    |W^{*}| \leq (q^{(k+1)(K+b(k+1))}-1)^{\frac{j}{K+b(k+1)}}.
\end{equation}
Indeed, when subdividing a word of length $j$ into contiguous subwords of length $K+b(k+1)$ (leaving a possibly short subword at the end), each word not containing the word $\omega_{0}^{K+b(k+1)}$ must necessarily not contain $\omega_{0}^{K+b(k+1)}$ as these subwords. There are $q^{(k+1)(K+b(k+1))}-1$ different choices of words other than $\omega_{0}^{K+b(k+1)}$ per subword, and there are fewer than $j/(K+b(k+1))$ full subwords.

In what follows, we will focus on words containing $\omega_{0}^{K+b(k+1)}$, those that do not contain this subword will be trivially estimated.
\\We have
\begin{equation*}
    p_{\mathbf{r_{0}},\mathbf{r_{j}}}^{(j)}=\dfrac{1}{q^{(k+1)j}}\sum_{\substack{\mathbf{e}=(\mathbf{e^{(1)}},\dots,\mathbf{e^{(j)}}) \in (\{0,\dots,q-1\}^{k+1})^{j}\\\mathbf{r_{j}}=\mathbf{\delta_{e^{(j)}}\circ\dots\circ \delta_{e^{(1)}}(r_{0})}}}W(\mathbf{e}),
\end{equation*}
where
\begin{equation}
\label{Wecriture}
    W(\mathbf{e})=\e\left(\dfrac{\ell}{b}\sum_{t=1}^{j}\left(q\Tilde{S}(\mathbf{\delta_{e^{(t)}}\circ\dots\circ \delta_{e^{(1)}}(r_{0})})-\Tilde{S}(\mathbf{\delta_{e^{(t-1)}}\circ\dots\circ \delta_{e^{(1)}}(r_{0})})\right)\right),
\end{equation}
with the initialization, for $t=1$:
\begin{equation*}
    \mathbf{\delta_{e^{(t-1)}}\circ\dots\circ \delta_{e^{(1)}}(r_{0})}=\mathbf{r_{0}}.
\end{equation*}
Let us define $F(\mathbf{r_{j}})$ ($=F(\mathbf{r_{0}},\mathbf{r_{j}})$, to avoid cumbersome notation we will omit the dependency in $\mathbf{r_{0}}$) the set of the elements $\mathbf{e}=(\mathbf{e^{(1)}},\dots, \mathbf{e^{(j)}})$ such that there exists $1 \leq m \leq j-K-b(k+1)+1$ such that $\mathbf{e^{(m)}}=\cdots=\mathbf{e^{(m+K+b(k+1)-1)}}=\mathbf{0}$. For $0<m\leq j-K-b(k+1)+1$, we note $F_{m}(\mathbf{r_{j}}) \subset F(\mathbf{r_{j}})$ the set of the elements $\mathbf{e}=(\mathbf{e^{(1)}},\dots, \mathbf{e^{(j)}})$ such that $m$ is the smallest integer satisfying $\mathbf{e^{(m)}}=\cdots=\mathbf{e^{(m+K+b(k+1)-1)}}=\mathbf{0}$. With these definitions, we have the partition
\begin{equation*}
    F(\mathbf{r_{j}})=\biguplus_{m=1}^{j-K-b(k+1)}F_{m}(\mathbf{r_{j}}).
\end{equation*}
In the following subsection, we will associate to each ${\bf e}$ containing the subword $\omega_0^{K+b(k+1)}$, $(b-1)$  words $\varphi_1 ({\bf e}),\ldots , \varphi _{b-1} ({\bf e})$ such that for $\lambda=1,\ldots, b-1$, $\varphi_{\lambda}$ replaces the  $\lambda$ first subwords $\omega_0^{k+1}$ by $\lambda$ subwords $\omega_1$. Roughly speaking, $\varphi_{\lambda}$ replaces $\lambda$ \textit{little} loops at $\mathbf{0}$ by $\lambda$ \textit{great} loops at $\mathbf{0}$.  Next we will get some cancellation between the weights corresponding to all these words.
\subsubsection{Switching and coupling with Konieczny's path}

\begin{de}[$\lambda$-switching function]
\label{Switch}
For $\lambda \in \{1,\dots,b-1\}$, we define the function $\varphi_{\lambda} \,:\, F(\mathbf{r_{j}}) \longrightarrow (\{0,\dots,q-1\}^{k+1})^{j}$ as follows:
\\For $\mathbf{e}=(\mathbf{e^{(1)}},\ldots,\mathbf{e^{(j)}}) \in F_{m}(\mathbf{r_{j}})$, we set $\varphi_{\lambda}(\mathbf{e})=(\mathbf{g^{(1)}},\ldots,\mathbf{g^{(j)}})$, with
\begin{center}
$
\begin{cases}
 \mathbf{g^{(m+K+t)}}=\mathbf{e_{Ko}^{t \bmod{(k+1)}}}\ &\text{if }\, t \in \{0,\dots,\lambda(k+1)-1\};\\
 \mathbf{g^{(t)}}=\mathbf{e^{(t)}}\ &\text{if }\, t \in \{0,\dots,m-1\}\cup \{m+K+b(k+1),\ldots,j\},
\end{cases}
$
\end{center}
where $\mathbf{e_{Ko}^{(0)}},\ldots,\mathbf{e_{Ko}^{(k)}}$ are the vectors defined in the beginning of Section~\ref{effectiveeta0} and corresponding to \textit{Konieczny's path}.
\end{de}
We write
\begin{equation}
\label{Cassen2}
    p_{\mathbf{r_{0}},\mathbf{r_{j}}}^{(j)}=\dfrac{1}{q^{(k+1)j}}\sum_{\mathbf{e} \in F(\mathbf{r_{j}})\cup \varphi_{1}(F(\mathbf{r_{j}}))\cup\dots\cup \varphi_{b-1}(F(\mathbf{r_{j}}))}W(\mathbf{e})+\dfrac{1}{q^{(k+1)j}}\sum_{\mathbf{e} \notin F(\mathbf{r_{j}})\cup \varphi_{1}(F(\mathbf{r_{j}}))\cup\dots\cup \varphi_{b-1}(F(\mathbf{r_{j}}))}W(\mathbf{e}).\\
\end{equation}
For all $1\leq m \leq j-K-b(k+1)+1$, let us denotes $H_{m}(\mathbf{r_{j}})$ the set of the elements $\mathbf{e}=(\mathbf{e^{(1)}},\ldots,\mathbf{e^{(j)}}) \in  F(\mathbf{r_{j}})\cup \varphi_{1}(F(\mathbf{r_{j}}))\cup\dots\cup \varphi_{b-1}(F(\mathbf{r_{j}}))$ such that $m$ is the smallest integer such that the word linked to $(\mathbf{e^{(m)}},\dots, \mathbf{e^{(m+K+b(k+1)-1)}})$ belongs to the set 
\begin{equation*}
    \{\omega_{0}^{K+k+1},\omega_{0}^{K}\circ \omega_{1} \circ \omega_{0}^{(k+1)(b-1)},\dots, \omega_{0}^{K} \circ \omega_{1}^{b-1}\circ \omega_{0}\}.
\end{equation*}
With this definition, we have the partition
\begin{equation*}
    F(\mathbf{r_{j}})\cup \varphi_{1}(F(\mathbf{r_{j}}))\cup\dots\cup \varphi_{b-1}(F(\mathbf{r_{j}}))=\biguplus\limits_{m=1}^{j-K-b(k+1)+1}H_{m}(\mathbf{r_{j}}).
\end{equation*}
Thus,
\begin{align*}
    \sum_{\mathbf{e} \in  F(\mathbf{r_{j}})\cup \varphi_{1}(F(\mathbf{r_{j}}))\cup\cdots\cup \varphi_{b-1}(F(\mathbf{r_{j}}))}W(\mathbf{e})&=\sum_{m=1}^{j-K-b(k+1)+1}\sum_{\mathbf{e} \in H_{m}(\mathbf{r_{j}})}W(\mathbf{e})\\
    &=\sum_{m=1}^{j-K-b(k+1)+1}\sum_{\mathbf{e} \in F_{m}(\mathbf{r_{j}})}(W(\mathbf{e})+W(\varphi_{1}(\mathbf{e}))+\dots+W(\varphi_{b-1}(\mathbf{e}))).
\end{align*}
Moreover, we have the following lemma.

\begin{lemme}
For all $\mathbf{e} \in F_{m}(\mathbf{r_{j}})$ and for $1\leq m \leq j-K-b(k+1)+1$, we have 
\begin{equation}
\label{sumpoids}
W(\mathbf{e})+W(\varphi_{1}(\mathbf{e}))+\dots+W(\varphi_{b-1}(\mathbf{e}))=0.
\end{equation}
\end{lemme}
\begin{proof}
We prove the case $2\leq m \leq j-K-b(k+1)$. The cases $m=1$ and $m=j-K-b(k+1)+1$ are proven in the same way. The only difference lies in the number of parts unchanged by our function (only one unchanged part). 
\\According to~\eqref{Wecriture} and the definition of $\varphi_\lambda$, we have, for $2\leq m \leq j-K-b(k+1)$, $\lambda \in \{0,\dots,b-1\}$ and for $t \in \{0,\dots,m-1\}\cup \{m+K+b(k+1),\ldots,j\}$,
\begin{align}
&q\Tilde{S}(\mathbf{\delta_{\varphi_\lambda(e^{(t)})}\circ\dots\circ \delta_{\varphi_\lambda(e^{(1)})}(r_{0})})-\Tilde{S}(\mathbf{\delta_{\varphi_\lambda(e^{(t-1)})}\circ\dots\circ \delta_{\varphi_\lambda(e^{(1)})}(r_{0})})\nonumber\\
&\qquad =q\Tilde{S}(\mathbf{\delta_{e^{(t)}}\circ\dots\circ \delta_{e^{(1)}}(r_{0})})-\Tilde{S}(\mathbf{\delta_{e^{(t-1)}}\circ\dots\circ \delta_{e^{(1)}}(r_{0})}).\label{pointfixe}
\end{align}
These two quantities, which are unchanged by the function $\varphi_\lambda$, will constitute the common weights (of modulus 1) for the numbers $W(\mathbf{e}),W(\varphi_{1}(\mathbf{e})),\ldots,W(\varphi_{b-1}(\mathbf{e}))$. They will not intervene in the calculation of the quantity
\begin{equation*}
|W(\mathbf{e})+W(\varphi_{1}(\mathbf{e}))+\dots+W(\varphi_{b-1}(\mathbf{e}))|. 
\end{equation*}
Moreover, the part changed by the function $\varphi_\lambda$ (which corresponds to $\lambda$ Konieczny's paths) 
has, according to Definition~\ref{Switch} and (\ref{Calcpoids}), a weight which is equal to $$\e\left((-1)^{k+1}\lambda(q-1)\ell/b\right),$$
since one Koniezcny's path has a weight equal to $\e\left((-1)^{k+1}(q-1)\ell/b\right)$.
We then deduce that

\begin{align*}
    |W(\mathbf{e})+W(\varphi_{1}(\mathbf{e}))+&\dots+W(\varphi_{b-1}(\mathbf{e}))|\\&=\Big|1+\e\left((-1)^{k+1}\dfrac{(q-1)\ell}{b}\right)+\e\left((-1)^{k+1}\dfrac{2(q-1)\ell}{b}\right)+\cdots+ \e\left((-1)^{k+1}\dfrac{(b-1)(q-1)\ell}{b}\right)\Big|\\
    &=0.
\end{align*}
\end{proof}
Since $(b,q-1)=1$ and $0<\ell<b$ (see~\eqref{hypob}), the number $(q-1)\ell/b$ is not an integer.
Hence,
\begin{equation*}
     \sum_{\mathbf{e} \in  F(\mathbf{r_{j}})\cup \varphi_{1}(F(\mathbf{r_{j}}))\cup\cdots\cup \varphi_{b-1}(F(\mathbf{r_{j}}))}W(\mathbf{e})=0.
\end{equation*}

\subsubsection{Final estimates}

We write
$F=\bigcup\limits_{\mathbf{r_{j}}}F(\mathbf{r_{j}})$ and reconsider equation~(\ref{Cassen2}). We trivially bound the second sum in~(\ref{Cassen2}). We have, according to the inequality~(\ref{cardW}),
\begin{align*}
    \sum_{\mathbf{r_{j}}\in \mathfrak{R}}\Big|\sum_{\mathbf{e} \notin F(\mathbf{r_{j}})\cup \varphi_{1}(F(\mathbf{r_{j}}))\cup\dots\cup \varphi_{b-1}(F(\mathbf{r_{j}}))}W(\mathbf{e})\Big| &\leq \sum_{\mathbf{r_{j}}\in \mathfrak{R}}\;\sum_{\mathbf{e} \notin F(\mathbf{r_{j}})\cup \varphi_{1}(F(\mathbf{r_{j}}))\cup\dots\cup \varphi_{b-1}(F(\mathbf{r_{j}}))}1\\
    &\leq \Big|\sum_{\mathbf{e} \notin F\cup \varphi_{1}(F)\cup\dots\cup \varphi_{b-1}(F)}1\Big|\\
    &\leq (q^{(k+1)(K+b(k+1))}-1)^{\frac{j}{K+b(k+1)}}.
\end{align*}
We insert this information in~(\ref{Cassen2}) and get 
\begin{equation*}
\sum\limits_{\mathbf{r_{j}}\in \mathfrak{R}}|p^{(j)}_{\mathbf{r_{0}},\mathbf{r_{j}}}| \leq \dfrac{1}{q^{(k+1)j}}(q^{(k+1)(K+b(k+1))}-1)^{\frac{j}{K+b(k+1)}}.
\end{equation*}

At this point, we take
\begin{equation*}
    j=K+b(k+1).
\end{equation*}
By taking the maximum over $\mathbf{r_{0}} \in \mathfrak{R}$, we obtain
\begin{equation}
\label{MajoM}
M=\max\limits_{\mathbf{r_{0}}\in \mathfrak{R}}\left(\sum\limits_{\mathbf{r_{j}}\in \mathfrak{R}}|p^{(j)}_{\mathbf{r_{0}},\mathbf{r_{j}}}|\right) \leq 1-\dfrac{1}{q^{(k+1)(K+b(k+1))}}.
\end{equation}

Using inequality~(\ref{veroinneq}), we get, for each term of the sequence $(v_{\rho})_{\rho \geq 0}=(\max\limits_{\mathbf{r}\in \mathfrak{R}}|A(\rho,\mathbf{r})|)_{\rho \geq 0}$, the bound   
\begin{align*}
v_{\rho} &\leq v_{\rho-j}\max\limits_{\mathbf{r_{0}}\in \mathfrak{R}}\left(\sum\limits_{\mathbf{r_{j}}\in \mathfrak{R}}|p^{(j)}_{\mathbf{r_{0}},\mathbf{r_{j}}}|\right)\\
&\leq M v_{\rho-j}.
\end{align*}
Thus, by iterating, and using the fact that for all $r \geq 0$, $v_r \leq 1$, we get
\begin{equation}
\label{vineq}
    v_{\rho} \leq M^{\lfloor\frac{\rho}{K+b (k+1)}\rfloor}
\end{equation}

and
\[
A(\rho, {\bf r_0})\leq \left ( 1-\frac{1}{q^{(k+1)(K+b(k+1))}}\right )^{\lfloor\frac{\rho}{K+b (k+1)}\rfloor}
\leq 2 \left ( 1-\frac{1}{q^{(k+1)(K+b(k+1))}}\right )^{\frac{\rho}{K+b (k+1)}}\ll 
q^{ -\eta_0\rho},
\]
with
\[ 
\eta_0 =-\frac{\log (1-\frac{1}{q^{(k+1)(K+b(k+1))}})}{(\log q) (K+b(k+1))}
.
\]
Using the inequality $\log(1-x) \leq -x$ (valid for all $0\leq x<1$), we have 
\[
\eta_0  \geq \frac{1}{(\log q)(K+b(k+1))q^{(k+1)(K+b(k+1))}}.
\]

\section{Proof of Theorem~\ref{Thm2}}\label{Theo2.5}
Let $\rho_{2}\geq\rho_{1}>0$ and let $N, D>1$ be two integers
such that
\begin{equation*}
    N^{\rho_{1}} \leq D \leq N^{\rho_{2}}.
\end{equation*}
We recall that
\begin{equation*}
S_{0}(N,q^{\nu},\xi)=\sum\limits_{q^{\nu} \leq m <q^{\nu+1}}\max_{a \geq 0}\Big|\sum\limits_{0 \leq n < N}\e\left(\dfrac{\ell}{b}s_{q}(nm+a)\right)\e(n\xi)\Big|.
\end{equation*}
It is sufficient to show that there exists $\eta_{1}=\eta_{1}(\rho_1,\rho_2)>0$ and $C=C(\rho_1,\rho_2)>0$ such that
\begin{equation}
\label{butsuff}
\dfrac{S_{0}(N,q^{\nu},\xi)}{Nq^{\nu}} \leq CN^{-\eta_{1}}
\end{equation}

for $\nu \in \N$ satisfying $D<q^{\nu}\leq qD.$
Indeed, if we suppose  that the inequality~(\ref{butsuff}) is true, then we have
\begin{align*}
S_{0}(N,D,\xi)&=\sum\limits_{D\leq m <qD}\max_{a \geq 0}\Big|\sum\limits_{0 \leq n < N}\e\left(\dfrac{\ell}{b}s(nm+a)\right)\e(n\xi)\Big|\\
&\leq \sum_{q^{\nu-1} \leq m<q^{\nu}}\max_{a \geq 0}\Big|\sum\limits_{0 \leq n < N}\e\left(\dfrac{\ell}{b}s(nm+a)\right)\e(n\xi)\Big|+\sum_{q^{\nu} \leq m<q^{\nu+1}}\max_{a \geq 0}\Big|\sum\limits_{0 \leq n < N}\e\left(\dfrac{\ell}{b}s(nm+a)\right)\e(n\xi)\Big|\\
&\leq S_{0}(N,q^{\nu-1},\xi)+S_{0}(N,q^{\nu},\xi)\\
&\leq CN^{1-\eta_{1}}q^{\nu-1}+CN^{1-\eta_{1}}q^{\nu}\\
&\leq (q+1)C N^{1-\eta_{1}}D.
\end{align*}

We follow the proof method of~\cite{Spi}. As a first step we iterate the van der Corput inequality $k$ times. After that, the next step consists in rewriting the final expression as a Gowers norm. In the last step, we will choose the various parameters and use the Gowers norm estimate to conclude the proof.

\subsection{Step 1: Iterating the van der Corput inequality}

\subsubsection{The first iteration}
The first iteration of the van der Corput inequality allows to reduce the initial problem to a problem where $s_q$ is replaced by the periodic function $s_q^\lambda$ with
\begin{equation}
\label{defLambda}
\lambda>\nu+1,
\end{equation}
at the cost of an error generated by carry propagation (Lemma~\ref{PropaRetenue}). We rewrite the error term in a form that is convenient with respect to the iterations that we wish to perform.

\begin{lemme}
\label{VDC1}
Let $1<H_{0}\leq N$ and $\lambda>\nu+1$ be integers. We define for $h_0 \in \N$,
\begin{equation*}
S_{1}=S_1(h_0)=\sum_{{0 \leq n<N}}\e\left(\dfrac{\ell}{b}\left(\som{\lambda}((n+h_{0})m+a)-\som{\lambda}(nm+a)\right)\right)
\end{equation*}
and
\begin{equation}\label{exprE0}
E_{0}=\dfrac{3}{H_{0}}+\dfrac{4H_{0}}{q^{\lambda-\nu-1}}+\dfrac{8H_{0}}{N}.
\end{equation}
Then, we have
\begin{equation*}
\Big|S_{0}(N,q^{\nu},\xi)\Big|^{2} \leq (q^{\nu+1}N)^{2}E_{0}+\dfrac{2^{2}Nq^{\nu+1}}{H_{0}}\sum_{m=q^{\nu}}^{q^{\nu+1}}\max_{a \geq 0}\sum\limits_{1\leq h_{0}<H_{0}}
|S_{1}|.
\end{equation*}
\end{lemme}
\begin{proof}
By the Cauchy--Schwarz inequality, we have
\begin{align*}
\Big|S_{0}(N,q^{\nu},\xi)\Big|^{2} &\leq (q^{\nu+1}-q^{\nu})\sum_{m=q^{\nu}}^{q^{\nu+1}}\left(\max_{a \geq 0}\Big|\sum\limits_{0\leq n < N}\e\left(\dfrac{\ell}{b}s_{q}(nm+a)\right)\e(n\xi)\Big|\right)^{2}\\
&\leq q^{\nu+1}\sum_{m=q^{\nu}}^{q^{\nu+1}}\left(\max_{a \geq 0}\Big|\sum\limits_{0\leq n < N}\e\left(\dfrac{\ell}{b}s_{q}(nm+a)\right)\e(n\xi)\Big|\right)^{2}.
\end{align*}
Let $1<H_{0}\leq N$ be an integer. We apply Lemma~\ref{VDCC} and get
\begin{align*}
\Big|\sum\limits_{0\leq n < N}\e\left(\dfrac{\ell}{b}s_{q}(nm+a)\right)\e(n\xi)\Big|^{2} &\leq \dfrac{2N^{2}}{H_{0}}+\dfrac{4N}{H_{0}}\sum\limits_{1\leq h_{0}<H_{0}}
\Big|\sum_{{0 \leq n<N-h_{0}}}\e\left(\dfrac{\ell}{b}\left(s_{q}((n+h_{0})m+a)-s_{q}(nm+a)\right)\right)\Big|\\
&\leq \dfrac{2N^{2}}{H_{0}}+\dfrac{4N}{H_{0}}\sum\limits_{1\leq h_{0}<H_{0}}\Big|\sum_{{0 \leq n<N}}\e\left(\dfrac{\ell}{b}\left(s_{q}((n+h_{0})m+a)-s_{q}(nm+a)\right)\right)\Big|\\
&\quad+\dfrac{4N}{H_{0}}\sum\limits_{1\leq h_{0}<H_{0}}h_{0}.
\end{align*}
Thus, by the carry propagation lemma (Lemma~\ref{PropaRetenue})  we get, for $\lambda>\nu+1$, 
\begin{align*}
\Big|\sum\limits_{0\leq n < N}\e\left(\dfrac{\ell}{b}s_{q}(nm+a)\right)\e(n\xi)\Big|^{2} &\leq \dfrac{2N^{2}}{H_{0}}+\dfrac{4N}{H_{0}}\sum\limits_{1\leq h_{0}<H_{0}}
\Big|\sum_{{0 \leq n<N}}\e\left(\dfrac{\ell}{b}\left(\som{\lambda}((n+h_{0})m+a)-\som{\lambda}(nm+a)\right)\right)\Big|\\
&\quad +\dfrac{4N}{H_{0}}\sum\limits_{1\leq h_{0}<H_{0}}h_{0}\left(\dfrac{Nm}{q^{\lambda}}+3\right).
\end{align*}
This implies 
\begin{align*}
\Big|\sum\limits_{0\leq n < N}\e\left(\dfrac{\ell}{b}s_{q}(nm+a)\right)\e(n\xi)\Big|^{2}&\leq \dfrac{2N^{2}}{H_{0}}+4NH_{0}\left(\dfrac{Nm}{q^{\lambda}}+3\right)\\
&\quad +\dfrac{4N}{H_{0}}\sum\limits_{1\leq h_{0}<H_{0}}
\Big|\sum_{{0 \leq n<N}}\e\left(\dfrac{\ell}{b}\left(\som{\lambda}((n+h_{0})m+a)-\som{\lambda}(nm+a)\right)\right)\Big|.
\end{align*}

 This leads to

\begin{equation*}
\Big|S_{0}(N,q^{\nu},\xi)\Big|^{2} \leq (q^{\nu+1}N)^{2}E_{0}+\dfrac{2^{2}Nq^{\nu+1}}{H_{0}}\sum_{m=q^{\nu}}^{q^{\nu+1}}\max_{a \geq 0}\sum\limits_{1\leq h_{0}<H_{0}}
|S_{1}|,
\end{equation*}
where
\begin{equation*}
S_{1}=\sum_{{0 \leq n<N}}\e\left(\dfrac{\ell}{b}\left(\som{\lambda}((n+h_{0})m+a)-\som{\lambda}(nm+a)\right)\right)
\end{equation*}
and
\begin{equation*}
E_{0}=\dfrac{3}{H_{0}}+\dfrac{4H_{0}}{q^{\lambda-\nu-1}}+\dfrac{8H_{0}}{N}.
\end{equation*}
\end{proof}
\subsubsection{The further iterations}\label{furtheriterations}
Following Spiegelhofer~\cite{Spi}, we continue to apply the van der Corput inequality.
At this point, for the $(k-1)$ next applications of the van der Corput inequality, we will use Lemma~\ref{VDCG}. Let $\mu>0$, and another parameter
\begin{equation}
\label{defSigma}
 \sigma<\mu.   
\end{equation}
For $q^{\nu}\leq m\leq q^{\nu+1}$, we define 
\begin{equation}
\label{defK1}
K_{1}=Q_{q^{2\mu+2\sigma}}\left(\dfrac{m}{q^{2\mu}}\right)Q_{q^{\sigma}}\left(\dfrac{P_{q^{2\mu+2\sigma}}(m/q^{2\mu})}{q^{(k-2)\mu}}\right) \,,\qquad M_{1}=P_{q^{2\mu+2\sigma}}\left(\dfrac{m}{q^{2\mu}}\right)Q_{q^{\sigma}}\left(\dfrac{P_{q^{2\mu+2\sigma}}(m/q^{2\mu})}{q^{(k-2)\mu}}\right)
\end{equation}
and for $2\leq i<k-1$, we define\\

\begin{equation}\label{defKi}
K_{i}=Q_{q^{\mu+2\sigma}}\left(\dfrac{m}{q^{(i+1)\mu}}\right)Q_{q^{\sigma}}\left(\dfrac{P_{q^{\mu+2\sigma}}(m/q^{(i+1)\mu})}{q^{(k-1-i)\mu}}\right)\,,\qquad M_{i}=P_{q^{\mu+2\sigma}}\left(\dfrac{m}{q^{(i+1)\mu}}\right)Q_{q^{\sigma}}\left(\dfrac{P_{q^{\mu+2\sigma}}(m/q^{(i+1)\mu})}{q^{(k-i-1)\mu}}\right).
\end{equation}
Finally, we define
\begin{equation}
\label{defKk-1}
K_{k-1}=Q_{q^{\mu+\sigma}}\left(\dfrac{m}{q^{k\mu}}\right)\,,\qquad  M_{k-1}=P_{q^{\mu+\sigma}}\left(\dfrac{m}{q^{k\mu}}\right).
\end{equation}
The numbers $K_{1},\dots,K_{k-1}$  will be the shifts appearing in Lemma~\ref{VDCG}.
The choice of these shifts is motivated by our wish to approximate, during the numerous digit cutting phases, the emerging fractions by integers (multiples of powers of $q$), using Proposition \ref{ApproxFarey}. Classical results on Farey fractions will later allow us to count the number of repetitions of an element modulo $q^{\rho}$, by Lemma \ref{CardA}.

To begin with, since $Q_{n}(\alpha) \leq n$ for all $\alpha \in \R$ and for all $n>0$, we see that
\begin{equation}
\label{MajoKi}
    K_{1} \leq q^{2\mu+3\sigma}
\end{equation}
and for all $2 \leq i \leq k-1$, we have
\begin{equation}
\label{MajoKi2}
K_{i} \leq q^{\mu+3\sigma}.
\end{equation}
The second iteration of the van der Corput inequality is described in the following lemma. This time, we use Lemma~\ref{VDCG}.
\begin{lemme}
\label{VDC2}
We take the same notations as in Lemma~\ref{VDC1}, with the additional condition $E_{0} \leq 1$. Let $1<H_{1}\leq N$ be an integer. We define for $h_0,h_1 \in \N$,
\begin{equation*}
S_{2}=S_2(h_0,h_1)=\sum_{0\leq n<N}\e\left(\dfrac{\ell}{b}\sum_{w_{0},w_{1}\in \{0,1\}}(-1)^{w_{0}+w_{1}}\som{\lambda}((n+w_{0}h_{0}+w_{1}h_{1}K_{1})m+a)\right)
\end{equation*}
and
\begin{equation*}
E_{1}=\dfrac{3\times 2^{5}H_{1}q^{2\mu+3\sigma}}{N}.
\end{equation*}
Then, we have
\begin{equation*}
\Big|S_{0}(N,q^{\nu},\xi)\Big|^{4} \leq 2(q^{\nu+1}N)^{4}(E_{0}+E_{1})+\dfrac{2^{6}N^{3}q^{3(\nu+1)}}{H_{0}H_{1}}\sum_{m=q^{\nu}}^{q^{\nu+1}}\max_{a \geq 0}\sum_{\substack{1 \leq h_{0}<H_{0}\\0 \leq h_{1}<H_{1}}}|S_{2}|.
\end{equation*}
\end{lemme}

\begin{proof}
By the inequality 
$(a+b)^{2} \leq 2a^{2}+2b^{2}$,
valid for $a,b \geq 0$,
we square both sides of the inequality in Lemma~\ref{VDC1} to get
\begin{align*}
\Big|S_{0}(N,q^{\nu},\xi)\Big|^{4} &\leq \left((q^{\nu+1}N)^{2}E_{0}+\dfrac{2^{2}Nq^{\nu+1}}{H_{0}}\sum_{m=q^{\nu}}^{q^{\nu+1}}\max_{a \geq 0}\sum\limits_{1\leq h_{0}<H_{0}}
|S_{1}|\right)^{2}\\
&\leq 2(q^{\nu+1}N)^{4}E_{0}^{2}+\dfrac{2^{5}N^{2}q^{2(\nu+1)}}{H_{0}^{2}}\left(\sum_{m=q^{\nu}}^{q^{\nu+1}}\max_{a \geq 0}\sum\limits_{1\leq h_{0}<H_{0}}
|S_{1}|\right)^{2}.
\end{align*}
By the Cauchy--Schwarz inequality (applied twice in order to handle the ``$\max_{a \geq 0}$''),
\begin{equation*}
\left(\sum_{m=q^{\nu}}^{q^{\nu+1}}\max_{a \geq 0}\sum\limits_{1\leq h_{0}<H_{0}}
|S_{1}|\right)^{2} \leq q^{\nu+1}H_{0}\sum_{m=q^{\nu}}^{q^{\nu+1}}\max_{a \geq 0}\sum\limits_{1\leq h_{0}<H_{0}}|S_{1}|^{2},
\end{equation*}
we have
\begin{equation}
\label{Ineq0}
\Big|S_{0}(N,q^{\nu},\xi)\Big|^{4} \leq 2(q^{\nu+1}N)^{4}E_{0}^{2}+\dfrac{2^{5}N^{2}q^{3(\nu+1)}}{H_{0}}\sum_{m=q^{\nu}}^{q^{\nu+1}}\max_{a \geq 0}\sum\limits_{1\leq h_{0}<H_{0}}
|S_{1}|^{2}.
\end{equation}
We will use this inequality later.
We apply Lemma~\ref{VDCG} for $S_{1}$, $H=H_{1}$ and $K=K_{1}$, and we get
\begin{align*}
|S_{1}|^{2} &\leq 2\dfrac{N+K_{1}(H_{1}-1)}{H_{1}}\sum_{0\leq h_{1}<H_{1}}\left(\Big|\underbrace{\sum_{0\leq n<N}\e\left(\dfrac{\ell}{b}\sum_{w_{0},w_{1}\in \{0,1\}}(-1)^{w_{0}+w_{1}}\som{\lambda}((n+w_{0}h_{0}+w_{1}h_{1}K_{1})m+a)\right)}_{S_{2}}\Big|+h_{1}K_{1}\right)\\
&\leq 2\dfrac{N+K_{1}(H_{1}-1)}{H_{1}}\left(\sum_{0\leq h_{1}<H_{1}}|S_{2}|+H^{2}_{1}K_{1}\right)\\
&\leq \dfrac{2N}{H_{1}}\sum_{0\leq h_{1}<H_{1}}|S_{2}|+\dfrac{2}{H_{1}}\left(NH_{1}^{2}K_{1}+H_{1}^{2}K_{1}^{2}(H_{1}-1)+K_{1}(H_{1}-1)\sum_{0\leq h_{1}<H_{1}}|S_{2}|\right).
\end{align*}
We trivially estimate
\begin{align*}
\sum_{0\leq h_{1}<H_{1}}|S_{2}|\leq NH_{1}
\end{align*}
and get 
\begin{align*}
|S_{1}|^{2} &\leq \dfrac{2N}{H_{1}}\sum_{0\leq h_{1}<H_{1}}|S_{2}|+\dfrac{2}{H_{1}}\left(NH_{1}^{2}K_{1}+H_{1}^{2}K_{1}^{2}(H_{1}-1)+NH_{1}K_{1}(H_{1}-1)\right)\\
&\leq \dfrac{2N}{H_{1}}\sum_{0\leq h_{1}<H_{1}}|S_{2}|+ 2\left(NH_{1}K_{1}+H_{1}K_{1}^{2}(H_{1}-1)+NK_{1}(H_{1}-1)\right).
\end{align*}
Since $H_{1} \leq N$, $K_{1} \geq 2$ and $K_{1}(H_{1}-1)<N$ we have
\begin{equation*}
|S_{1}|^{2} \leq \dfrac{2N}{H_{1}}\sum_{0\leq h_{1}<H_{1}}|S_{2}|+6NH_{1}K_{1}.
\end{equation*}
Using the inequalities~(\ref{MajoKi}) and~(\ref{MajoKi2}), we get
\begin{equation}
\label{majosommax}
\sum_{m=q^{\nu}}^{q^{\nu+1}}\max_{a \geq 0}\sum\limits_{1\leq h_{0}<H_{0}}
|S_{1}|^{2} \leq \dfrac{2N}{H_{1}}\sum_{m=q^{\nu}}^{q^{\nu+1}}\max_{a \geq 0}\sum_{\substack{1 \leq h_{0}<H_{0}\\0 \leq h_{1}<H_{1}}}|S_{2}|+6Nq^{\nu+1}H_{1}H_{0}q^{2\mu+3\sigma}.
\end{equation}
Since $E_{0} \leq 1$, we have $E_{0}^{2} \leq E_{0}$.
Hence, injecting this inequality and~(\ref{majosommax}) in~(\ref{Ineq0}), we have
\begin{equation*}
\Big|S_{0}(N,q^{\nu},\xi)\Big|^{4} \leq2(q^{\nu+1}N)^{4}(E_{0}+E_{1})+\dfrac{2^{6}N^{3}q^{3(\nu+1)}}{H_{0}H_{1}}\sum_{m=q^{\nu}}^{q^{\nu+1}}\max_{a \geq 0}\sum_{\substack{1 \leq h_{0}<H_{0}\\0 \leq h_{1}<H_{1}}}|S_{2}|,
\end{equation*}
where
\begin{equation*}
S_{2}=\sum_{0\leq n<N}\e\left(\dfrac{\ell}{b}\sum_{w_{0},w_{1}\in \{0,1\}}(-1)^{w_{0}+w_{1}}\som{\lambda}((n+w_{0}h_{0}+w_{1}h_{1}K_{1})m+a)\right)
\end{equation*}
and
\begin{equation*}
E_{1}=\dfrac{3\times 2^{5}H_{1}q^{2\mu+3\sigma}}{N}.
\end{equation*}
\end{proof}
\noindent Until the end of the proof, we set
\begin{equation}
\label{defRho}
\rho=\lambda-k\mu.  
\end{equation}
We choose $H_{1}=q^{\rho}$. This choice will allow us to use the bound $E_{1}+E_{0} \leq 1$ and
thus
$(E_{1}+E_{0})^{2} \leq E_{1}+E_{0}.$
By applying exactly the same argument as before, we have, by a second application of Lemma~\ref{VDCG}, the following lemma:


\begin{lemme}
We take the same notations as in Lemma~\ref{VDC1} and Lemma~\ref{VDC2}, with the condition $E_{0}+E_{1} \leq 1$. Let $1<H_{2}\leq N$ be an integer. We set
\begin{equation*}
S_{3}=\sum\limits_{0 \leq n<N}\e\left(\dfrac{\ell}{b}\sum\limits_{w_{0},w_{1},w_{2}\in \{0,1\}}(-1)^{w_{0}+w_{1}+w_{2}}\som{\lambda}((n+w_{0}h_{0}+w_{1}h_{1}K_{1}+w_{2}h_{2}K_{2})m+a)\right)
\end{equation*}
and
\begin{equation*}
E_{2}=\dfrac{3\times 2^{12}H_{2}K_{2}}{N}.
\end{equation*}
Then, we have
\begin{equation*}
\Big|S_{0}(N,q^{\nu},\xi)\Big|^{8} \leq 2^{2}(q^{\nu}N)^{8}(E_{0}+E_{1}+E_{2})+\dfrac{2^{14}N^{7}q^{7(\nu+1)}}{H_{0}H_{1}H_{2}}\sum_{\substack{1\leq h_{0}<H_{0}\\0\leq h_{1}<H_{1}\\0\leq h_{2}<H_{2}}}\;\sum_{m=q^{\nu}}^{q^{\nu+1}}\max_{a \geq 0}|S_{3}|.
\end{equation*}
\end{lemme}
At this point, we can have a recursive formula for the $k$-th iteration of the van der Corput inequality. In the following lemma, we first sum over $m$. The next step will consist in carrying out some digits shiftings for a fixed parameter $m$.
\begin{lemme}
\label{VDCK}
Let $k \geq 2$ be an integer. Let $H_{0}, H_{1},\dots, H_{k-1}>1$ be integers. We set
\begin{align}
    E_{0}&=\dfrac{3}{H_{0}}+\dfrac{4H_{0}}{q^{\lambda-\nu}}+\dfrac{8H_{0}}{N},\nonumber\\
    E_{i}&=\dfrac{3\times 2^{2^{i+2}-1}}{2^{i+1}}\dfrac{H_{i}q^{2\mu+3\sigma}}{N},\qquad \qquad 1\leq i \leq k-1.\label{exprEi}
\end{align}
Moreover, we assume that
\begin{equation}\label{sumEi}
    E_{0}+E_{1}+\cdots+E_{k-1} \leq 1.
\end{equation}
Set
\begin{equation*}
\label{S4}
    S_{4}=\sum_{0 \leq n<N}\e\left(\dfrac{\ell}{b}\sum\limits_{\mathbf{w}=(w_{0},\dots,w_{k-1})\in \{0,1\}^{k}}(-1)^{s_{2}(w)}\som{\lambda}\left(nm+a +w_{0}h_{0}m+\sum_{i=1}^{k-1}w_iK_i m\right) \right).
\end{equation*}

Then we have the inequality
\begin{equation}
\label{EqInd}
\Big|\dfrac{S_{0}(N,q^{\nu},\xi)}{q^{\nu+1}N}\Big|^{2^{k}} \leq 2^{k-1}(E_{0}+\dotsb+E_{k-1})+\dfrac{2^{2^{k+2}-2}}{H_{0}\dotsb H_{k-1}q^{\nu+1}N}\sum_{m=q^{\nu}}^{q^{\nu+1}}\sum_{\substack{1 \leq h_{0}<H_{0}\\ 0\leq h_{i}<H_{i},\ 1 \leq i \leq k-1}}\max_{a \geq 0}|S_{4}|.
\end{equation}
\end{lemme}

Lemma~\ref{VDCK} finishes the first step. Our final choice of the parameters will well respect \eqref{sumEi}. We will verify all arising inequalities in Annexe \ref{choosingsection}. 

\subsection{Step 2: Digits shifting and cutting}
The next goal is to transform $S_{4}$ in order to discard many blocks of digits by digits shifting. We recall the notation from Lemma \ref{CardA}:
\begin{equation*}
\mathfrak{M}_{1}=P_{q^{\sigma}}\left(\dfrac{P_{q^{2\mu+2\sigma}}(m/q^{2\mu})}{q^{(k-2)\mu}}\right).
\end{equation*}
For $1<i<k-1$,
\begin{equation*}
\mathfrak{M}_{i}=P_{q^{\sigma}}\left(\dfrac{P_{q^{\mu+2\sigma}}(m/q^{(i+1)\mu})}{q^{(k-i-1)\mu}}\right)
\end{equation*}
and
\begin{equation*}
\mathfrak{M}_{k-1}=P_{q^{\sigma}}\left(\dfrac{P_{q^{\mu+\sigma}}(m/q^{k\mu})}{q^{k\mu}}\right).
\end{equation*}
We have the following lemma, which is the first transformation of $S_{4}$.
\begin{lemme}
\label{T1}
We take the same notations as in Lemma~\ref{VDCK} and the definition (\ref{defRho}) of $\rho$.
Then, we have
\begin{equation}
\label{S4Decompo}
S_{4}=S_{5}+O_{k}\left(N\sum_{i=2}^{k-1}D_{N}\left(\dfrac{m}{q^{i\mu}}\right)+\dfrac{2N}{q^{\sigma}}\left(H_{1}+H_{2}+\dotsb+H_{k-1}\right)\right),
\end{equation}
where
\begin{equation}
\label{S5}
    S_{5}=\sum_{0 \leq n<N}\e\left(\dfrac{\ell}{b}\sum\limits_{\mathbf{w}=(w_{0},\dots,w_{k-1})\in \{0,1\}^{k}}(-1)^{s_{2}(w)}\som{\rho}\left(\Big\lfloor\dfrac{nm+a+w_{0}h_{0}m}{q^{k\mu}}+\sum\limits_{i=1}^{k-1}w_{i}h_{i}\mathfrak{M}_{i}\Big\rfloor\right)\right).
\end{equation}
\end{lemme}
\begin{proof}
By~(\ref{defK1}), and by the definition of Farey sequences (see Annexe \ref{appendix1}), we have 
\begin{equation*}
\dfrac{M_{1}}{K_{1}}=\dfrac{m}{q^{2\mu}},
\end{equation*}
so that
\begin{equation*}
    K_{1}m=q^{2\mu}M_{1}.
\end{equation*}
In $S_{4}$, defined in~(\ref{S4}), we replace $w_{1}h_{1}K_{1}m$ by $w_{1}h_{1}q^{2\mu}M_{1}$ and apply Lemma~\ref{diffab} with $u_{1}=h_{1}M_{1}$ and $\alpha=2\mu$. This gives for any  $w_0,w_2,\ldots,w_{k-1}$, $n,h_i$:
\begin{align*}
    &-\som{\lambda}(nm+a+w_0h_0m+h_1q^{2\mu}m+\sum_{j=2}^{k-1}w_ih_iK_im)+\som{\lambda}(nm+a+w_0h_0m+\sum_{j=2}^{k-1}w_ih_iK_im)\\
&=-\somm{2\mu}{\lambda}(nm+a+w_0h_0m+h_1q^{2\mu}m+\sum_{j=2}^{k-1}w_ih_iK_im)+\somm{2\mu}{\lambda}(nm+a+w_0h_0m+\sum_{j=2}^{k-1}w_ih_iK_im).
\end{align*}
We insert this in $S_4$ and recall that in the previous formula the first term of the right side   corresponds to $w_1=1$, the second to $w_1=0$,  with the ``$-$'' sign associated to the
corresponding $(-1)^{s_2 (w)}$.
Furthermore, in the second line below we will use the fact that for $u\in\N$, 
$\somm{2\mu}{\lambda} (u)=\som{\lambda-2\mu}(\lfloor u/q^{2\mu}\rfloor )$:

\begin{align}
S_{4}&=\sum_{0 \leq n<N}\e\left(\dfrac{\ell}{b}\sum\limits_{\mathbf{w}=(w_{0},\dots,w_{k-1})\in \{0,1\}^{k}}(-1)^{s_{2}(w)}\somm{2\mu}{\lambda}(nm+a +w_{0}h_{0}m+w_{1}h_{1}q^{2\mu}m+\dotsb +w_{k-1}h_{k-1}K_{k-1}m\right) \nonumber\\
&=\sum_{0 \leq n<N}\e\left(\dfrac{\ell}{b}\sum\limits_{\mathbf{w}=(w_{0},\dots,w_{k-1})\in \{0,1\}^{k}}(-1)^{s_{2}(w)}\som{\lambda-2\mu}\left(\Big\lfloor\dfrac{1}{q^{2\mu}}\left(nm+a+w_{0}h_{0}m+\sum\limits_{i=1}^{k-1}w_{i}h_{i}K_{i}m)\right)\Big\rfloor\right)\right) \nonumber\\
&=\sum_{0 \leq n<N}\e\left(\dfrac{\ell}{b}\sum\limits_{\mathbf{w}=(w_{0},\dots,w_{k-1})\in \{0,1\}^{k}}(-1)^{s_{2}(w)}\som{\lambda-2\mu}\left(\Big\lfloor\dfrac{nm+a+w_{0}h_{0}m}{q^{2\mu}}+w_{1}h_{1}M_{1}+w_{2}h_{2}\dfrac{K_{2}m}{q^{2\mu}}+\sum\limits_{i=3}^{k-1}w_{i}\dfrac{h_{i}K_{i}m}{q^{2\mu}}\Big\rfloor\right)\right) \label{S4A}.
\end{align}
We now wish to substitute ${K_{2}m}/{q^{2\mu}}$ by a multiple of a power of $q$. We first use Proposition~\ref{ApproxFarey} to approximate the fraction ${K_{2}m}/{q^{2\mu}}$ by the integer $q^{\mu}M_{2}$ at the cost of an admissible error. By~(\ref{defKi}) we have
\begin{align}
\Big|\dfrac{K_{2}m}{q^{2\mu}}-q^{\mu}M_{2}\Big|&=q^{\mu}Q_{q^{\sigma}}\left(\dfrac{P_{q^{\mu+2\sigma}}(m/q^{3\mu})}{q^{(k-3)\mu}}\right)\Big|\dfrac{m}{q^{3\mu}}Q_{q^{\mu+2\sigma}}\left(\dfrac{m}{q^{3\mu}}\right)-P_{q^{\mu+2\sigma}}\left(\dfrac{m}{q^{3\mu}}\right)\Big| \nonumber\\
&< q^{\sigma+\mu}\times \dfrac{1}{q^{\mu+2\sigma}} \nonumber\\
&= q^{-\sigma} \label{eqcc}.
\end{align}

Moreover, we introduce the set
\begin{equation*}
A=\left\{0 \leq n < N \, :\, \Big\|\dfrac{nm+a+w_{0}h_{0}m}{q^{2\mu}}+w_{1}h_{1}M_{1}+\sum\limits_{i=3}^{k-1}w_{i}\dfrac{h_{i}K_{i}m}{q^{2\mu}}\Big\| \geq \dfrac{H_{2}}{q^{\sigma}}\right\}.
\end{equation*}
For $n \in A$, the inequality~(\ref{eqcc}) shows that
\begin{equation*}
\Big\|\dfrac{K_{2}m}{q^{2\mu}}\Big \|<q^{-\sigma}.
\end{equation*}
At this point, we assume that for all $1 \leq i \leq k-1$, 
\begin{equation}
\label{HiCond}
H_{i}<q^{\sigma-1}. 
\end{equation}
In Annexe~\ref{choosingsection}, we choose for all $1 \leq i \leq k-1$, $H_{i}=q^{\rho}$. Hence, we have the condition 
\begin{equation}
\label{rhosigmaj}
    \rho<\sigma-1.
\end{equation}

As $0 \leq h_{2} \leq H_{2}<q^{\sigma-1}$, we have by the inequality~(\ref{petit3}) in  Proposition~\ref{aplusb},
$$
\Big\langle \dfrac{h_{2}K_{2}m}{q^{2\mu}}\Big\rangle =h_{2}\Big\langle\dfrac{K_{2}m}{q^{2\mu}}\Big\rangle
=h_{2}q^{\mu}M_{2}.
$$
Hence, by the inequality~(\ref{petit1}) in Proposition~\ref{aplusb}, we get

\begin{align*}
\Big\lfloor \dfrac{nm+a+w_{0}h_{0}m}{q^{2\mu}}+w_{1}h_{1}M_{1}+&\sum\limits_{i=3}^{k-1}w_{i}\dfrac{h_{i}K_{i}m}{q^{2\mu}}+w_2\dfrac{h_{2}K_{2}m}{q^{2\mu}} \Big\rfloor\\
&=\Big\lfloor \dfrac{nm+a+w_{0}h_{0}m}{q^{2\mu}}+w_{1}h_{1}M_{1}+\sum\limits_{i=3}^{k-1}w_{i}\dfrac{h_{i}K_{i}m}{q^{2\mu}}\Big\rfloor+w_2\Big\langle \dfrac{h_{2}K_{2}m}{q^{2\mu}}\Big\rangle \\
&=\Big\lfloor \dfrac{nm+a+w_{0}h_{0}m}{q^{2\mu}}+w_{1}h_{1}M_{1}+\sum\limits_{i=3}^{k-1}w_{i}\dfrac{h_{i}K_{i}m}{q^{2\mu}}\Big\rfloor+w_2h_{2}q^{\mu}M_{2}\\
&=\Big\lfloor \dfrac{nm+a+w_{0}h_{0}m}{q^{2\mu}}+w_{1}h_{1}M_{1}+w_2h_{2}q^{\mu}M_{2}+\sum\limits_{i=3}^{k-1}w_{i}\dfrac{h_{i}K_{i}m}{q^{2\mu}} \Big\rfloor.
\end{align*}
By separating the sum in~(\ref{S4A}) on whether $n \in A$ or not, and by using Lemma~\ref{DisPetitN} in order to bound the sum for $n \notin A$, we get 
\begin{align*}
S_{4}&=\sum_{0 \leq n<N}\e\left(\dfrac{\ell}{b}\sum\limits_{\mathbf{w}=(w_{0},\dots,w_{k-1})\in \{0,1\}^{k}}(-1)^{s_{2}(w)}\som{\lambda-2\mu}\left(\Big\lfloor\dfrac{nm+a+w_{0}h_{0}m}{q^{2\mu}}+w_{1}h_{1}M_{1}+w_{2}h_{2}q^{\mu}M_{2}+\sum\limits_{i=3}^{k-1}w_{i}\dfrac{h_{i}K_{i}m}{q^{2\mu}}\Big\rfloor\right)\right)\\
&\qquad +O\left(ND_{N}\left(\dfrac{m}{q^{2\mu}}\right)+\dfrac{2H_{1}N}{q^{\sigma}}\right).
\end{align*}
Repeating the same argument as before, we have 
\begin{align*}
S_{4}&=\sum_{0 \leq n<N}\e\left(\dfrac{\ell}{b}\sum\limits_{\mathbf{w}=(w_{0},\dots,w_{k-1})\in \{0,1\}^{k}}(-1)^{s_{2}(w)}\som{\lambda-3\mu}\left(\Big\lfloor\dfrac{nm+a+w_{0}h_{0}m}{q^{3\mu}}+w_{1}h_{1}\dfrac{M_{1}}{q^{\mu}}+w_{2}h_{2}M_{2}+\sum\limits_{i=3}^{k-1}w_{i}\dfrac{h_{i}K_{i}m}{q^{3\mu}}\Big\rfloor\right)\right)\\
&\qquad +O\left(N\left(D_{N}\left(\dfrac{m}{q^{2\mu}}\right)+D_{N}\left(\dfrac{m}{q^{3\mu}}\right)\right)+\dfrac{2N}{q^{\sigma}}(H_{1}+H_{2})\right).
\end{align*}
We continue the process and get
\begin{align*}
S_{4}&=\sum_{0 \leq n<N}\e\left(\dfrac{\ell}{b}\sum\limits_{\mathbf{w}=(w_{0},\dots,w_{k-1})\in \{0,1\}^{k}}(-1)^{s_{2}(w)}\som{\lambda-k\mu}\left(\Big\lfloor\dfrac{nm+a+w_{0}h_{0}m}{q^{k\mu}}+\sum\limits_{i=1}^{k-1}w_{i}\dfrac{h_{i}M_{i}}{q^{(k-1-i)\mu}}\Big\rfloor\right)\right)\\
&\qquad +O\left(N\sum_{i=2}^{k-1}D_{N}\left(\dfrac{m}{q^{i\mu}}\right)+\dfrac{2N}{q^{\sigma}}\left(H_{1}+\dotsb+H_{k-1}\right)\right).
\end{align*}
In order to remove ${M_{i}}/{q^{(k-1-i)\mu}}$ ($1 \leq i \leq k-1$), we use the Farey approximation again (Proposition~\ref{ApproxFarey}). We have, 
\begin{align*}
\Bigg|\underbrace{\dfrac{P_{q^{2\mu+2\sigma}}\left(\dfrac{m}{q^{2\mu}}\right)}{q^{(k-2)\mu}}Q_{q^{\sigma}}\left(\dfrac{P_{q^{2\mu+2\sigma}}(m/q^{2\mu})}{q^{(k-2)\mu}}\right)}_{M_{1}/q^{(k-2)\mu}}-P_{q^{\sigma}}\left(\dfrac{P_{q^{2\mu+2\sigma}}(m/q^{2\mu})}{q^{(k-2)\mu}}\right)\Bigg| < q^{-\sigma}.
\end{align*}
For $1<i<k-1$, we have
\begin{align*}
\Bigg|\underbrace{\dfrac{P_{q^{\mu+2\sigma}}\left(\dfrac{m}{q^{(i+1)\mu}}\right)}{q^{(k-i-1)\mu}}Q_{q^{\sigma}}\left(\dfrac{P_{q^{\mu+2\sigma}}(m/q^{(i+1)\mu})}{q^{(k-i-1)\mu}}\right)}_{M_{i}/q^{(k-i-1)\mu}}-P_{q^{\sigma}}\left(\dfrac{P_{q^{\mu+2\sigma}}(m/q^{(i+1)\mu})}{q^{(k-i-1)\mu}}\right)\Bigg| < q^{-\sigma}
\end{align*}
and
\begin{align*}
\Bigg|\underbrace{\dfrac{P_{q^{\mu+\sigma}}\left(\dfrac{m}{q^{k\mu}}\right)}{q^{k\mu}}Q_{q^{\sigma}}\left(\dfrac{P_{q^{\mu+\sigma}}(m/q^{k\mu})}{q^{k\mu}}\right)}_{M_{k-1}/q^{k\mu}}-P_{q^{\sigma}}\left(\dfrac{P_{q^{\mu+\sigma}}(m/q^{k\mu})}{q^{k\mu}}\right)\Bigg| < q^{-\sigma}.
\end{align*}

With the same argument as before, we get
\begin{align*}
S_{4}&=\underbrace{\sum_{0 \leq n<N}\e\left(\dfrac{\ell}{b}\sum\limits_{\mathbf{w}=(w_{0},\dots,w_{k-1})\in \{0,1\}^{k}}(-1)^{s_{2}(w)}\som{\lambda-k\mu}\left(\Big\lfloor\dfrac{nm+a+w_{0}h_{0}m}{q^{k\mu}}+\sum\limits_{i=1}^{k-1}w_{i}h_{i}\mathfrak{M}_{i}\Big\rfloor\right)\right)}_{S_{5}}\\
&\qquad +O_{k}\left(N\sum_{i=2}^{k-1}D_{N}\left(\dfrac{m}{q^{i\mu}}\right)+\dfrac{2N}{q^{\sigma}}\left(H_{1}+H_{2}+\dotsb+H_{k-1}\right)\right),
\end{align*}
which ends the proof of Lemma \ref{T1}.
\end{proof}
\noindent
Our next goal is to handle the term
\begin{equation*}
\Big\lfloor \dfrac{nm+a}{q^{k\mu}}\Big\rfloor
\end{equation*}
present in $S_{5}$ in order to simplify even further the argument of $s_q^{\lambda-k\mu}$. Recall that $\rho=\lambda-k\mu$, by \eqref{defRho}.  

\begin{lemme}
\label{6.6}
Let $T, h_0, m \geq 2$ be  integers. We set
\begin{equation}
\label{Lambdadef}
\Lambda=\Lambda(h_0,m)=\Big\{0 \leq t < T :\quad \Big[\dfrac{t}{T}+\dfrac{h_{0}m}{q^{k\mu}}, \,\dfrac{t+1}{T}+\dfrac{h_{0}m}{q^{k\mu}}\Big[\cap \Z=\emptyset \Big\}.
\end{equation} Then for all $k\geq 1$,
\begin{align*}
    \Big|\dfrac{S_{0}(N,\nu,\xi)}{q^{\nu+1}N}\Big|^{2^{k}}  &\leq 2^{k-1}(E_{0}+\dots+E_{k-1})+\dfrac{2^{2^{k+2}-2}}{T}+\dfrac{2^{2^{k+2}-2}}{H_{0}\cdots H_{k-1}q^{\nu+1+\rho}T}\sum_{m=q^{\nu}}^{q^{\nu+1}}\sum_{\substack{1 \leq h_{0}<H_{0}\\1 \leq i \leq k-1,\ 0\leq h_{i}<H_{i}}}\sum_{t \in \Lambda}|S_{8}|\\
    &+2^{2^{k+2}-2}O\left(\dfrac{T}{q^{\nu+1-\rho}}\sum_{m=q^{\nu}}^{q^{\nu+1}}D_{N}\left(\dfrac{m}{q^{\lambda}}\right)+\dfrac{1}{q^{\nu+1}}\sum_{m=q^{\nu}}^{q^{\nu+1}}\sum_{i=2}^{k-1}D_{N}\left(\dfrac{m}{q^{i\mu}}\right)+\dfrac{2(H_{1}+\dotsb+H_{k-1})}{q^{\sigma}}\right),
\end{align*}
where
\begin{equation*}
S_{8}=\sum\limits_{0 \leq n'<q^{\rho}}\e\left(\dfrac{\ell}{b}\sum\limits_{\mathbf{w}=(w_{0},\dots,w_{k-1})\in \{0,1\}^{k}}(-1)^{s_{2}(w)}\som{\rho}\left(n'+\Big\lfloor \dfrac{t}{T}+\dfrac{w_{0}h_{0}m}{q^{k\mu}}\Big\rfloor+\sum\limits_{1 \leq i \leq k-1}w_{i}h_{i}\mathfrak{M}_{i}\right)\right).
\end{equation*}
\end{lemme}
\begin{proof}
For all $0\leq n \leq N$, there exists a unique couple $(t,n')$ such that $0 \leq t < T$ and $0 \leq n'<q^{\rho}$ where $t/T \leq \{(nm+a)/q^{k\mu}\}<(t+1)/T$ and $\lfloor \frac{nm+a}{q^{k\mu}}\rfloor \equiv n'\Mod{q^{\rho}}$. We rewrite $S_{5}$ (cf.~(\ref{S5})) and split the sum into two parts according to the values of $1\leq  t<T$ whether they lie in $\Lambda$ or not:
\begin{align*}
S_{5}&=\sum_{0\leq t<T}\sum_{0\leq n'<q^{\rho}}\sum_{\substack{0 \leq n<N\\t/T \leq \{(nm+a)/q^{k\mu}\}<(t+1)/T\\\lfloor \frac{nm+a}{q^{k\mu}}\rfloor \equiv n'\Mod{q^{\rho}}}}\e\left(\dfrac{\ell}{b}\sum\limits_{\mathbf{w}=(w_{0},\dots,w_{k-1})\in \{0,1\}^{k}}(-1)^{s_{2}(w)}\som{\rho}\Biggl(\left\lfloor \dfrac{nm+a}{q^{k\mu}} +\dfrac{w_{0}h_{0}m}{q^{k\mu}}\right\rfloor 
\right.\\
&\qquad\qquad\qquad\qquad\qquad\qquad\qquad\qquad\qquad\qquad\qquad\qquad\qquad\qquad \left.+\sum\limits_{1 \leq i \leq k-1}w_{i}h_{i}\mathfrak{M}_{i}\Biggr)\right)\\
&=\sum_{t\in \Lambda } + \sum_{t\not \in \Lambda } =:S_{6}+S_{7}.
\end{align*}
Since
\begin{equation*}
    \bigcup_{t=0}^{T-1}\Big[\frac{t}{T}+\frac{h_{0}m}{q^{k\mu}}, \frac{t+1}{T}+\frac{h_{0}m}{q^{k\mu}}\Big[=\Big[\frac{h_{0}m}{q^{k\mu}}, \frac{h_{0}m}{q^{k\mu}}+1\Big[
\end{equation*}
is an interval of length $1$, this interval contains exactly one integer that must fall into exactly one of the $T$ subintervals. Thus,   $|\Lambda|= T-1.$ 

\medskip

\noindent $\bullet$ For $t \notin \Lambda$ (one single value of $t$), we trivially bound $S_{7}$ with Proposition~\ref{PropUnNEUF}:
\begin{align*}
|S_{7}| &\leq \sum_{0\leq n'<q^{\rho}}\sum_{\substack{0 \leq n<N\\t/T \leq \{(nm+a)/q^{k\mu}\}<(t+1)/T\\\lfloor \frac{nm+a}{q^{k\mu}}\rfloor \equiv n'\Mod{q^{\rho}}}} 1\\
&\leq q^{\rho}\left(\dfrac{N}{Tq^{\rho}}+O\left(ND_{N}\left(\dfrac{m}{q^{\rho+k\mu}}\right)\right)\right)\\
&=\dfrac{N}{T}+O\left(Nq^{\rho}D_{N}\left(\dfrac{m}{q^{\lambda}}\right)\right).
\end{align*}

\noindent $\bullet$ For $t \in \Lambda$ (which is, for $T-1$ values of $t$), we write
\begin{equation*}
\dfrac{nm+a}{q^{k\mu}}=\Big\lfloor \dfrac{nm+a}{q^{k\mu}}\Big\rfloor+\underbrace{\Big\{ \dfrac{nm+a}{q^{k\mu}}\Big\}}_{\in [t/T,(t+1)/T[},
\end{equation*}
we deduce that for $w_{0} \in \{0,1\}$
\begin{equation*}
\Big\lfloor \dfrac{nm+a}{q^{k\mu}}\Big\rfloor+\dfrac{t}{T}+\dfrac{w_{0}h_{0}m}{q^{k\mu}} \leq \dfrac{nm+a+w_{0}h_{0}m}{q^{k\mu}}<\Big\lfloor \dfrac{nm+a}{q^{k\mu}}\Big\rfloor+\dfrac{t+1}{T}+\dfrac{w_{0}h_{0}m}{q^{k\mu}}.
\end{equation*}
Since $t \in \Lambda$, we have for $w_0=1$,
\begin{equation*}
\Big[\Big\lfloor \dfrac{nm+a}{q^{k\mu}}\Big\rfloor+\frac{t}{T}+\frac{h_{0}m}{q^{k\mu}}, \Big\lfloor \dfrac{nm+a}{q^{k\mu}}\Big\rfloor+\frac{t+1}{T}+\frac{h_{0}m}{q^{k\mu}}\Big[ \cap \Z= \emptyset
\end{equation*}
and therefore necessarily
\begin{equation}\label{splitfloor}
\Big\lfloor \dfrac{nm+a}{q^{k\mu}} +\dfrac{w_{0}h_{0}m}{q^{k\mu}}\Big\rfloor=\Big\lfloor \dfrac{nm+a}{q^{k\mu}}\Big\rfloor+\Big\lfloor \dfrac{t}{T} +\dfrac{w_{0}h_{0}m}{q^{k\mu}}\Big\rfloor.
\end{equation}
We note that \eqref{splitfloor} is true also in the case of $w_0=0$ and all $0\leq t<T$ with one exception.
Thus, we have by periodicity of $s_q^{\rho}$, 
\begin{align*}
S_{6}&=\sum_{t \in \Lambda}\sum\limits_{0 \leq n'<q^{\rho}}\sum_{\substack{0\leq n<N\\t/T \leq \{(nm+a)/q^{k\mu}\}<(t+1)/T\\ \lfloor (nm+a)/q^{k\mu}\rfloor \equiv n' \Mod{q^{\rho}}}}\e\left(\dfrac{\ell}{b}\sum\limits_{\mathbf{w}=(w_{0},\dots,w_{k-1})\in \{0,1\}^{k}}(-1)^{s_{2}(w)}\som{\rho}\Biggl(n'+\left\lfloor \dfrac{t}{T}+\dfrac{w_{0}h_{0}m}{q^{k\mu}}\right\rfloor\right.\\
&\qquad\qquad\qquad\qquad\qquad\qquad\qquad\qquad\qquad\qquad\qquad\qquad\qquad\qquad+\left.\sum\limits_{1 \leq i \leq k-1}w_{i}h_{i}\mathfrak{M}_{i}\Biggr)\right)\\
&=\sum_{t \in \Lambda}{\sum\limits_{0 \leq n'<q^{\rho}}\e\left(\dfrac{\ell}{b}\sum\limits_{\mathbf{w}=(w_{0},\dots,w_{k-1})\in \{0,1\}^{k}}(-1)^{s_{2}(w)}\som{\rho}\Biggl(n'+\Big\lfloor \dfrac{t}{T}+\dfrac{w_{0}h_{0}m}{q^{k\mu}}\Big\rfloor+\sum\limits_{1 \leq i \leq k-1}w_{i}h_{i}\mathfrak{M}_{i}\Biggr)\right)}\\
&\qquad\qquad\qquad\qquad\qquad\times \sum_{\substack{0\leq n<N\\t/T \leq \{(nm+a)/q^{k\mu}\}<(t+1)/T\\\lfloor (nm+a)/q^{k\mu}\rfloor \equiv n' \Mod{q^{\rho}}}}1.
\end{align*}
By Proposition~\ref{PropUnNEUF}, the sum over $n$ is $N/Tq^{\rho}+O\left(ND_{N}\left(q^{\rho+k\mu}\right)\right)$. We trivially bound $|S_{8}|$ by $q^{\rho}$ in the estimate  of the contribution of the error term $ND_{N}\left(q^{\rho+k\mu}\right)$,
\begin{align*}
|S_{6}|&=\Big|\sum\limits_{t \in \Lambda}S_{8}\left(\dfrac{N}{q^{\rho}T}+O\left(ND_{N}\left(\dfrac{m}{q^{\rho+k\mu}}\right)\right)\right)\Big|\\
&\leq \dfrac{N}{Tq^{\rho}}\sum_{t \in \Lambda}|S_{8}|+O\left(q^{\rho}NTD_{N}\left(\dfrac{m}{q^{\lambda}}\right)\right).
\end{align*}
Hence, using the triangular inequality, we get
\begin{align*}
|S_{4}|&\leq |S_{6}|+|S_{7}|+O\left(N\sum_{i=2}^{k-1}D_{N}\left(\dfrac{m}{q^{i\mu}}\right)+\dfrac{2N}{q^{\sigma}}(H_{1}+\dotsb+H_{k-1})\right)\\
&\leq \dfrac{N}{Tq^{\rho}}\sum_{t \in \Lambda}|S_{8}|+\dfrac{N}{T}+O\left(q^{\rho}NTD_{N}\left(\dfrac{m}{q^{\lambda}}\right)+N\sum_{i=2}^{k-1}D_{N}\left(\dfrac{m}{q^{i\mu}}\right)+\dfrac{2N}{q^{\sigma}}(H_{1}+\dots+H_{k-1})\right).
\end{align*}

We inject the last inequality in the equation $(\ref{EqInd})$ and get 
\begin{align}
    \Big|\dfrac{S_{0}(N,\nu,\xi)}{q^{\nu+1}N}\Big|^{2^{k}} &\leq 2^{k-1}(E_{0}+\dots+E_{k-1})+\dfrac{2^{2^{k+2}-2}}{H_{0}\cdots H_{k-1}q^{\nu+1}N}\sum_{m=q^{\nu}}^{q^{\nu+1}}\sum_{\substack{1 \leq h_{0}<H_{0}\\ 0\leq h_{i}<H_{i},\ 1 \leq i \leq k-1}}\max_{a \geq 0}|S_{4}|\nonumber\\
    &\leq 2^{k-1}(E_{0}+\dots+E_{k-1})+\dfrac{2^{2^{k+2}-2}}{T} \nonumber\\
    &\qquad+\dfrac{2^{2^{k+2}-2}}{H_{0}\cdots H_{k-1}q^{\nu+1+\rho}T}\sum_{m=q^{\nu}}^{q^{\nu+1}}\sum_{\substack{1 \leq h_{0}<H_{0}\\ 0\leq h_{i}<H_{i},\ 1 \leq i \leq k-1}}\sum_{t \in \Lambda}|S_{8}|\label{estimS8}\\
    &\qquad+2^{2^{k+2}-2}O\Biggl(\dfrac{T}{q^{\nu+1-\rho}}\sum_{m=q^{\nu}}^{q^{\nu+1}}D_{N}\left(\dfrac{m}{q^{\lambda}}\right)+\dfrac{1}{q^{\nu+1}}\sum_{m=q^{\nu}}^{q^{\nu+1}}\sum_{i=2}^{k}D_{N}\left(\dfrac{m}{q^{i\mu}}\right) \nonumber\\
    &\qquad\qquad\qquad\qquad\qquad\qquad\qquad\qquad\qquad\qquad\qquad\qquad+\dfrac{2(H_{1}+\dotsb+H_{k-1})}{q^{\sigma}}\Biggr),\nonumber
\end{align}
where for the last term we use a very rough bound.

\end{proof}
\noindent Until the end of the proof, for all $1 \leq i \leq k-1$, we take
\begin{equation}\label{exprHi}
H_{i}=q^{\rho}.
\end{equation}
(We note that since $\rho\in \mathbb{N}$ we also have $H_i\in \mathbb{N}$.) In order to bound \eqref{estimS8}, we will focus on the quantity
\begin{equation*}
S_{9}:=\sum_{m=q^{\nu}}^{q^{\nu+1}}\,\sum\limits_{0 \leq h_{1},\dots, h_{k-1}<q^{\rho}}|S_{8}|\leq\sum_{m=0}^{q^{\nu+1}}\,\sum\limits_{0 \leq h_{1},\dots, h_{k-1}<q^{\rho}}|S_{8}|.
\end{equation*}
For a fixed integer $\gamma>0$ (to be chosen later), let
\begin{equation}
\label{defEnsB}
B=\{m \in \{0,\dots,q^{\nu+1}\}\,:\, \exists p|q,\,\exists i \in \{1,\dots,k-1\},\,\;p^{3\gamma}| \,\mathfrak{M}_{i}\}.
\end{equation}
We first split the sum $S_{9}$ based on whether $m \in B$ or not. For $m \notin B$, we replace $h_{i}\mathfrak{M}{i}$ (for $1 \leq i<k)$ by $h'_{i}$ and count the number of repetitions with Lemma~\ref{PropA}. For $m \in B$, we bound the sum with Lemma~\ref{CardA}. We will choose 
\begin{equation}
\label{condlaga}
  \nu+1 \geq 3\gamma, 
\end{equation}
see Annexe \ref{choosingsection}. We therefore get
\begin{equation}
\label{majoS9}
    S_{9} \leq q^{3\gamma (k-1)}\sum_{m=0}^{q^{\nu+1}}\;\sum\limits_{0 \leq h'_{1},\dots,h'_{k-1}<q^{\rho}}|S_{10}|+q^{\nu+1+k\rho-3\gamma\log_{q}(P^{-}(q))},
\end{equation}
where 
\begin{equation*}
    S_{10}:=\sum\limits_{0 \leq n'<q^{\rho}}\e\left(\dfrac{\ell}{b}\sum\limits_{\mathbf{w}=(w_{0},\dots,w_{k-1})\in \{0,1\}^{k}}(-1)^{s_{2}(w)}\som{\rho}\left(n'+\Big\lfloor \dfrac{t}{T}+\dfrac{w_{0}h_{0}m}{q^{k\mu}}\Big\rfloor+\sum\limits_{1 \leq i \leq k-1}w_{i}h'_{i}\right)\right).
\end{equation*}
We now want to simplify the terms 
$\theta_{h_0,m}:=\lfloor t/T+w_0h_{0}m/q^{k\mu}\rfloor$ when $w_0=1$.
Since $s_q^\rho$ is $q^\rho$-periodic, it is convenient to split the sums on the parameters $h_0$ and $m$ according to classes modulo
$q^\rho$ of the $\theta_{h_0,m}$:
\[
\sum_{m=0}^{q^{\nu+1}}|S_{10}|
=\sum_{0\le m'<q^\rho}
a(m')\left |
\sum\limits_{0 \leq n'<q^{\rho}}\e\left(\dfrac{\ell}{b}\sum\limits_{\mathbf{w}=(w_{0},\dots,w_{k-1})\in \{0,1\}^{k}}(-1)^{s_{2}(w)}\som{\rho}\left(n'+w_0 m'+\sum\limits_{1 \leq i \leq k-1}w_{i}h'_{i}\right)\right)\right |,
\]
with
\[
a(m')=\Big|\{m \in \{0,\dots,q^{\nu+1}\}\,:\,\lfloor t/T+h_{0}m/q^{k\mu}\rfloor \equiv m' \Mod{q^{\rho}}\}\Big|.
\]
Next we observe\footnote{For $h_0 \leq q^{k\mu}$ and $w_{0}=1$, the set of $m \in \{0,\dots,q^{\nu+1}\}$ such that $\lfloor t/T+h_{0}m/q^{k\mu}\rfloor \equiv m' \Mod{q^{\rho}}$ is the set of $m \in \{0,\dots,q^{\nu+1}\}$ such that $\lfloor t/T+h_{0}m/q^{k\mu}\rfloor=m'+\alpha q^{\rho}$ for some $\alpha \in \Z$. For fixed $\alpha$, the number of $m$ satisfying $m'+\alpha q^{\rho}\leq t/T+h_{0}m/q^{k\mu}<m'+\alpha q^{\rho}+1$ is  $O(q^{k\mu}/h_{0})$. Moreover, by the inequalities $m'+\alpha q^{\rho}\leq t/T+h_{0}m/q^{k\mu}$ and $m \leq q^{\nu+1}$, we have $\alpha \ll h_{0}q^{\nu+1}/q^{k\mu+\rho}$. Therefore, we have $\ll q^{k\mu}/h_{0}\times h_{0}q^{\nu+1}/q^{k\mu+\rho}=q^{\nu+1-\rho}$ integers $m \in \{0,\ldots,q^{\nu+1}\}$ such that $\lfloor t/T+w_{0}h_{0}m/q^{k\mu}\rfloor \equiv m' \Mod{q^{\rho}}$.}  that  for any $1\leq h_0\leq H_0$ we have
\begin{equation*}
    \Big|\{m \in \{0,\dots,q^{\nu+1}\}\,:\,\lfloor t/T+h_{0}m/q^{k\mu}\rfloor \equiv m' \Mod{q^{\rho}}\}\Big| \ll q^{\nu+1-\rho}.
\end{equation*}
We deduce that 
\begin{equation*}
 \sum_{m=0}^{q^{\nu+1}}|S_{10}| \ll q^{\nu+1-\rho}\sum_{0 \leq m'<q^{\rho}}\left |
\sum\limits_{0 \leq n'<q^{\rho}}\e\left(\dfrac{\ell}{b}\sum\limits_{\mathbf{w}=(w_{0},\dots,w_{k-1})\in \{0,1\}^{k}}(-1)^{s_{2}(w)}\som{\rho}\left(n'+m'w_0\sum\limits_{1 \leq i \leq k-1}w_{i}h'_{i}\right)\right)\right |.
\end{equation*}
Injecting this in~\eqref{majoS9}, we have (replacing $m'$ by $h_{0}$)
\begin{equation}
\label{estimateS9}
S_9\ll 
q^{3\gamma(k-1)+\nu+1-\rho}
\sum\limits_{0 \leq h_0, h_{1},\dots,h_{k-1}<q^{\rho}}|S_{10}(h_{0},\dots, h_{k})|+
q^{\nu+1+k\rho-3\gamma\log_{q}(P^{-}(q))},
\end{equation}
with
\begin{equation*}
S_{10}(h_{0},\dots, h_{k-1})=\sum\limits_{0 \leq n'<q^{\rho}}\e\left(\dfrac{\ell}{b}\sum\limits_{\mathbf{w}=(w_{0},\dots,w_{k-1})\in \{0,1\}^{k}}(-1)^{s_{2}(w)}\som{\rho}\left(n'+\sum\limits_{0 \leq i \leq k}w_{i}h_{i}\right)\right).
\end{equation*}

We choose
\begin{equation}
\label{majoh0}
 H_0 \leq q^{k\mu}. 
\end{equation}
We rewrote $h_{i}$ instead of $h'_{i}$ for clarity. Now, we want to remove the absolute values around $S_{10}$. For that purpose, we will use the following result.
\begin{lemme}
\label{Perio}

We have
\begin{equation*}
    |S_{10}(h_0,h_{1},\dots,h_{k-1})|^{2}=\sum_{0 \leq h_{k}<q^{\rho}}S_{10}(h_0,h_{1},\dots,h_{k-1},h_k).
\end{equation*}
\end{lemme}
\begin{proof}
For $n' \in \{0,\dots,q^{\rho}-1\}$, we define $g(n')$ by
\begin{equation*}
   g(n')=\sum\limits_{\mathbf{w}=(w_{0},\dots,w_{k-1})\in \{0,1\}^{k}}(-1)^{s_{2}(w)}\som{\rho}\left(n'+\sum\limits_{0 \leq i \leq k-1}w_{i}h_{i}\right).
\end{equation*}
\begin{align*}
    |S_{10}(h_{0},\dots,h_{k-1})|^{2}&=S_{10}(h_{0},\dots,h_{k-1})\overline{S_{10}(h_{0},\dots,h_{k-1})}\\
    &=\sum_{0 \leq n_{1}<q^{\rho}}\;\sum_{0 \leq n_{2}<q^{\rho}}\e\left(\dfrac{\ell}{b}(g(n_{1})-g(n_{2}))\right).
\end{align*}
For $n_{1} \in \{0,\dots,q^{\rho}-1\}$, we perform the linear transformation $h_{k}=n_{2}-n_{1}$ in the second sum. Thus, we have
\begin{equation*}
|S_{10}(h_{0},\dots,h_{k-1})|^{2}=\sum_{0 \leq n_{1}<q^{\rho}}\;\;\sum_{-n_{1} \leq h_{k}<q^{\rho}-n_{1}}\e\left(\dfrac{\ell}{b}(g(n_{1})-g(n_{1}+h_{k}))\right).
\end{equation*}
For $n_{1} \in \{0,\dots,q^{\rho}-1\}$, the function $h_{k} \mapsto \e\left(\dfrac{\ell}{b}(g(n_{1})-g(n_{1}+h_{k}))\right)$ is $q^{\rho}$-periodic, since $g$ is $q^{\rho}$-periodic. Furthermore, the support of the sum over $h_{k}$ has length $q^{\rho}$. We conclude that
\begin{equation*}
\sum_{-n_{1} \leq h_{k}<q^{\rho}-n_{1}}\e\left(\dfrac{\ell}{b}(g(n_{1})-g(n_{1}+h_{k}))\right)=\sum_{0 \leq h_{k}<q^{\rho}}\e\left(\dfrac{\ell}{b}(g(n_{1})-g(n_{1}+h_{k}))\right).    
\end{equation*}
Hence, by interchanging the sums, we deduce that
\begin{align*}
|S_{10}(h_{0},\dots,h_{k-1})|^{2}&=\sum_{0 \leq h_{k}<q^{\rho}}\;\sum_{0 \leq n_{1}<q^{\rho}}\e\left(\dfrac{\ell}{b}(g(n_{1})-g(n_{1}+h_{k}))\right)\\
&=\sum_{0 \leq h_{k}<q^{\rho}}\;\sum_{0 \leq n_{1}<q^{\rho}}\e\left(\dfrac{\ell}{b}\sum\limits_{\mathbf{w}=(w_{0},\dots,w_{k})\in \{0,1\}^{k+1}}(-1)^{s_{2}(w)}\som{\rho}\left(n'+\sum\limits_{1 \leq i \leq k}w_{i}h_{i}\right)\right)\\
&=\sum_{0 \leq h_{k}<q^{\rho}}S_{10}(h_{0},\dots,h_{k-1},h_{k}).
\end{align*}
\end{proof}
\noindent We come back to \eqref{estimateS9}. We use Cauchy--Schwarz inequality and Lemma~\ref{Perio}
to obtain
\begin{align*}
    \sum\limits_{0 \leq h_{0},\dots,h_{k-1}<q^{\rho}}|S_{10}(h_{0},\dots, h_{k-1})|&\leq  q^{k\rho/2} \left(\sum\limits_{0 \leq h_{0},\dots,h_{k-1}<q^{\rho}}|S_{10}(h_{0},\dots, h_{k-1})|^2\right)^{1/2}.\\
    &\leq q^{k\rho/2} \left(\sum\limits_{0 \leq h_{0},\dots,h_{k-1},h_k<q^{\rho}}S_{10}(h_{0},\dots, h_{k-1},h_k)\right)^{1/2}.
\end{align*}

Thus, by Theorem~\ref{NDGG} (for $k+1$ instead of $k$), we get
\begin{equation}
\label{theorembound}
 \sum\limits_{0 \leq h_{1},\dots,h_{k-1}<q^{\rho}}|S_{10}(h_{1},\dots, h_{k-1})|\leq q^{k\rho/2}\times q^{((k+2)\rho-\eta_{0}\rho)/2}=q^{k\rho +\rho-\rho \eta_0/2}.   
\end{equation}

Summing up, we finally get\footnote{By means of Theorem~\ref{NDGG} and Lemma~\ref{Perio} we succeed to compensate the term $q^{3\gamma (k-1)}$ by $q^{-\rho \eta_{0}/2}$. Indeed, Lemma~\ref{Perio} allows us to remove the absolute value without losing the full summation over $\{0,\dots,q^{\rho}-1\}$, which is necessary to apply Theorem~\ref{NDGG}.}
\begin{align*}
    S_{9} &\leq q^{3\gamma(k-1)+\nu+1-\rho}\sum\limits_{0 \leq h_{1},\dots,h_{k-1}<q^{\rho}}|S_{10}(h_{1},\dots, h_{k-1})|+q^{\nu+1+k\rho-3\gamma\log_{q}(P^{-}(q))}\\
    &\ll q^{3\gamma(k-1)+\nu+1+k\rho-\rho\eta_{0}/2}+q^{\nu+1+k\rho-3\gamma\log_{q}(P^{-}(q))}.\\
\end{align*}
Since $H_{i}=q^{\rho}$ for all $1 \leq i \leq k-1$ (cf. \eqref{exprHi}), we get
\begin{align*}
 \dfrac{1}{H_{0}\cdots H_{k-1}q^{\nu+1+\rho}T}\sum_{m=q^{\nu}}^{q^{\nu+1}}&\sum_{\substack{1 \leq h_{0}<H_{0}\\ 0\leq h_{i}<H_{i},\ 1 \leq i \leq k-1}}\sum_{t \in \Lambda}|S_{8}|\\
 & \leq \dfrac{1}{H_{0}\cdots H_{k-1}q^{\nu+1+\rho}T}\sum_{1 \leq h_{0}<H_{0}}\sum_{t \in \Lambda}S_{9}\\
 &\ll q^{3\gamma(k-1)-\rho\eta_{0}/2}+q^{-3\gamma \log_{q}(P^{-}(q))}. 
\end{align*}

Recall that (cf. Lemma~\ref{6.6})
\begin{align*}
\Big|\dfrac{S_{0}(N,q^{\nu},\xi)}{q^{\nu+1}N}\Big|^{2^{k}}
    &\leq 2^{k-1}(E_{0}+\dots+E_{k-1})+\dfrac{2^{2^{k+2}-2}}{T}+\dfrac{2^{2^{k+2}-2}}{H_{0}\cdots H_{k-1}q^{\nu+1+\rho}T}\sum_{m=q^{\nu}}^{q^{\nu+1}}\sum_{\substack{1 \leq h_{0}<H_{0}\\1 \leq i \leq k-1,\ 0\leq h_{i}<H_{i}}}\sum_{t \in \Lambda}|S_{8}|\\
    &+2^{2^{k+2}-2}O\left(\dfrac{T}{q^{\nu+1-\rho}}\sum_{m=q^{\nu}}^{q^{\nu+1}}D_{N}\left(\dfrac{m}{q^{\lambda}}\right)+\dfrac{1}{q^{\nu+1}}\sum_{m=q^{\nu}}^{q^{\nu+1}}\sum_{i=2}^{k-1}D_{N}\left(\dfrac{m}{q^{i\mu}}\right)+\dfrac{H_{1}+\dotsb+H_{k-1}}{q^{\sigma}}\right).
\end{align*}
Until the end of the proof, we suppose
\begin{equation}
    \label{minonumu}
    \nu +1 \geq k\mu.
\end{equation}
Using Proposition~\ref{DiscRes} we have, for all $2 \leq i \leq k-1$,
\begin{align*}
\sum_{m=q^{\nu}}^{q^{\nu+1}}D_{N}\left(\dfrac{m}{q^{i\mu}}\right)&\leq
q^{\nu+1-i\mu}\sum_{m=0}^{q^{i \mu}}D_{N}\left(\dfrac{m}{q^{i\mu}}\right)\\
&\ll q^{\nu+1-i\mu}\dfrac{N+q^{i\mu}}{N}(\log^{+}N)^{2}
\\
&\ll q^{\nu+1}(\log^{+}N)^{2}\left(\dfrac{1}{N}+\dfrac{1}{q^{2\mu}}\right).
\end{align*}
Moreover,
$$
\sum_{m=q^{\nu}}^{q^{\nu+1}}D_{N}\left(\dfrac{m}{q^{\lambda}}\right) \leq \sum_{m=0}^{q^{\lambda}}D_{N}\left(\dfrac{m}{q^{\lambda}}\right)
\ll \dfrac{N+q^{\lambda}}{N}(\log^{+}N)^{2}=q^{\lambda}(\log^{+}N)^{2}\left(\dfrac{1}{N}+\dfrac{1}{q^{\lambda}}\right).
$$
According to $(\ref{defLambda})$, we have $\lambda>\nu+1$.

Finally, using \eqref{exprE0}, \eqref{exprEi} and \eqref{exprHi}, we have
\begin{align}
\label{Equafinale}
    \nonumber \Big|\dfrac{S_{0}(N,q^{\nu},\xi)}{q^{\nu+1}N}\Big|^{2^{k}} &\ll 2^{k}\left(\dfrac{1}{H_{0}}+\dfrac{H_{0}}{q^{\lambda-\nu}}+\dfrac{H_{0}}{N}+k2^{2^{k+1}}\dfrac{q^{\rho+2\mu+3\sigma}}{N}\right)+\dfrac{2^{2^{k+2}}}{T}\\
    &+2^{2^{k+2}}(q^{-3\gamma\log_{q}P^{-}(q)}+q^{3\gamma (k-1) -\eta_{0}\rho/2 })\\ \nonumber 
    &+2^{2^{k+2}}\left( Tq^{\lambda-\nu-1+\rho}(\log^{+}N)^{2}\left(\dfrac{1}{N}+\dfrac{1}{q^{2\mu}}\right)+\dfrac{k}{q^{\sigma-\rho}}\right).
\end{align}

It remains to choose the parameters introduced in the proof in order to get Theorem~\ref{Thm1}  from~\eqref{Equafinale}. This is the aim of the next section. Since the principle of convexity applied for $H_{0}$ in~\eqref{Equafinale} gives similar results as in~\cite[p.2581]{Spi},  we follow the choice of Spiegelhofer. 
 
\medskip

In order to define $\mu$, we recall that, by hypothesis (see beginning of Section \ref{Theo2.5}), we have
\begin{equation}
N^{\rho_{1}} < q^{\nu} \leq qN^{\rho_{2}}.
\end{equation}

We set $k=3(\lfloor \rho_{2}\rfloor+1) \geq 3$ 
and define
\begin{equation*}
\mu=\Big\lfloor\dfrac{\nu-1}{k+1/8}\Big\rfloor,
\end{equation*}
and observe that $\mu \leq\nu/k$ (thus, this choice respects \eqref{minonumu}). 
Furthermore, we set
\begin{equation}
\sigma=\Big\lfloor\dfrac{\mu}{4}\Big\rfloor,\qquad  \Tilde{\rho}=\nu-k\mu.
\end{equation}
With the definition of $\Tilde{\rho}$, we set
\begin{equation}
\gamma=\Big\lfloor\dfrac{\eta_{0}\Tilde{\rho}}{12(k-1)}\Big\rfloor.
\end{equation}
Then, with the definition of $\gamma$, we define
\begin{equation*}
    T=q^{\gamma},\qquad H_{0}=\lfloor q^{\gamma/4}\rfloor.
\end{equation*}
With this definition of $H_0$, the condition~\eqref{majoh0} is satisfied for large enough $\gamma$ (we refer the reader to Corollary~\ref{MajoGR}). 
Moreover, we set
\begin{equation*}
\lambda=\nu+\Big\lfloor\dfrac{\gamma}{2}\Big\rfloor.
\end{equation*}
Condition \eqref{defLambda} is satisfied, since $\gamma \rightarrow +\infty$.
We recall again that
\begin{equation}
\rho=\lambda-k\mu.
\end{equation}
The reader will find the somewhat tedious calculations in Annexe~\ref{choosingsection} leading to the following estimate:

\begin{equation*}
\Big|\dfrac{S_{0}(N,q^{\nu},\xi)}{q^{\nu+1}N}\Big|^{2^{k}} \ll N^{-\frac{\rho_{1}\eta_{0}}{(384k+48)(k-1)}}+N^{-\frac{3\eta_{0}\log_{q}P^{-}(q)}{(k-1)(96k+12)}\rho_{1}}.
\end{equation*}
By noticing that $384k+48=4\times(96k+12)$, we conclude that
\begin{equation*}
\dfrac{S_{0}(N,q^{\nu},\xi)}{q^{\nu+1}N} \ll N^{-\eta_{1}},
\end{equation*}
where
\begin{equation}
\label{eta1def}
    \eta_{1}=\frac{\rho_{1}\eta_{0}\min\left(1/4,3\log_q\left(P^{-}(q)\right)\right)}{2^{k}(k-1)(96k+12)}.
\end{equation}
Now, we bound $\eta_{0}$ 

given by Theorem~\ref{NDGG} from below  (we recall that we applied this theorem for $k+1$ instead of $k$).
Using the fact that 
$$  k+1=3\left(\lfloor \rho_{2}\rfloor+1\right)+1
    \leq 3\rho_{2}+4,
$$
we have
$$
    K=\Big\lfloor\dfrac{\log(k+1)}{\log(q)}\Big\rfloor+1\\
     \leq \dfrac{\log((3\rho_{2}+4)q)}{\log(q)}.
$$
Hence, we get
\begin{align*}
    K+(k+2)b &\leq \dfrac{\log((3\rho_{2}+4)q)}{\log(q)}+b(3\rho_{2}+5)\\
    &\leq \dfrac{\log((3\rho_{2}+4)q)+b\log(q)(3\rho_{2}+5)}{\log(q)}.
\end{align*}

Therefore, we can bound from below $\eta_{0}$ by
\begin{align*}
    \eta_{0}&=\dfrac{1}{\log(q)(K+(k+2)b)q^{(k+2)(K+(k+2)b)}}\\
    &\geq \dfrac{1}{\log((3\rho_2+4)q)+b\log(q)(3\rho_2+5)}\exp\left(-(3\rho_2+5)\left(\log((3\rho_2+4)q)+b\log(q)(3\rho_2+5)\right)\right).
\end{align*}
Therefore, injecting all theses inequalities into \eqref{eta1def}, we have
\begin{align*}
    \eta_{1} &\geq \dfrac{\rho_{1}\eta_{0}\min\left(1/4,3\log_q\left(P^{-}(q)\right)\right)}{8^{1+\rho_{2}}(288\rho_{2}+300)(3\rho_{2}+2)}\\
    &\geq \dfrac{\rho_{1}\min\left(1/4,3\log_q\left(P^{-}(q)\right)\right)}{8^{1+\rho_{2}}(288\rho_{2}+300)(3\rho_{2}+2)\left(\log((3\rho_{2}+4)q)+b\log(q)(3\rho_{2}+5)\right)}\\
    &\times\exp\left({-(3\rho_{2}+5)\left(\log((3\rho_{2}+4)q)+(3\rho_{2}+5)b\log(q)\right)}\right)\\
    &=\eta(\rho_{1},\rho_{2}).
\end{align*}
This ends the proof of Theorem~\ref{Thm2}.

\section{Proof of Theorem~\ref{Thm1}}\label{secpreuveThm1}
Let $ 0< \delta_{1}\leq\delta_{2}<1$ be two real numbers and $0<\varepsilon_{1},\varepsilon_{2} <1$. Let $D$ be an integer such that 
$
x^{\delta_{1}} \leq D \leq x^{\delta_{2}}
$
and $r \in \{0,\dots,m-1\}$. Let $y,z$ be two real numbers such that $0\leq y<z$ and $z-y \leq x$. Recall \eqref{defN},
\begin{equation*}
    N_{y,z}(a,b;r,m)=\Big|\{y \leq n <z\ :\ n \equiv r \Mod{m},\quad s_{q}(n) \equiv a \Mod{b}\}\Big|.
\end{equation*}
We have
\begin{equation*}
   N_{y,z}(a,b;r,m)-\dfrac{z-y}{bm}= \sum_{\substack{y\leq n <z,\\n \equiv r \Mod{m}}}\mathbb{1}_{b\Z}(s_{q}(n)-a)-\dfrac{z-y}{bm}.
\end{equation*}
Writing $n=r+km$ and  detecting the congruence with an exponential sum, we get

\begin{align}
\label{EgalNyz}
   N_{y,z}(a,b;r,m)-\dfrac{z-y}{bm}&=\sum_{\frac{y-r}{m} \leq k < \frac{z-r}{m}}\mathbb{1}_{b\Z}(s_{q}(r+km)-a)-\dfrac{z-y}{bm} \nonumber\\
   &=\dfrac{1}{b}\sum_{\ell=1}^{b-1}\;\sum_{\frac{y-r}{m} \leq k < \frac{z-r}{m}}\e\left(\dfrac{\ell}{b}(s_{q}(r+km)-a)\right)+O(1).
\end{align}
With the change of variable
\begin{equation*}
    k'=k-\Big\lfloor \dfrac{y-r}{m}\Big\rfloor,
\end{equation*}
we get
\begin{equation*}
N_{y,z}(a,b;r,m)-\dfrac{z-y}{bm}=\dfrac{1}{b}\sum_{\ell=1}^{b-1}\sum_{0\leq k < \frac{z-y}{m}}\e\left(\dfrac{\ell}{b}\left(s_{q}\left(r+m\Big\lfloor\dfrac{y-r}{m}\Big\rfloor+km\right)-a\right)\right)+O(1).
\end{equation*}
By taking absolute values and applying the triangle inequality we have
\begin{align*}
    \Big|N_{y,z}(a,b;r,m)-\dfrac{z-y}{bm}\Big| &\leq \dfrac{1}{b}\sum_{\ell=1}^{b-1}\Big|\sum_{k \leq \frac{z-y}{m}}\e\left(\dfrac{\ell}{b}\left(s_{q}\left(r+m\Big\lfloor\dfrac{y-r}{m}\Big\rfloor+km\right)-a\right)\right)\Big|+O(1)\\
    &\leq \dfrac{1}{b}\sum_{\ell=1}^{b-1} \max_{u \leq x}\;\max_{r' \geq 0}\Big|\sum_{k \leq \frac{u}{m}}\e\left(\dfrac{\ell}{b}s_{q}(r'+km)\right)\Big|+O(1).
\end{align*}
Thus, by taking the maximum over $r$, we deduce that
\begin{equation*}
    \max_{r \geq 0}\Big|N_{y,z}(a,b;r,m)-\dfrac{z-y}{bm}\Big|\leq \dfrac{1}{b}\sum_{\ell=1}^{b-1} \max_{u \leq x}\;\max_{r \geq 0}\Big|\sum_{k \leq \frac{u}{m}}\e\left(\dfrac{\ell}{b}s_{q}(r+km)\right)\Big|+O(1).
\end{equation*}
Taking the maximum over $y,z$ and summing over $m \in \{D+1,\dots,qD\}$, we have
\begin{equation}
\label{SumoverD}
\sum_{D<m\leq qD}\;\max_{\substack{y,z\\0\leq y<z\\z-y\leq x}}\max_{r \geq 0}\Big|N_{y,z}(a,b;r,m)-\dfrac{z-y}{bm}\Big| \leq \dfrac{1}{b}\sum_{\ell=1}^{b-1}\sum_{D<m\leq qD} \max_{u \leq x}\;\max_{r \geq 0}\Big|\sum_{k \leq \frac{u}{m}}\e\left(\dfrac{\ell}{b}s_{q}(r+km)\right)\Big|+O(D).  
\end{equation}
For $D<m\leq qD$, we apply Lemma~\ref{Lemma37}, with $x=0,y=u/m$ and $z=x/D$, and we get
\begin{equation*}
\Big|\sum_{k \leq \frac{u}{m}}\e\left(\dfrac{\ell}{b}s_{q}(r+km)\right)\Big| \leq \int_{0}^{1}\min\left(\dfrac{u}{m}+1,\|\xi\|^{-1}\right)\Big|\sum_{k \leq \frac{x}{D}}\e\left(\dfrac{\ell}{b}s_{q}(r+km)\right)\e(n\xi)\Big|d\xi.
\end{equation*}

We use \begin{equation*}
    \min\left(\dfrac{u}{m}+1,\|\xi\|^{-1}\right) \leq \min\left(\dfrac{x}{D}+1,\|\xi\|^{-1}\right)
\end{equation*}
to get
\begin{equation*}
\max_{u \leq x}\Big|\sum_{k \leq \frac{u}{m}}\e\left(\dfrac{\ell}{b}s_{q}(r+km)\right)\Big| \leq \int_{0}^{1}\min\left(\dfrac{x}{D}+1,\|\xi\|^{-1}\right)\Big|\sum_{k \leq \frac{x}{D}}\e\left(\dfrac{\ell}{b}s_{q}(r+km)\right)\e(n\xi)\Big|d\xi.
\end{equation*}

Injecting this in~(\ref{SumoverD}), we deduce that
\begin{align*}
    \sum_{D<m\leq qD}\; \max_{\substack{y,z\\0\leq y<z\\z-y\leq x}}&\max_{r \geq 0}\Big|N_{y,z}(a,b;r,m)-\dfrac{z-y}{bm}\Big|\\
    &\leq \dfrac{1}{b}\sum_{1 \leq \ell<b}\int_{0}^{1}  \min\left(\dfrac{x}{D}+1,\|\xi\|^{-1}\right)\sum\limits_{D<m\leq qD}\max_{r \geq 0}\Big|\sum_{k \leq \frac{x}{D}}\e\left(\dfrac{\ell}{b}s_{q}(r+km)\right)\e(n\xi)\Big|d\xi+O(D).
 \end{align*}
Using Theorem~\ref{Thm2}, with $\rho_{1}=\dfrac{\delta_{1}}{1-\delta_{1}}$ and $\rho_{2}=\dfrac{\delta_{2}}{1-\delta_{2}}$ there exists $C=C(\eta)>0$ such that
\begin{equation*}
    \sum_{D<m\leq qD} \max_{r \geq 0} \Big|\sum_{k \leq \frac{x}{D}}\e\left(\dfrac{\ell}{b}s_{q}(r+km)\right)\e(n\xi)\Big| \leq CD\left(\dfrac{x}{D}\right)^{1-\eta},
\end{equation*}
with $\eta=\eta(\rho_{1},\rho_{2})$ defined in Theorem~\ref{Thm2}. Hence, it follows that
\begin{equation*}
 \sum_{D<m\leq qD} \max_{\substack{y,z\\0\leq y<z\\z-y\leq x}}\max_{r \geq 0}\Big|N_{y,z}(a,b;r,m)-\dfrac{z-y}{bm}\Big|\leq CD\left(\dfrac{x}{D}\right)^{1-\eta} \int_{0}^{1}  \min\left(\dfrac{x}{D}+1,\|\xi\|^{-1}\right)d\xi+O(D).
 \end{equation*}    

Since $D \leq x^{\delta_{2}}<x$, we have $x/D>1$ and
using 
\begin{equation*}
\int_{0}^{1}\min\left(A,\|\xi\|^{-1}\right)d\xi \ll \log(A),
\end{equation*}
there exists $C_{1}=C_1(\rho_1,\rho_2)>0$ such that 

\begin{equation}
\label{majoC'}
\sum_{D<m\leq qD} \max_{\substack{y,z\\0\leq y<z\\z-y\leq x}}\max_{r \geq 0}\Big|N_{y,z}(a,b;r,m)-\dfrac{z-y}{bm}\Big|\leq  C_{1}D\left(\dfrac{x}{D}\right)^{1-\eta}\log\left(\dfrac{x}{D}\right)+O(D).
\end{equation}

Moreover, by hypothesis, we have
\begin{equation*}
    x^{\delta_{1}} \leq D \leq x^{\delta_{2}}.
\end{equation*}
Injecting this in~(\ref{majoC'}), we have for any $0<\varepsilon_{2}<1$,
\begin{align*}
\sum_{D<m\leq qD} \max_{\substack{y,z\\0\leq y<z\\z-y\leq x}}\max_{r \geq 0}\Big|N_{y,z}(a,b;r,m)-\dfrac{z-y}{bm}\Big|&=C_{1}D^{\eta}x^{1-\eta}(\log(x)-\log(D))+O(D)\\
&\leq C_{1}x^{\delta_{2}\eta}x^{1-\eta}(1-\delta_{1})\log(x)+O(D)\\
&\leq C_{2}x^{1-(1-\delta_{2})\eta/(1+\varepsilon_{2})}+O(D)\\
&\leq C_{3}x^{1-(1-\delta_{2})\eta/(1+\varepsilon_{2})},
\end{align*}
since $D \leq x^{\delta_{2}}$ and $1-(1-\delta_{2})\eta/(1+\varepsilon_2)\geq 1-(1-\delta_2) = \delta_{2}$, because $\delta_{2}, \eta/(1+\varepsilon_2)<1$. 
\\In order to end the proof of Theorem~\ref{Thm1}, we are going to bound from below $(1-\delta_{2})\eta/(1+\varepsilon_2)$. We first use the explicit value of $\eta=\eta(\rho_{1},\rho_{2})$,
\begin{align}
    \eta&=\dfrac{\rho_{1}\min\left(1/4,3\log_q\left(P^{-}(q)\right)\right)}{8^{1+\rho_{2}}(288\rho_{2}+300)(3\rho_{2}+2)\left(\log((3\rho_{2}+4)q)+b\log(q)(3\rho_{2}+5)\right)} \nonumber\\
    &\times\exp\left({-(3\rho_{2}+5)\left(\log((3\rho_{2}+4)q)+(3\rho_{2}+5)b\log(q)\right)}\right) \nonumber.
\end{align}

Using the definition of $\rho_{1}=\dfrac{\delta_{1}}{1-\delta_{1}}$ and $\rho_{2}=\dfrac{\delta_{2}}{1-\delta_{2}}$, we have
\begin{align*}
\dfrac{(1-\delta_{2})\eta}{1+\varepsilon_{2}}&=\dfrac{\delta_{1}(1-\delta_{2})^{2}\min\left(1/4,3\log_q\left(P^{-}(q)\right)\right)}{(1-\delta_{1})8^{(1-\delta_{2})^{-1}}(300-12\delta_{2})(2+\delta_{2})\left(\log((4-\delta_2) q/(1-\delta_{2}))+b\log(q)(5-2\delta_{2})/(1-\delta_2)\right)(1+\varepsilon_{2})}\\
    &\quad \times \exp\left({-\dfrac{5-2\delta_{2}}{1-\delta_{2}}\left(\log(q(4-\delta_2)(1-\delta_{2})^{-1})+\dfrac{5-2\delta_{2}}{1-\delta_{2}}b\log(q)\right)}\right).
\end{align*}
Since $1-\delta_1 \leq 1$, $300-12\delta_2 \leq 300$, $2+\delta_2 \leq 3$, $4-\delta_{2} \leq 4$, $5-2\delta_{2} \leq 5$ and $1+\varepsilon_2 \leq 2$, in order to simplify the expression, we remove the dependence of the numerators of $\delta_2$ of the appearing fractions  at the cost of a slightly less precise bound:
\begin{align*}
\dfrac{(1-\delta_{2})\eta}{1+\varepsilon_{2}}&\geq\dfrac{\delta_{1}(1-\delta_{2})^{2}\min\left(1/4,3\log_q\left(P^{-}(q)\right)\right)}{1800\times 8^{(1-\delta_{2})^{-1}}\left(\log(4q/(1-\delta_{2}))+5b\log(q)/(1-\delta_2)\right)}\\
    &\quad \times \exp\left({-\dfrac{5}{1-\delta_{2}}\left(\log(4q(1-\delta_{2})^{-1})+\dfrac{5b\log(q)}{1-\delta_{2}}\right)}\right)\\
    &=:\eta(\delta_{1},\delta_2),
\end{align*}
as stated. This ends the proof of Theorem~\ref{Thm1}.

\section{Proof of Theorem~\ref{Thm0}}\label{secthm0}

Let $0<\varepsilon<1$. We set $\delta_{2}=1-\varepsilon$ and $\delta_{1}=\dfrac{\varepsilon}{2}$ (we note that the condition $\delta_{1} \leq \delta_{2}$ implies $\varepsilon \leq 2/3$.)

For $m \in [1,x^{1-\varepsilon}]\cap \N$, we set
\begin{equation*}
    E(m)=\max_{\substack{y,z\\0\leq y<z\\z-y\leq x}}\max_{r \geq 0}\Big|N_{y,z}(a,b;r,m)-\dfrac{z-y}{bm}\Big|.
\end{equation*}
We split 
\begin{equation*}
    \sum_{1 \leq m \leq x^{1-\varepsilon}}E(m)=\sum_{1 \leq m < x^{\delta_{1}}}E(m)+\sum_{x^{\delta_{1}} \leq m \leq x^{\delta_{2}}}E(m).
\end{equation*}

According to Theorem~\ref{Gelfond} we have for $y,z$ such that $0 \leq y<z$ and $z-y \leq x$, 
\begin{align*}
N_{y,z}(a,b;r,m)&=\Big|\{y \leq n <z\ :\ n \equiv r \Mod{m},\quad s_{q}(n) \equiv a \Mod{b}\}\Big|\\
&=\dfrac{z-y}{bm}+O\left((z-y)^{\lambda}\right).
\end{align*}
Since $z-y \leq x$, we have
\begin{equation*}
N_{y,z}(a,b;r,m)-\dfrac{z-y}{bm}=O\left(x^{\lambda}\right),
\end{equation*}
where the implied constant depends on $b$ and $q$. We deduce that 
\begin{equation*}
\sum\limits_{1 \leq m < x^{\delta_{1}}}E(m)=O(x^{\delta_{1}+\lambda}).
\end{equation*}
Concerning the second sum, we split the interval into $q$-adic intervals
\begin{equation*}
    [x^{\delta_{1}},x^{\delta_{2}}] \subset \biguplus_{\alpha=A_{1}}^{A_{2}}[q^{\alpha},q^{\alpha+1}],
\end{equation*}
where
\begin{equation*}
    A_{1}=\Big \lfloor\dfrac{\delta_{1}\log(x)}{\log(q)}\Big\rfloor, \qquad A_{2}=\Big \lfloor\dfrac{\delta_{2}\log(x)}{\log(q)}\Big\rfloor-1.
\end{equation*}
Thus,
\begin{equation*}
 \sum_{x^{\delta_{1}} \leq m \leq x^{\delta_{2}}}E(m) \leq \sum_{\alpha=A_{1}}^{A_{2}}\;\sum_{q^{\alpha}\leq m \leq q^{\alpha+1}}E(m).   
\end{equation*}
Moreover, for $\alpha \in \{A_{1},\dots,A_{2}\}$, and for large $x$, we have 
\begin{equation*}
    x^{\delta_{1}/2}\leq q^{\alpha} \leq x^{\delta_{2}}.
\end{equation*}
Applying Theorem~\ref{Thm1} for $D=q^{\alpha}$, with $\delta_{1}/2$ and $\delta_{2}$ in place of $\delta_{1}$ and $\delta_{2}$,
\begin{equation}
\label{Splitsum}
\sum_{1 \leq m \leq x^{1-\varepsilon}}E(m)=\sum_{1 \leq m < x^{\delta_{1}}}E(m)+\sum_{x^{\delta_{1}} \leq m \leq x^{\delta_{2}}}E(m)
\leq O\left(x^{\delta_1+\lambda}\right)+Cx^{1-\frac{\eta}{1+\varepsilon_{3}}},
\end{equation}
for all $0<\varepsilon_{3}<1$ and for $\eta=\eta(\delta_{1},\delta_{2})$ defined in Theorem~\ref{Thm1}. Recalling $\delta_2=1-\varepsilon$ and $\delta_1=\varepsilon/2$, we get with (\ref{defetadd}),
\begin{align}
\label{valeuretae3}
    \dfrac{\eta}{1+\varepsilon_3}&:=\dfrac{(\delta_{1}/2)(1-\delta_{2})^{2}\min\left(1/4,3\log_q\left(P^{-}(q)\right)\right)}{1800\times 8^{(1-\delta_{2})^{-1}}\left(\log(4q/(1-\delta_{2}))+5b\log(q)/(1-\delta_2)\right)(1+\varepsilon_3)} \nonumber \\
    &\quad\times \exp\left({-\dfrac{5}{1-\delta_{2}}\left(\log(4q(1-\delta_{2})^{-1})+\dfrac{5b\log(q)}{1-\delta_{2}}\right)}\right) \nonumber\\
    &=\dfrac{\varepsilon^{3}\min\left(1/4,3\log_q\left(P^{-}(q)\right)\right)}{3600\times 8^{1/\varepsilon}\left(\log(4q/\varepsilon)+5b\log(q)/\varepsilon\right)(1+\varepsilon_3)}\\
    &\quad \times \exp\left({-\dfrac{5}{\varepsilon}\left(\log(4q/\varepsilon)+\dfrac{5b\log(q)}{\varepsilon}\right)}\right)\nonumber.
\end{align}
In particular, $\dfrac{\eta}{1+\varepsilon_3}\leq \varepsilon$. Indeed, since $0<\varepsilon<1$, we have $\log(4q/\varepsilon)>0$ and $5b\log(q)/\varepsilon > 5b\log(q)>1$. Thus,
\begin{equation*}
    3600\times 8^{1/\varepsilon}(\log(4q/\varepsilon)+5b\log(q)/\varepsilon)>1
\end{equation*}
and $\varepsilon^3<\varepsilon$. Therefore, $$\dfrac{\eta}{1+\varepsilon_3}\leq \varepsilon.$$
Now, we want to compare the two exponents appearing in~\eqref{Splitsum}. We write
\begin{equation*}
   \delta_1+\lambda-\left(1-\dfrac{\eta}{1+\varepsilon_{3}}\right)\leq \dfrac{\varepsilon}{2}+\lambda-1+\varepsilon.
\end{equation*}
For $\varepsilon<\dfrac{2}{3}(1-\lambda)$, the last quantity is negative. Thus, we deduce that there exists $C'>0$ such that
\begin{equation*}
 \sum_{1 \leq m \leq x^{1-\varepsilon}}E(m) \leq C'x^{1-\frac{\eta}{1+\varepsilon_{3}}}.
\end{equation*}

In order to have the $\eta$ announced in Theorem \ref{Thm0}, we use the equation~\eqref{valeuretae3} and we use the fact that $1+\varepsilon_{3} \leq 2$.

Finally,
\begin{align*}
    \dfrac{\eta}{1+\varepsilon_3}&\geq\dfrac{\varepsilon^{3}\min\left(1/4,3\log_q\left(P^{-}(q)\right)\right)}{7200\times 8^{1/\varepsilon}\left(\log(4q/\varepsilon)+5b\log(q)/\varepsilon\right)}\\
    &\times \exp\left({-\dfrac{5}{\varepsilon}\left(\log(4q/\varepsilon)+\dfrac{5b\log(q)}{\varepsilon}\right)}\right)
\end{align*}
and we can therefore take
\begin{equation*}
  \eta\left(\varepsilon\right)=\dfrac{\varepsilon^{3}\min\left(1/4,3\log_q\left(P^{-}(q)\right)\right)}{7200\times 8^{1/\varepsilon}\left(\log(4q/\varepsilon)+5b\log(q)/\varepsilon\right)}\times \exp\left({-\dfrac{5}{\varepsilon}\left(\log(4q/\varepsilon)+\dfrac{5b\log(q)}{\varepsilon}\right)}\right)
\end{equation*}
as stated. This ends the proof of Theorem~\ref{Thm0}.
\begin{rem}
Instead of applying Theorem~\ref{Gelfond} for the first sum of~\eqref{Splitsum} we could also use a result on average in the spirit of Theorem~\ref{ThmFM2}. However, it requires a generalization of Theorem~\ref{ThmFM2} to any base $q$. This would give us another condition for $\varepsilon$ but the explicit value for $\eta$ would not change.

\end{rem}
\newpage
\section*{Annexes}
\appendix

\begin{appendices}
\section{Choosing the parameters}\label{choosingsection}
First, let us give an inventory of all the parameters used along this paper with their roles and restrictions. 
\medskip

\noindent \underline{Parameters introduced during the iterations of the van der Corput inequality}

\begin{itemize}
    \item $k\geq 3$ the number of iterations of the van der Corput inequality (Subsection \ref{furtheriterations}),
    \item $H_{0} \leq q^{k\mu}$ appearing in the first iteration of the van der Corput inequality (Lemma  \ref{VDC1}). The bound appears in~\eqref{majoh0},
    \item $\lambda > \nu+1$ used to apply the carry propagation lemma (\ref{defLambda}),
    \item $H_{1}=\cdots=H_{k-1}=q^{\rho}$ appearing in Subsection \ref{furtheriterations} (see \eqref{exprHi}); we will discuss the size of $E_i$ and \eqref{sumEi} later (see \eqref{twofoldbound}).
\end{itemize}

\underline{Parameters used to define the shifts $K_{1},\dots, K_{k-1}$}

\begin{itemize}
    \item $\mu>0$ such that $\nu+1 \geq k\mu$~(see (\ref{minonumu})),
    \item $\rho+1<\sigma<\mu$  (see  \eqref{rhosigmaj} and \eqref{defSigma}).
\end{itemize}

\underline{Parameters introduced during the various digit shifts}

\begin{itemize}
    \item $\rho=\lambda-k\mu$  (see \eqref{defRho}),
    \item $T \geq 2$ 
    defined in Lemma \ref{6.6},
    \item $0<\gamma\leq \rho$ used to define the set $B$ (see \eqref{defEnsB}).
\end{itemize}

In order to define $\mu$, we recall that, by hypothesis (see beginning of Section \ref{Theo2.5}), we have
\begin{equation}
\label{HypoEncadre}
N^{\rho_{1}} < q^{\nu} \leq qN^{\rho_{2}}.
\end{equation}
We set $k=3(\lfloor \rho_{2}\rfloor+1) \geq 3$ 
and define
\begin{equation*}
\mu=\Big\lfloor\dfrac{\nu-1}{k+1/8}\Big\rfloor,
\end{equation*}
and observe that $\mu \leq\nu/k$ (thus, this choice respects \eqref{minonumu}). 
Furthermore, we set
\begin{equation}\label{sigmavalue}
\sigma=\Big\lfloor\dfrac{\mu}{4}\Big\rfloor,\qquad  \Tilde{\rho}=\nu-k\mu.
\end{equation}
With the definition of $\Tilde{\rho}$, we set
\begin{equation}\label{gammavalue}
\gamma=\Big\lfloor\dfrac{\eta_{0}\Tilde{\rho}}{12(k-1)}\Big\rfloor.
\end{equation}
Then, with the definition of $\gamma$, we define
\begin{equation*}
    T=q^{\gamma},\qquad H_{0}=\lfloor q^{\gamma/4}\rfloor.
\end{equation*}
Moreover, we set
\begin{equation}
\label{defLambda2}
\lambda=\nu+\Big\lfloor\dfrac{\gamma}{2}\Big\rfloor.
\end{equation}
Condition \eqref{defLambda} is satisfied, since $\gamma \rightarrow +\infty$.
We recall again that
\begin{equation}\label{rhovalue}
\rho=\lambda-k\mu.
\end{equation}

\bigskip
We first give some asymptotic equivalents of our parameters which will be of importance to prove that our choice of parameters is admissible. Since $\nu\to +\infty$, these equivalents will be useful to determine the rate of convergence for every term in~(\ref{Equafinale}). This will be sufficient to conclude.
\begin{lemme}
\label{EquivPara}
For $\mu \rightarrow +\infty$, we have the following equivalents:
\begin{enumerate}
    \item[(1)] 
    \begin{equation}
    \label{rhotildapp}
        \Tilde{\rho} \sim \dfrac{\mu}{8},
    \end{equation}
    \item[(2)] 
    \begin{equation}
    \label{gammaapp}
         \gamma \sim \dfrac{\eta_{0}\mu}{96(k-1)},
    \end{equation}
    \item[(3)]
    \begin{equation}
    \label{rhoapp}
      \rho \sim \dfrac{\mu}{8}+\dfrac{\eta_{0}\mu}{192(k-1)}.
    \end{equation}
\end{enumerate}
\end{lemme}
\begin{proof}

For~(\ref{rhotildapp}), since $k$ is constant, we have

$$
\Tilde{\rho}=\nu-k\Big\lfloor \dfrac{\nu-1}{k+1/8}\Big\rfloor
    \sim \nu-k \dfrac{\nu-1}{k+1/8}
    \sim \dfrac{1}{8}\left(\dfrac{\nu-1+1+8k}{k+1/8}\right)
    \sim \dfrac{\mu}{8}.
$$
For~(\ref{gammaapp}), we use the definition of $\gamma$ and \eqref{rhotildapp}.
 We get 
$$
\gamma
\sim \dfrac{\eta_{0}\Tilde{\rho}}{12(k-1)}
\sim \dfrac{\eta_{0}\mu}{96(k-1)}.
$$
As for~(\ref{rhoapp}), we use~(\ref{rhotildapp}),~(\ref{gammaapp}) and the definition~\eqref{defLambda2} of $\lambda$, to get 
$$
    \rho=\lambda-k\mu
    =\Tilde{\rho}+\Big\lfloor\dfrac{\gamma}{2}\Big\rfloor
    \sim \dfrac{\mu}{8}+\dfrac{\eta_{0}\mu}{192(k-1)}.
$$
\end{proof}

Since $k\geq 3$ and $0<\eta_{0}<1$ we get the following explicit inequalities.
\begin{cor}
\label{MajoGR}
For sufficiently large $\mu$,
$$\gamma \leq \mu/192\quad\text{and}\qquad \rho \leq \mu/8+\mu/384.$$ 
\end{cor}

We first note that condition~\eqref{condlaga} is satisfied. Indeed, by \eqref{sigmavalue} and \eqref{gammavalue} we have
$$
    3\gamma\leq \frac{\Tilde{\rho}}{8}<\frac{\nu}{8}<\nu+1.
$$
Moreover, the condition~\eqref{majoh0} is satisfied for $H_0$ for sufficiently large $\mu$.
In order to proceed with the proof of Theorem~\ref{Thm2}, we need bounds for $N$ in terms of $q^{\mu}$. 
\begin{lemme}
We have the following bounds for $N$ and $q^{\mu}$: 
\begin{enumerate}
    \item[(1)] \begin{equation}
    \label{MinoN}
N \geq q^{3\mu},
\end{equation}
\item[(2)] \begin{equation}
\label{MinoQmu}
q^{\mu} \gg N^{\frac{\rho_{1}}{k+1/8}},
\end{equation}
where the implied constant depends only on $k$ and $q$.
\end{enumerate}
\end{lemme}
\begin{proof}
Using (\ref{HypoEncadre}) we get
$$
    N \geq q^{\frac{\nu-1}{\rho_{2}}}
    \geq q^{\frac{\nu-1}{\lfloor\rho_{2}\rfloor+1}}
    =q^{3\frac{\nu-1}{k}}
    \geq q^{3\mu},
$$
which is~\eqref{MinoN}.
Also, again using \eqref{HypoEncadre}, we get
$N^{\rho_{1}}<q^{\nu}$ and
$$
    \dfrac{1}{q^{(k+1/8)^{-1}}}N^{\frac{\rho_{1}}{k+1/8}}<q^{\frac{\nu-1}{k+1/8}}
    < q^{\lfloor\frac{\nu-1}{k+1/8}\rfloor+1}
    =q^{\mu+1}.
$$
We conclude that
\begin{equation*}
    q^{\mu}>C N^{\frac{\rho_{1}}{k+1/8}},
\end{equation*}
with
$
    C=1/(q^{(k+1/8)^{-1}+1}).
$
\end{proof}

\bigskip

After this preliminary work, we now 
focus on bounding each term in~(\ref{Equafinale}).

\medskip

$\bullet$ \underline{Bounding $\dfrac{1}{H_{0}}$.}
\vspace{0.2cm}
\\To begin with, we note that 
$$
    H_{0}=\lfloor q^{ \frac{\gamma}{4}}\rfloor
    \leq q^{\mu/768}\leq q^{3\mu}\leq N,
$$
a condition that we imposed on $H_0$ in Lemma \ref{VDC1}. Now, we use~(\ref{gammaapp}) and~(\ref{MinoQmu}),
$$
\dfrac{1}{H_{0}}=\dfrac{1}{\lfloor q^{\gamma/4} \rfloor}
\sim q^{-\gamma/4}
\ll q^{-\frac{\eta_{0}\mu}{384(k-1)}}
\ll N^{\frac{-\rho_{1}\eta_{0}}{(384k+48)(k-1)}} .
$$

\medskip

$\bullet$ \underline{Bounding $\dfrac{H_{0}}{q^{\lambda-\nu}}$.}
\vspace{0.2cm}
$$
\dfrac{H_{0}}{q^{\lambda-\nu}}=\dfrac{\lfloor q^{ \gamma/4}\rfloor}{q^{\lfloor \gamma/2 \rfloor}}\\
\ll q^{-\gamma/4}.
$$
We note that $\dfrac{1}{H_{0}}$ and $\dfrac{H_{0}}{q^{\lambda-\nu}}$ have the same rate of convergence. 

\medskip

$\bullet$ \underline{Bounding $\dfrac{H_{0}}{N}$.}
\vspace{0.2cm}
\\We notice that $\dfrac{H_{0}}{N} <\dfrac{1}{H_{0}}$. Indeed, by~(\ref{MinoN}) we have
\begin{equation*}
    \dfrac{H_{0}}{N} \leq \dfrac{q^{\gamma/4}}{q^{3\mu}}.
\end{equation*}
Using the fact that $\mu \geq 192\gamma$  (see Corollary~(\ref{MajoGR})), we have
\begin{equation*}
    \dfrac{H_{0}}{N} \leq q^{\frac{-2303}{4}\gamma}.
\end{equation*}
The above estimates show that $E_{0}$ defined in Lemma~\ref{VDCK} is smaller than 1 for sufficiently large $\nu$.

\medskip

$\bullet$ \underline{Bounding $\dfrac{q^{\rho+2\mu+3\sigma}}{N}$.}
\vspace{0.2cm}
\\Using~(\ref{MinoN}) we have
\begin{equation*}
\dfrac{q^{\rho+2\mu+3\sigma}}{N} \leq q^{\rho-\mu+3\sigma}.
\end{equation*}
\\Furthermore, by Corollary~(\ref{MajoGR}), we get, for sufficiently large $\mu$,
\begin{equation*}
\rho-\mu+3\sigma \leq \dfrac{-335\mu}{384}+3\sigma.
\end{equation*}
By the fact that $\mu \geq 4\sigma$ (see \eqref{sigmavalue}), we find
\begin{equation*}
\rho-\mu+3\sigma \leq \dfrac{-47}{384}\mu.
\end{equation*}
The inequality~(\ref{MinoQmu}) then gives
\begin{equation}\label{twofoldbound}
\dfrac{q^{\rho+2\mu+3\sigma}}{N} \leq q^{-\frac{47}{384}\mu}
\ll N^{-\frac{47\rho_{1}}{384k+48}}.
\end{equation}
This estimate also assures that  $E_{1},\dots,E_{k-1}$ defined in~Lemma \ref{VDCK} satisfy the condition \eqref{sumEi} for sufficiently large $\nu$.

\medskip

$\bullet$ \underline{Bounding $\dfrac{1}{T}$.} 
\vspace{0.2cm}
\\As before, we use~(\ref{gammaapp}) and~(\ref{MinoQmu}), and get

$$
\dfrac{1}{T}=\dfrac{1}{q^{\gamma}} \ll q^{-\frac{\eta_{0}}{96(k-1)}\mu}
\ll N^{-\frac{\eta_{0} \rho_{1}}{96(k-1)(k+1/8)}} = N^{-\frac{\eta_{0} \rho_{1}}{(96k+12)(k-1)}} .
$$

\medskip

$\bullet$ \underline{Bounding $q^{-3\gamma\log_{q}P^{-}(q)}$.} 
\vspace{0.2cm}
\\Once again, we use~(\ref{gammaapp}) and~(\ref{MinoQmu}) and get
$$
 q^{-3\gamma\log_{q}P^{-}(q)} \ll q^{-\frac{3\eta_{0}\log_{q}P^{-}(q)}{96(k-1)}\mu}
 \ll N^{-\frac{3\eta_{0}\log_{q}P^{-}(q)}{(96k+12)(k-1)}\rho_{1}}.
$$

\medskip

$\bullet$ \underline{Bounding $q^{3\gamma(k-1)-\eta_{0}/2 \rho}$.} 
\vspace{0.2cm}
\\We first bound the exponent. Since $\lambda > \nu$, we note that $\Tilde{\rho} \leq \rho$ (see \eqref{sigmavalue} and \eqref{rhovalue}). We use the definition of $\gamma$ and get
$$
3\gamma(k-1)-\eta_{0}\rho/2 \leq 3(k-1) \dfrac{\eta_{0}\rho}{12(k-1)}-\eta_{0}\rho/2
= -\dfrac{\eta_{0}\rho}{4}.
$$
We use~(\ref{rhoapp}) together with~\eqref{MinoQmu} to get
$$
q^{3\gamma(k-1)-\eta_{0} \rho/2} \leq q^{-\frac{\eta_{0}\rho}{4}}
\ll N^{-\frac{\eta_{0}}{4}(\frac{1}{8}+\frac{\eta_{0}}{192(k-1)})\frac{\rho_{1}}{k+1/8}}
\ll N^{-\frac{\eta_{0}\rho_{1}}{32(k+1/8)}}\\.
$$

\medskip

$\bullet$ \underline{Bounding $Tq^{\lambda-\nu-1+\rho}(\log^{+}N)^{2}\left(\dfrac{1}{N}+\dfrac{1}{q^{2\mu}}\right)$.} 
\vspace{0.2cm}
\\Since by \eqref{MinoN} we have $N \geq q^{2\mu}$, we just need to estimate the term
\begin{equation*}
Tq^{\lambda-\nu-1+\rho-2\mu}(\log^{+}N)^{2}.
\end{equation*}
Since $T=q^{\gamma}$ we get
\begin{equation*}
Tq^{\lambda-\nu-1+\rho-2\mu}(\log^{+}N)^{2}\leq q^{\lambda-\nu+\rho+\gamma-2\mu}(\log^{+}N)^{2}.\\
\end{equation*}
Using Corollary~(\ref{MajoGR}) we have, for sufficiently large $\mu$,
\begin{align*}
\lambda-\nu+\rho+\gamma-2\mu&=\Big\lfloor \dfrac{\gamma}{2}\Big\rfloor +\rho +\gamma-2\mu \\
&\leq \dfrac{\mu}{384}+\dfrac{\mu}{8}+\dfrac{\mu}{384}+\dfrac{\mu}{192}-2\mu \\
&\leq \dfrac{-179}{96}\mu.
\end{align*}
Thus,
\begin{equation*}
Tq^{\lambda-\nu-1+\rho-2\mu}(\log^{+}N)^{2} \leq \dfrac{(\log N)^{2}}{q^{(179/96)\mu}}.
\end{equation*}
By the inequality~(\ref{MinoQmu}) and since $\log(N) \leq N^{1/4}$, we get
\begin{align*}
Tq^{\lambda-\nu-1+\rho-2\mu}(\log^{+}N)^{2} &\ll (\log N)^{2}N^{-\frac{179\rho_{1}}{96k+12}}\\
&\ll N^{-\frac{179\rho_{1}}{192k+24}}.
\end{align*}
$\bullet$ \underline{Bounding $\dfrac{1}{q^{\sigma-\rho}}$.}

\medskip

We use $\sigma=\Big\lfloor\dfrac{\mu}{4}\Big\rfloor$ and Corollary~(\ref{MajoGR}). This yields
$$
\dfrac{1}{q^{\sigma-\rho}} \leq \dfrac{1}{q^{\lfloor \mu/4\rfloor-\mu/8-\mu/384}}
\ll q^{-(47/384)\mu}
\ll  N^{-\frac{47\rho_{1}}{384k+48}},
$$
which is the same bound as in \eqref{twofoldbound}.

\bigskip
\bigskip

\hspace*{1cm} Finally, we collect all these bounds and reconsider \eqref{Equafinale}. We have
\begin{equation*}
\Big|\dfrac{S_{0}(N,q^{\nu},\xi)}{q^{\nu+1}N}\Big|^{2^{k}} \ll N^{-\frac{\rho_{1}\eta_{0}}{(384k+48)(k-1)}}+N^{-\frac{47\rho_{1}}{384k+48}}+N^{-\frac{\rho_{1}\eta_{0} }{(96k+12)(k-1)}}+N^{-\frac{3\eta_{0}\log_{q}P^{-}(q)}{(96k+12)(k-1)}\rho_{1}}+N^{-\frac{\rho_{1}\eta_{0}}{32k+4}}+N^{-\frac{179\rho_{1}}{192k+24}}.
\end{equation*}
By comparing the summands, we get
\begin{equation*}
\Big|\dfrac{S_{0}(N,q^{\nu},\xi)}{q^{\nu+1}N}\Big|^{2^{k}} \ll N^{\frac{-\rho_{1}\eta_{0}}{(384k+48)(k-1)}}+N^{-\frac{3\eta_{0}\log_{q}P^{-}(q)}{(k-1)(96k+12)}\rho_{1}}.
\end{equation*}
\section{Farey sequences}\label{Appendice}
\label{appendix1}
The following  results are standard and can be found in~\cite{Hardy}.
\begin{de}[Farey sequence]
Let $n \geq 1$ be an integer. We call the Farey sequence of order $n$, denoted by $\mathcal{F}_{n}$, the sequence formed by irreducible fractions $0\leq \frac{a}{b}\leq 1$ with $b \leq n$. Two fractions are said to be neighbours in the sequence $\mathcal{F}_{n}$ if they are consecutive in $\mathcal{F}_{n}$.
\end{de}

\begin{rem}
By translation, we define the Farey sequence of order $n$ over $\mathbb{R}$, which we will continue to denote as $\mathcal{F}_{n}$ hereafter.
\end{rem}

\begin{lemme}[Midpoint lemma]
Let $\frac{a}{b}<\frac{c}{d}$ be two irreducible fractions. Then
\begin{equation*}
    \dfrac{a}{b}<\dfrac{a+c}{b+d}<\dfrac{c}{d}.
\end{equation*}
The fraction $\frac{a+c}{b+d}$ is called the midpoint of the fractions $\frac{a}{b}$ and $\frac{c}{d}$.
\end{lemme}
\begin{prop}[Farey's criterion]
\label{CritFarey}
Let $\frac{a}{b}<\frac{c}{d}$ be two neighbouring fractions in $\mathcal{F}_{n}$. Then,
\begin{equation*}
\Delta=bc-ad=1 \ \text{ and }\  b+d>n.
\end{equation*}
Additionally,
\begin{equation*}
    \dfrac{a+c}{b+d}-\dfrac{a}{b}<\dfrac{1}{bn}
\end{equation*}
and 
\begin{equation*}
    \dfrac{c}{d}-\dfrac{a+c}{b+d}<\dfrac{1}{dn}.
\end{equation*}

\end{prop}

\begin{de}
\label{DefFarey}
Let $\alpha \in \mathbb{R}$ and let $\frac{a}{b}, \frac{c}{d} \in \mathcal{F}_{n}$ be two neighbouring fractions such that
\begin{equation*}
    \dfrac{a}{b} \leq \alpha < \dfrac{c}{d}.
\end{equation*}

If $\alpha < \frac{a+c}{b+d}$, then we define
\begin{equation*}
    \dfrac{P_{n}(\alpha)}{Q_{n}(\alpha)}=\dfrac{a}{b}.
\end{equation*}
Otherwise, we define
\begin{equation*}
    \dfrac{P_{n}(\alpha)}{Q_{n}(\alpha)}=\dfrac{c}{d}.
\end{equation*}
\end{de}

\begin{prop}[Farey approximation]

\label{ApproxFarey}
For $\alpha \in \mathbb{R}$ and for $n \in \mathbb{N}$, we have
\begin{equation*}
    |Q_{n}(\alpha)\alpha-P_{n}(\alpha)|<\dfrac{1}{n}.
\end{equation*}
\end{prop}

\section{Graph of the function $\varepsilon \mapsto \log(\eta(\varepsilon))$}\label{appendixeta}
A straightforward Python program allows to plot the graph of the quantity defined in~(\ref{value-eta}), seen as a function of $\varepsilon$.
\begin{center}
    \includegraphics[scale=0.75]{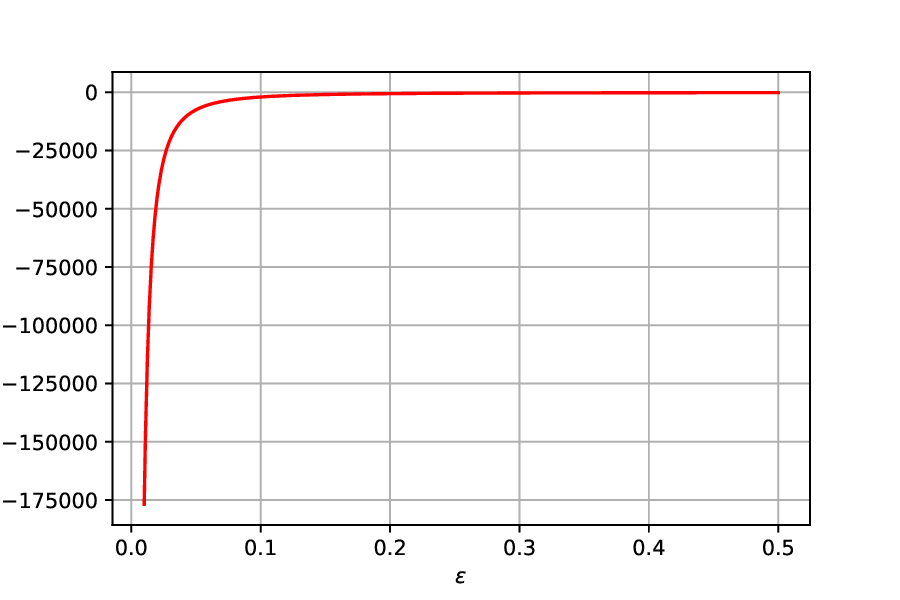}
\end{center}
\begin{center}
    \textsc{Figure 1.}
    The graph of the function $\varepsilon \mapsto \log\eta(\varepsilon)$ in $]0.01,0.5[$ for $q=b=2$.
\end{center}
\begin{center}
    \includegraphics[scale=0.75]{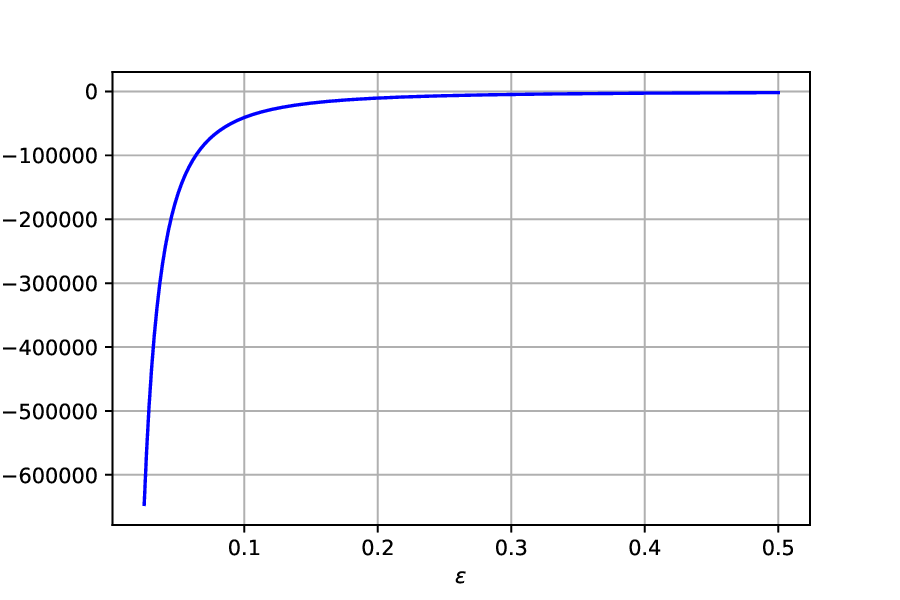}
\end{center}
\begin{center}
    \textsc{Figure 2.}
    The graph of the function $\varepsilon \mapsto \log\eta(\varepsilon)$ in $]0.025,0.5[$ for $q=10$ and $b=7$.
\end{center}
We remark that $\eta({\varepsilon})$ becomes extremely small as $\varepsilon$ approaches 0. Furthermore, we observe that for larger parameters $q$ and $b$, it becomes very difficult to calculate $\eta({\varepsilon})$ for small values of $\varepsilon$. Indeed, the domain taken for the simulation of the second graph is more restricted than that of the first because of the increased complexity of the computation.
The following table gives (for $q=b=2$)  some values for the exponents $\eta(\varepsilon)$ defined by \eqref{value-eta}, and  $\xi'_{2,\varepsilon}$ defined by \eqref{xi2eps}.
\begin{table}[h]
\centering
\begin{tabular}{|c|c|c|}
\hline
$\varepsilon$ & $\log\eta(\varepsilon)$ &$\log\xi'_{2,\varepsilon}$ \\
\hline
$0.3$ & $-270.77$ & $-5.85$\\
$0.2$& $-553.97$ & $-6.26$\\
$0.1$ & $-1993.60$ & $-6.95$\\
$0.05$& $-7504.14$ & $-7.65$\\
$0.01$ & $-176~866.99$ & $-9.25$\\
$0.005$ & $-700~973.54$ &$-9.95$\\
$0.001$ & $-17~375~734.08$ & $-11.56$\\
\hline
\end{tabular}
\begin{center}
    \textsc{Figure 3.} Few values of $\log \eta(\xi)$ and $\log \xi'_{2,\varepsilon}$ (case $b=q=2$).
\end{center}
\label{tab:valeurs}
\end{table}
We notice that the exponent provided in~\cite{Mart2014} is much larger than ours. One of  reasons is the fact that we handle the $\max$ over the residue classes, while~\cite{Mart2014} deals with individual bounds.
\end{appendices}

\section*{Acknowledgements}
The autor would like to greatly thank my thesis supervisors Cécile Dartyge and Thomas Stoll for their teaching and advice, which have greatly helped me throughout my work. Moreover, he is very grateful to the referee for the careful reading of the paper and for the comments and suggestions.
The author is supported by the French projects ANR-18-CE40-0018 (EST) and ANR-20-CE91-0006 (ArithRand).

\end{document}